\documentclass{amsart}
\usepackage{xcolor}
\usepackage{hyperref}
\usepackage{mathtools}
\usepackage{comment}
\usepackage{ amssymb }

\newtheorem{theorem}{Theorem}[section]

\newtheorem{fact}[theorem]{Fact}
\newtheorem{lemma}[theorem]{Lemma}
\newtheorem{corollary}[theorem]{Corollary}
\newtheorem{claim}[theorem]{Claim}
\newtheorem{proposition}[theorem]{Proposition}

\newtheorem*{cor}{Corollary}

\newtheorem{introtheorem}{Theorem}

\theoremstyle{definition}
\newtheorem{definition}[theorem]{Definition}
\newtheorem{example}[theorem]{Example}
\newtheorem{convention}[theorem]{Convention}
\newtheorem{notation}[theorem]{Notation}

\newtheorem{assumption}[theorem]{Assumption}

\theoremstyle{remark}
\newtheorem{remark}[theorem]{Remark}
\newtheorem{warning}[theorem]{Warning}

\newcommand{\rk}{\operatorname{rk}}
\newcommand{\acl}{\operatorname{acl}}

\newcommand{\dprk}{\operatorname{dprk}}
\newcommand{\mr}{\operatorname{MR}}
\newcommand{\tp}{\operatorname{tp}}

\def\sub{\subseteq}
\def\0{\emptyset}
\def\CM{\mathcal M}
\def\CK{\mathcal K}

\def\Ind#1#2{#1\setbox0=\hbox{$#1x$}\kern\wd0\hbox to 0pt{\hss$#1\mid$\hss}
\lower.9\ht0\hbox to 0pt{\hss$#1\smile$\hss}\kern\wd0}

\def\Notind#1#2{#1\setbox0=\hbox{$#1x$}\kern\wd0\hbox to 0pt{\mathchardef
\nn=12854\hss$#1\nn$\kern1.4\wd0\hss}\hbox to
0pt{\hss$#1\mid$\hss}\lower.9\ht0 \hbox to
0pt{\hss$#1\smile$\hss}\kern\wd0}

\newenvironment{claimproof}[1][\proofname]
  {%
    \proof[#1]%
  }
  {%
    \endproof%
  }

\title{Zilber's Trichotomy in Hausdorff Geometric Structures}
\author{Benjamin Castle} 
\address{Department of Mathematics, University of Maryland,
	College Park, MD 20742,
	USA}
\email{bcastle@math.berkeley.edu}

\author{Assaf Hasson} 

\address{Department of Mathematics, Ben Gurion University of the Negev, Be'er-Sheva 84105, Israel}
\email{hassonas@math.bgu.ac.il}

\author{Jinhe Ye}
\address{Mathematical Institute, University of Oxford,
	Oxford OX2 6GG, United Kingdom}
\email{jinhe.ye@maths.ox.ac.uk}

\begin{document}
	
	\thanks{The first author was partially supported by the Fields Institute for Research in Mathematical Sciences; NSF grant DMS 1800692; ISF grant No. 555/21; a BGU Kreitman foundation fellowship; and a UMD Brin postdoc. The second author was supported by ISF grants No. 555/21. The third author was partially supported by the Fondation Sciences Math\'ematiques de Paris}
	
	\begin{abstract}
		
		We give a new axiomatic treatment of the Zilber trichotomy, and use it to complete the proof of the trichotomy for relics of algebraically closed fields, i.e., reducts of the ACF-induced structure on ACF-definable sets. More precisely, we introduce a class of geometric structures equipped with a Hausdorff topology, called \textit{Hausdorff geometric structures}. Natural examples include the complex field; algebraically closed valued fields; o-minimal expansions of real closed fields; and characteristic zero Henselian fields (in particular $p$-adically closed fields). We then study the Zilber trichotomy for relics of Hausdorff geometric structures, showing that under additional assumptions, every non-locally modular strongly minimal relic on a real sort interprets a one-dimensional group. Combined with recent results, this allows us to prove the trichotomy for strongly minimal relics on the real sorts of algebraically closed valued fields.
		Finally, we make progress on the imaginary sorts, reducing the trichotomy for \textit{all} ACVF relics (in all sorts) to a conjectural technical condition that we prove in characteristic $(0,0)$.
		
	\end{abstract}
	\maketitle
	\tableofcontents{\setcounter{tocdepth}{1}}
	
	\section{Introduction}
	
	
	
	
	The main goal of this paper was to complete the proof of Zilber's Restricted Trichotomy Conjecture for algebraically closed fields. In characteristic $0$ this conjecture was proven by the first author in \cite{CasACF0}. In order to adapt key analytic ideas from \cite{CasACF0} to positive characteristic, we opted, rather than working with formal schemes (as in \cite{HaSu}), to work in algebraically closed valued fields (ACVF). In this setting  -- in complete models -- the theory of analytic functions provides suitable analogues of the characteristic 0 statements we need.  Our main result is, thus: 
	
	\begin{introtheorem}\label{T: main}
		Let $\mathcal K$ be an algebraically closed valued field (ACVF). Let $\CM$ be a definable strongly minimal $\mathcal K$-relic. If $\CK$ is not locally modular then $\CK$ interprets a field $\mathcal K$-definably isomorphic to $K$.  
	\end{introtheorem}
	
	Let us briefly explain the terminologies used above. Recall that a $\CK$-\textit{relic} (see \cite{CasHas}) is a structure $\mathcal M=(M,...)$ such that the universe $M$, and all $\mathcal M$-definable subsets of powers of $M$, are interpretable sets in $\mathcal K$. If $M\subset K^n$ for some $n$, we call $\mathcal M$ a \textit{definable $\CK$-relic}. 
	
	A structure $\mathcal M=(M,...)$ is \textit{strongly minimal} if every definable subset of $M$ is either finite or co-finite (uniformly in definable families). \textit{Non-local modularity} is a necessary non-triviality condition (explained below) for $\CM$ to interpret a field. 
	
	By elimination of imaginaries, every ACF-relic is isomorphic to a definable ACF-relic. Moreover, clearly, every definable ACF-relic is also a definable ACVF-relic. In particular, as a special case, we obtain Zilber's restricted trichotomy for ACF: 
	\begin{cor}\label{C: ACF case}
		Let $\mathcal K$ be an algebraically closed field. Any strongly minimal $\mathcal K$-relic is either locally modular or interprets a field $\mathcal K$-definably isomorphic to $K$. In particular, an arbitrary $\mathcal K$-relic is either 1-based or interprets such a field. 
	\end{cor}
	
	Following Hrushovski's refutation \cite{Hr1} of Zilber's full trichotomy conjecture (suggesting that every non-locally modular strongly minimal structure interprets an algebraically closed field), and the seminal work of Hrushovski and Zilber \cite{HrZil} (proving the conjecture in the abstract setting of Zariski Geometries), the conjecture's status became somewhat unclear. On the one hand, no alternative formulation emerged that could withstand Hrushovski's technique for constructing counter-examples. On the other hand, Hrushovski's applications of special instances of the conjecture (ultimately, based on variants of \cite{HrZil}) in the solution of Diophantine problems (see \cite{BouML} for a survey of those) singled out Zilber's trichotomy as a powerful principle, with a tendency to hold in geometric settings. 
	
	Lacking a unifying conjecture, research focused on restricting attention to strongly minimal relics of various theories of a geometric nature. In the late 1980s, Zilber and Rabinovich studied ACF-relics (see \cite{Ra} and references therein), subsuming works of Martin \cite{Martin} and Marker-Pillay \cite{MaPi}. In 2006 Peterzil suggested an analogous conjecture for o-minimal relics, and in \cite{KowRand} Kowalski and Randriambololona conjectured an analogue for ACVF. In \cite{ZilJac} Zilber  applied Rabinovich's theorem to give a short model theoretic proof of a conjecture in anabelian geometry \cite{BoKoTs}. In his paper Zilber called for a new, modern, proof of Rabinovich's theorem (that was never published), preferably one proving the full Restricted Trichotomy for ACF-relics. 
	
	Zilber's challenge, combined with the above conjectures, rekindled interest in the Trichotomy and ultimately led to the present work. Indeed, our results contribute to all of the conjectures mentioned above. First, and most direct, is the positive solution to the conjecture restricted to definable ACVF-relics, and thus to ACF-relics. In residue characteristic 0, the result also covers \emph{general} (i.e. interpretable) ACVF-relics. 
	
	Secondly, these results have applications outside model theory. Recently the first and second author \cite{CasHasAV} used the Trichotomy for ACF-relics to prove a reconstruction theorem for abelian varieties from a sub-variety, expanding and generalizing \cite{ZilJac}. Similar techniques give a partial solution to a conjecture of Booher and Voloch \cite[Conjecture 2.6]{BooVol} extending Zilber's main result of \cite{ZilJac} to generalized Jacobians. We expect that, augmented by known algebro-geometric techniques, this partial result should suffice for providing a complete proof of that conjecture. 
	
	As will be explained below in more detail, Zilber's intention in calling for a modern proof of Rabinovich's Theorem was to apply the Zariksi Geometries technology. However, Zariski Geometries -- despite being a powerful (and essentially the only) tool in certain settings -- turned out to be hard to apply to relics. In practice, recovering a Zariski Geometry in a relic raises serious combinatorial problems that seem insurmountable, particularly for higher-dimensional relics. 
	
	In view of the above, the last (and possibly most far reaching) contribution of the current paper consists of setting up the new axiomatic framework of Hausdorff Geometric Structures. The aim of this framework is to streamline the proof of Trichotomy-type results -- in particular for relics -- giving more flexibility in proving crucial ingredients in such works. This new axiomatic framework has already proved applicable beyond the immediate context for which it was developed. In a recent preprint \cite{CasOmin} the first author builds on the results of the present work to prove the higher dimensional case of Peterzil's conjecture mentioned above. Combined with the results of the present work, this also provides a complete proof of the Trichotomy for relics of Compact Complex Manifolds. Then in \cite{CasHasTconv}, the first and second author use similar techniques to prove analogous results for real closed valued fields (and various expansions thereof). 
			
	
	Our main abstract result in the present paper (Theorem \ref{T: composite main thm}) states: 
	\begin{introtheorem}\label{T: abstract}
		Let $(\mathcal K,\tau)$ be a \emph{Hausdorff geometric structure}. Assume that $(\mathcal K,\tau)$ has \emph{ramification purity}, has \emph{definable slopes satisfying TIMI}, and is either \emph{differentiable} or has the \emph{open mapping property}. Let $\mathcal M=(M,...)$ be a non-locally modular strongly minimal definable $\mathcal K$-relic. Then $\dim(M)=1$, and $\mathcal M$ interprets a strongly minimal group.
	\end{introtheorem} 
	In the next section, we provide a more detailed informal overview of the statement of this result and the general strategy of proof. 
	
	Theorem \ref{T: abstract} relates to a recent work of Onshuus, Pinzon and the second author, \cite{HaOnPi}, where the problem of interpreting a field in definable ACVF-relics reduces to showing that the relic $\CM$ interprets a group locally isomorphic to a local subgroup of either $(K,+)$ or $(K^*, \cdot)$. As this can only happen if $M$ is (as a $\CK$-definable set) 1-dimensional, the remaining problem splits in two: (1) showing that non-local modularity of $\CM$ implies that $\dim(M)=1$, and then (2) showing that strongly minimal, non-locally modular, 1-dimensional relics interpret a group with the desired properties. We expect that (at least in characteristic 0) the axiomatic framework of Theorem \ref{T: abstract} may suffice (with minor adjustments) to produce a field \textit{directly}, without going through a group first. However, as this is not needed for our main results, and in order to not further lengthen the paper, we have not extended our axiomatization to cover this part of the argument.  
	
	The assumptions of Theorem \ref{T: abstract} are, to a large extent, formulated to capture relics of analytic expansions of algebraically closed fields, and we do not expect them to apply in full in significantly different settings. However, as we will see below, our proof is built from the bottom up, strengthening the results as we specialize the axiomatization. This allows us to capture important steps of the strategy of proof in settings where only some of the axioms hold.  

    \subsection{Strategy of proof}
The proofs of our main theorems are split into three uneven parts. Sections \ref{S: HGS} through \ref{S: summing up} introduce the axiomatic framework and culminate with the proof of Theorem \ref{T: abstract}. In Sections \ref{S: examples} and \ref{S: slopes}, we verify that ACVF satisfies the assumptions of Theorem \ref{T: abstract}, and thus (in combination with the main result of \cite{HaOnPi}) deduce Theorem \ref{T: main}. Section \ref{S: imaginaries} is independent of the rest of the paper. It uses techniques introduced in \cite{HaHaPeVF} to show that \emph{interpretable} non-locally modular strongly minimal ACVF$_{0,0}$-relics definably embed into a power of the valued field sort or the residue field sort (and thus are covered by Theorem \ref{T: main} and Corollary \ref{C: ACF case}).

Let us start by describing our axiomatic setting, and briefly explain the role of each set of axioms in the proof. The most general framework we work in, Hausdorff geometric structures, is introduced in Section \ref{S: HGS}. These are geometric structures (in the sense of \cite[\S 2]{HruPil}) equipped with a Hausdorff topology. The axiomatization is aimed to capture the tameness of the topology (in the sense of e.g. \cite{vdDries}, \cite{SimWal} or \cite{DolGooVisceral}), without assuming its definability. Roughly, it assures a good interaction between the dimension theory of definable sets and the topology. 

The main technical novelty in this section is the notion of \textit{enough open maps} -- a condition inspired by the most critical use of the complex analytic topology in \cite{CasACF0}, as well as one of the critical uses of o-minimal geometry in the o-minimal variant \cite{ElHaPe} of the restricted trichotomy. It aims to state, without assuming any differential structure, that the intersection of definable curves should be transverse, unless an obvious obstruction prevents it. It is designed to be a common generalization of this phenomenon in both the o-minimal and analytic settings, and as such it is rather technical; but the main point is that certain finite-to-one definable maps are stipulated to be open. 

As the notion of enough open maps is rather technical and unpleasant to verify directly in applications, we give two simpler notions, each of which implies enough open maps. These are inspired by the examples above -- namely, one (the \textit{open mapping property}) is an abstraction of the complex algebraic setting, and the other (\textit{differentiability}) is an abstraction of the o-minimal setting. 

Section \ref{s: relics} begins the study of strongly minimal relics of Hausdorff geometric structures (with enough open maps). At this level of generality (covering such contexts as o-minimal strongly minimal relics, 1-h-minimal strongly minimal relics and more) we can define and prove a key topological property, that we call \emph{weak detection of closures}. Detection of closures, as a key tool for the study of strongly minimal relics, goes back to \cite{PeStExpansions}, and plays a crucial role in several subsequent works on various restrictions of Zilber's trichotomy (\cite{HaKo}, \cite{ElHaPe}, and most importantly in \cite{CasACF0}). In Section \ref{s: relics} we define this notion (Definition \ref{D: weak detection of closures}), and in the subsequent section we prove one of the main results of this paper: 

\begin{introtheorem}\label{T: intro 4} Let $(\mathcal K,\tau)$ be a Hausdorff geometric structure with enough open maps. Let $\mathcal M$ be a non-locally modular strongly minimal definable $\mathcal K$-relic. Then, potentially after naming a small set of parameters, $\mathcal M$ weakly detects closures.
\end{introtheorem}

As was noted in \cite{CasACF0}, at this level of generality, it seems somewhat over-optimistic to hope for the relic to detect \textit{all} closure points of definable sets (in the sense, e.g., that if $X$ is $A$-definable in the relic and $x$ is in the frontier of $X$, then the Morley rank of $x$ over $A$ is less than that of $X$). The key idea, adapted from \cite{CasACF0} to our axiomatic setting, is to identify general enough situations when our over-optimistic hopes are in fact fulfilled. 
 The proof of this result follows rather closely the analogous result from \cite{CasACF0}, but has a few differences that we point out as we go.

In Section \ref{S: 1-dim} we give applications of the above theorem to the Zilber trichotomy. We first show that from weak detection of closures one can detect certain double intersections of plane curves (see Definition \ref{D: detects noninjectivities} for details):

\begin{introtheorem}\label{T: intro 5} Let $(\mathcal K,\tau)$ be a Hausdorff geometric structure. Let $\mathcal M$ be a non-locally modular strongly minimal definable $\mathcal K$-relic. If $\mathcal M$ weakly detects closures, then $\mathcal M$ detects multiple intersections.
\end{introtheorem}

We then conclude that, provided $(\CK,\tau)$ obeys a suitable form of the \textit{purity of the ramification locus}, non-locally modular strongly minimal definable relics must have one-dimensional universes (see Definition \ref{D: ramification purity} for the notion of ramification purity):

\begin{introtheorem}\label{T: intro 6} Let $(\mathcal K,\tau)$ be a Hausdorff geometric structure. Let $\mathcal M=(M,...)$ be a non-locally modular strongly minimal definable $\mathcal K$-relic. Assume that $\mathcal M$ detects multiple intersections, and $(\mathcal K,\tau)$ has ramification purity. Then $\dim(M)=1$ in the sense of $\mathcal K$.
\end{introtheorem}

We expect ramification purity to hold only in (expansions of) algebraically closed fields. In (o-minimal) expansions of real closed fields the first author proves a variant, sufficient to obtain an analogue of the above theorem, but we do not expect the same approach to apply in other settings. Of course, we do not expect non-locally modular strongly minimal relics to be definable in tame expansions of fields that are not real closed or algebraically closed -- so one could hope for a uniform, dimension-independent, proof of such a result. 

Section \ref{S: ACVF} axiomatizes the interpretation of a group when $\dim(M)=1$. We introduce a setting of \textit{definable slopes} -- requiring that curves in $K^2$ can be approximated near generic points by abstract `Taylor polynomials.' We then introduce a technical notion called \textit{TIMI} (tangent intersections are multiple intersections) -- stipulating that if two curves share the same Taylor polynomial to unusually high order at a point, the point forms a `topological multiple intersection' -- which allows us to prove Theorem \ref{T: abstract}. This axiom has a distinct analytic flavour. It would be interesting to find a weaker axiom implying TIMI for curves definable in strongly minimal relics (this was one of the main technical results of \cite{ElHaPe}). 

In Section \ref{S: examples} we give examples of Hausdorff geometric structures with enough open maps. In particular, we note that certain visceral theories (in the sense of \cite{DolGooVisceral}) share these properties (covering 1-h-minimal valued fields and o-minimal expansions of fields) and that a certain class of \'ez topological fields (essentially coming from the notion of \'ez fields in \cite{ez}) also satisfies these axioms. This is our main example, as it is the only one covering ACVF in all characteristics. 

Restricting further to algebraically closed \'ez fields, we next verify the axiom on the purity of the ramification locus. Then, for the rest of the paper, we work concretely in ACVF. First, in section \ref{S: slopes}, we show that ACVF has \emph{definable slopes} as well as the remaining axiom TIMI. The main novelty in this section is the introduction of \emph{Taylor groupoids}, allowing for greater flexibility in the study of slopes in definable families of curves, compared with earlier treatments of similar problems. Also, the definition of slope in positive characteristic is non-standard, absorbing powers of the Frobenius automorphism into the definition, allowing for a more uniform and less technical treatment of slopes.

The final section of the paper uses techniques from \cite{HaHaPeVF} to extend the results of the main theorem to \emph{interpretable} relics. This section is independent of the rest of the paper.  

It seems that, at least in characteristic $0$, large parts of the present argument could be extended to analytic expansions of ACVF (in the sense of Cluckers and Lipshitz \cite{CluLip}). Extending the results to arbitrary $V$-minimal theories may prove more challenging. More generally, in a recent preprint, \cite{JohnCMin}, Johnson shows that $C$-minimal expansions of ACVF are geometric, implying that they are also Hausdorff geometric structures. As Johnson proves, moreover, that definable functions are generically strictly differentiable, it seems that the proof of enough open maps for 1-h-minimal fields could be extended to the $C$-minimal setting,  but there are details to verify.

\section{Preliminaries}

We give a quick overview of some of the model-theoretic notions used extensively in the paper. Since modern detailed reviews of the relevant notions can be found in recent papers, we will be brief, directing interested readers to relevant references. 

Throughout the paper, structures are denoted by calligraphic letters $\CK, \CM$ etc. and their universes by the corresponding Roman letters $K, M$ etc. Roman letters $A,B$ will usually denote sets of parameters (i.e., subsets of the universe of the structure). Whenever the structure is assumed to be $\lambda$-saturated for some cardinal $\lambda$, all parameter sets are tacitly assumed to be of cardinality smaller than $\lambda$. Lowercase Latin letters $a,b,c$ usually denote finite tuples of elements in our structure. When no confusion can occur, we write $AB$ and $Aa$ etc. as a shorthand for $A\cup B$ and  $A\cup \mathrm{dom}(a)$. Definable sets are usually denoted by $X,Y,Z$ etc. Unless specifically stated otherwise, the word \emph{definable} allows parameters but does not allow imaginary sorts. If we wish to allow imaginary sorts, we either refer to \emph{interpretable} sets, or we write explicitly \emph{definable in $T^{eq}$},  \emph{definable in $\CK^{eq}$} as the case may be. 

We use standard model-theoretic notation and terminology. Readers are referred to any textbook in model theory such as \cite{Maalouf} or \cite{TZ} for further details. 

Throughout the paper, we often work with \textit{finite correspondences} between interpretable sets. There are competing versions of what this might mean -- so before moving on, we clarify:

\begin{definition}\label{D: finite correspondence}
    Let $X$ and $Y$ be interpretable sets in some structure. A \textit{finite correspondence} between $X$ and $Y$ is an interpretable set $Z\subset X\times Y$ whose projections to $X$ and $Y$ are finite-to-one and surjective. If there is a finite correspondence between $X$ and $Y$, we say that $X$ and $Y$ \textit{are in finite correspondence}. If $X$ is in finite correspondence with a subset of $Y$, we say that $X$ is \textit{almost embeddable} into $Y$.
\end{definition}

\subsection{Geometric Structures}
Throughout this paper we will be working with geometric structures: 
\begin{definition}\label{D: geom-struc}
  We use $\acl$ to denote the model-theoretic algebraic closure. Let $T$ be a complete first-order theory. We say $T$ is \emph{geometric} if: 
  \begin{itemize}
\item $\acl$ satisfies exchange. More explicitly, for any $\mathcal{K}\models T$, $A\subseteq K$, and $b,c\in K$, $c\in \acl(Ab)\setminus \acl(A)$ implies that $b\in \acl(Ac)$.
\item $T$ eliminates $\exists^\infty$.
  \end{itemize}
  A structure $\mathcal{K}$ is \emph{geometric} if its theory is geometric. 
\end{definition}

In the definition, by \emph{elimination of $\exists^\infty$} (also referred to as \emph{uniform finiteness} in the literature) we mean that for any model $\CK$ and any formula $\phi(x,\overline y)$ (where $|x|=1$ and $|\overline y|=n$, say), the set $\{b\in K^n: |\phi(K,b)|<\infty \} $ is definable.

Working in a $|T|^+$-saturated geometric $\mathcal{K}$, there is a well-established theory of dimension and independence on $\mathcal{K}$, defined as follows. Let $A\subset K$ be small, let $a\in K^m$, and let $X$ be $A$-definable.

\begin{itemize}
    \item We $\dim(a/A)$ to be the length of a maximal $\acl_A$-independent subtuple of $a$. By the exchange property, this is independent of the choice of the maximal independent subtuple.
    \item We define $\dim(X) = \max\{\dim(x/A) : x \in X\}$. By a compactness argument, $\dim(X)$ does not depend on the parameter set $A$.
    \item We say that $a$ is \textit{independent from $b$ over $A$} if $\dim(a/Ab)=\dim(a/A)$.
\end{itemize}

In this language, the exchange property can be equivalently written as the \textit{additivity formula}:

$$\dim(ab/A)=\dim(a/A)+\dim(b/Aa)$$ for any $a\in K^m$, $b\in K^n$, and $A\subset\mathcal K$. This is, arguably, the single most useful property of dimension in geometric structures. Among other things, additivity gives the following:

\begin{itemize}
    \item Independence is \textit{symmetric}: $a$ is independent from $b$ over $A$, if and only if $b$ is independent from $a$ over $A$, if and only if $\dim(ab/A)=\dim(a/A)+\dim(b/A)$ (thus, we freely use the ambiguous language `$a$ and $b$ are independent over $A$').
    \item Independence is \textit{transitive}: $a$ is independent from $bc$ over $A$ if and only if $a$ is independent both from $b$ over $A$ and from $c$ over $Ab$.
    \item Independence extends naturally to several tuples: $a_1,...,a_n$ are independent over $A$ if $$\dim(a_1...a_n/A)=\sum_i\dim(a_i/A),$$ if and only if each $a_i$ is independent from $\{a_j:j\neq i\}$ over $A$.
\end{itemize}

Using compactness, one can also characterize dimension as follows: $\dim(X)\ge n$ if and only if $\dim(\pi(X))=n$ for some coordinate projection $\pi: X\to K^n$, and $\dim(X)$ is the maximal such $n$; or, $\dim(X)\leq n$ if and only if there is a definable finite-to-one map $f:X\rightarrow K^n$, and $\dim(X)$ is the minimal such $n$. In fact, the finite-to-one map $f:X\rightarrow K^{\dim X}$ can be defined uniformly in families (by coordinate projections); combined with uniform finiteness, one concludes that \textit{dimension is definable in families}: given a formula $\phi(x,y)$, the set $\{y: \dim(\phi(x,y)=d\}$ is definable for any integer $d$. We will use these facts without further mention.

Many other basic properties follow easily from the definition and saturation. For example:

\begin{itemize}
    \item $\dim(X\times Y)=\dim(X)+\dim(Y)$.
    \item $\dim(X\cup Y)=\max\{\dim(X),\dim(Y)\}$.
    \item If $X$ and $Y$ are in finite correspondence, then $\dim(X)=\dim(Y)$.
\end{itemize}

We will use all of these facts without further mention for the rest of the paper.

For an $A$-definable set $X$ and $B\supseteq A$, we say that $x
\in X$ is a \emph{generic point of $X$ over $B$} if $\dim(x/B)=\dim X$. A definable set $Y\sub X$ is \emph{generic} in $X$ if $\dim(X)=\dim(Y)$, if and only if $Y$ contains some $y$ generic in $X$ (over some parameter set over which $Y$ is defined). Two definable sets $X$ and $Y$ are \emph{almost equal} if the symmetric difference $X\Delta Y$ is not generic in either of them (equivalently, by the above, if $X$ and $Y$ have the same generic points over some parameter set defining both of them). It is easy to see that almost equality is an equivalence relation on definable sets, and is definable in families (because dimension is).

Finally, we briefly note that dimension and independence can be extended to \textit{interpretable sets} in a way that preserves additivity (but may lose the equivalence with acl-independence for interpretable sets that are not definable) -- see 
section 3 of \cite{Gagelman}. Thus, the notation $\dim(a/A)$ is well-defined even for $a\in\mathcal K^{eq}$, and one still has all consequences of additivity. We will use this, but very rarely (in fact, only in Section 7, when computing dimensions of slopes and their coherent representatives).

For more details on the basic properties of dimension in geometric structures, we direct the reader to e.g. \cite[\S 2]{AcHa}.

\subsection{Strongly Minimal Structures}\label{ss: sm}

In the paper, we study strongly minimal relics of certain geometric structures. Recall that a structure is \textit{strongly minimal} if, in all of its elementary extensions, every definable set in one variable is finite or cofinite. Strongly minimal structures are themselves geometric structures. We now recall some facts and terminology specific to the strongly minimal case. First, the dimension associated with strongly minimal structures is known as Morley Rank, and in this paper will be called \textit{rank} and denoted $\rk$.

The main advantage of strongly minimal structures (viewed within the larger class of geometric structures) is the notion of \textit{stationarity}. In the context of geometric structures, a definable set is \textit{stationary} if it cannot be partitioned into two generic definable subsets. A key feature of stationary sets is that their almost equality classes can be coded: that is, if $X$ is stationary, there is $c\in\mathcal M^{eq}$ such that for all $\sigma\in Aut(\mathcal M)$, $\sigma(c)=c$ if and only if $\sigma(X)$ is almost equal to $X$. This tuple $c$ is unique up to interdefinability, and is called the \textit{canonical base of $X$} and denoted $\operatorname{Cb}(X)$. (See \cite[\S 8.2]{MaBook} for more on canonical bases and \cite[\S 2.3]{CasACF0} for some relevant applications). 

In many geometric structures, only singletons are stationary. Notably, strongly minimal structures are exactly those geometric structures $\mathcal M=(M,...)$ such that the universe $M$ is stationary. In fact, it follows that $M^n$ is also stationary for each $n$, and that every definable subset $X\subset M^n$ is a disjoint union of finitely many stationary definable sets $Y_i$ such that $\rk(Y_i)=\rk(X)$.  These stationary sets are unique up to almost equality, and are called the \textit{stationary components} of $X$.

In the context of a geometric structure, a \emph{curve} is a definable set of dimension $1$. For a strongly minimal structure $\mathcal M=(M,...)$, we frequently use \emph{plane curve} to describe curves in $M^2$. Note that this is not the same as an algebraic curve (even if $M$ is a field). To avoid confusion, we will clarify our interpretation of `curve' when it is not clear from the context. 

Following \cite{CasACF0}, we call a plane curve $C\subset M^2$ \textit{non-trivial} if both projections $C\rightarrow M$ are finite-to-one (equivalently, no stationary component of $C$ is almost equal to a horizontal or vertical line). By uniform finiteness, non-triviality is definable in families. We will largely be able to assume that all plane curves considered are non-trivial -- but we will clarify this as we go.

We will frequently use intersections of relic-definable families of curves (as well as higher dimensional objects) to study an ambient topology from within the relic. Since, a priori, these curves are not geometric objects, we have to identify some combinatorial properties of families of curves allowing us to study them geometrically. This is discussed extensively in \cite[\S 2]{CasACF0}, so we will be brief. 

Suppose $\mathcal M$ is strongly minimal. A parametrized family $\mathcal X=\{X_t:t\in T\}$ of subsets of $M^n$ is a \textit{definable family} if the parameter set $T$ and the graph $X=\{(x,t):x\in X_t, t\in T\}\subset M^n\times T$ are definable. If  $\mathcal X=\{X_t:t\in T\}$ is a definable family, we call $\mathcal X$ \textit{almost faithful} if 
\begin{enumerate}
    \item each fiber $X_t$ has the same dimension $d$,
    \item for each $t$, there are only finitely many $t'$ with $\dim(X_t\cap X_{t'})=d$.
\end{enumerate}
    If $\mathcal X$ is almost faithful, the \textit{rank} of $\mathcal X$ is the rank of the parameter set, $\rk(T)$. It is not hard to check that every stationary definable set $S$ is, up to almost equality, a generic member of an almost faithful definable family of rank $\rk(\operatorname{Cb}(S))$. 

For the purposes of the present work, it is convenient to use the following definition: 
\begin{definition}
    A strongly minimal structure is \emph{not locally modular} if it admits a definable rank 2 almost faithful family of plane curves. 
\end{definition}
This is well known to be equivalent to the standard definition (see, e.g., \cite[Theorem 8.2.11]{MaBook}). Clearly, algebraically closed fields are not locally modular (as witnessed, e.g., by the family of affine lines). Moreover, it is not hard to see that local modularity is preserved under interpretations between strongly minimal structures. Thus, non-local modularity is a necessary condition for a strongly minimal structure to interpret such a field. \\

Finally, we recall some special types of families of plane curves that were used extensively in \cite{CasACF0}. These are thought of as families that have been presented most conveniently and had various types of irregularities removed.

\begin{definition}\label{D: excellent} Let $\mathcal C=\{C_t:t\in T\}$ be a definable almost faithful family of plane curves in a strongly minimal structure $\mathcal M=(M,...)$, with graph $C\subset M^2\times T$.
\begin{enumerate}
    \item $\mathcal C$ is a \textit{standard family} if $T$ is a generic subset $M^n$ for some $n\geq 1$, and for each $p\in M^2$, the set $\{t\in T:p\in C_t\}$ is non-generic in $T$.
    \item $\mathcal C$ is \textit{excellent} if $T$ is a generic subset of $M^2$, each $C_t$ is non-trivial, and for each $p\in M^2$, the set $\{t\in T:p\in C_t\}$ is either empty or a non-trivial plane curve. (In particular, note that excellent families are also standard).
\end{enumerate}
\end{definition}

\section{Hausdorff Geometric Structures and Enough Open Maps}\label{S: HGS}

Throughout the next six sections, we work with a structure $\mathcal K=(K,...)$ in a language $\mathcal L$, endowed with a topology $\tau$ on $K$. We extend $\tau$ to a topology on each $K^n$ by taking the product topology, and subsequently to a topology on every definable subset of $K^n$ by taking the subspace topology. Our goal is to axiomatize in terms of $\mathcal K$ and $\tau$ a `sufficiently geometric' setting for various parts of the argument from \cite{CasACF0} to be adapted.

\subsection{The Definition}\label{ss: Hausdorf Def}

The following will be the basic framework we assume throughout:

\begin{definition}\label{D: easy axioms} Let $(\mathcal K,\tau,\mathcal L)$ be as above. We say that $(\mathcal K,\tau)$ is a \textit{Hausdorff geometric structure} if the following hold:
\begin{enumerate}
    \item $\tau$ is Hausdorff. 
    \item $\mathcal K$ is geometric and $\aleph_1$-saturated (but $\mathcal L$ might be uncountable).
    \item (Strong Frontier Inequality) If $X\subset K^n$ is definable over $A$ and 
    $a\in\overline{\operatorname{Fr}(X)}$ then $\dim(a/A)<\dim(X)$.
    \item (Baire Category Axiom) Let $X\subset K^n$ be definable over a countable set $A$, and let $a\in X$ be generic over $A$. Let $B\supset A$ be countable. Then every neighborhood of $a$ contains a generic of $X$ over $B$.
    \item (Generic Local Homeomorphism Property) Suppose $X$, $Y$, and $Z\subset X\times Y$ are definable over $A$ of the same dimension, and each of $Z\rightarrow X$ and $Z\rightarrow Y$ is finite-to-one. Let $(x,y)$ be generic in $Z$ over $A$. Then there are open neighborhoods $U$ of $x$ in $X$, and $V$ of $y$ in $Y$, such that the restriction of $Z$ to $U\times V$ is the graph of a homeomorphism $U\rightarrow V$.    
\end{enumerate}
\end{definition}

\begin{warning} We caution that Definition \ref{D: easy axioms} does not assume the existence of a definable basis for $\tau$ -- or even a basis whole members are definable. Thus $\tau$ is attached to the specific structure $\mathcal K$, and need not induce a similar topology on any elementary extension of $\mathcal K$. 

Indeed, we will work throughout in a fixed Hausdorff geometric structure $(\mathcal K,\tau)$, without considering any other models of its theory. This is why we assume $\mathcal K$ is $\aleph_1$-saturated in Definition \ref{D: easy axioms} -- so that we still have some of the tools of saturated models at our disposal (e.g. generic points will always exist over countable sets). In particular, with almost no exceptions, we will not use monster models of $\operatorname{Th}(\mathcal K)$ (and if we do use one, we will make clear what we are doing when it happens).

Our reason for making this choice (i.e. to not require a definable basis) is given by two of the motivating examples. First, in the complex numbers with the analytic topology (considered in \cite{CasACF0}), the only definable open sets are the Zariski open sets, which do not form a basis for the analytic topology. Second, in many \'ez fields (see below for a discussion of these fields), one has only an ind-definable basis (i.e. there is a basis that is a countable union of definable families of open sets). Thus, Definition \ref{D: easy axioms} is designed to accommodate these settings. Note, however, that in all other examples we know, the topology \textit{does} have a definable basis, and this issue is irrelevant.
\end{warning}

\begin{remark}\label{R: ctble expansion} It is easy to see that Hausdorff geometric structures are preserved under naming a countable set of constants (and indeed the same will hold of all additional properties we discuss). This will be a useful observation for the following reason: suppose $\mathcal K$ is a Hausdorff geometric structure, and $\mathcal M=(M,...)$ is a relic of $\mathcal K$ whose language is countable. Then we can add to $\mathcal L$ all parameters needed to define $\mathcal M$ (i.e. the universe $M$ and all basic relations of $\mathcal M$). Thus, we may assume that every $\emptyset$-definable set in $\mathcal M$ is $\emptyset$-definable in $\mathcal K$.
\end{remark}

We will see detailed accounts of various Hausdorff geometric structures in Section 9. For now, we list them without proof for the reader's intuition. Key examples include the complex field with the analytic (i.e. Euclidean) topology; any $\aleph_1$-saturated algebraically closed valued field with the valuation topology; any $\aleph_1$-saturated 1-h-minimal valued field with the valuation topology; any $\aleph_1$-saturated characteristic zero Henselian field with the valuation topology (in particular $\aleph_1$-saturated $p$-adically closed fields); and any $\aleph_1$-saturated o-minimal expansion of a real closed field with the order topology. 

Generalizing characteristic zero Henselian fields, one can also take a large class of \textit{\'ez }fields. Recall (see \cite{JTWY}) that the \textit{\'etale open topology} on a large field $K$ is a system of topologies on the sets $V(K)$ for all varieties $V$ over $K$, where basic open sets are taken to be images of \'etale morphisms. A pure field is \textit{\'ez} if every definable set is a finite union of definable \'etale open subsets of Zariski closed sets (see \cite{ez}). It is proved in \cite{JTWY} (Proposition 4.9) that the \'etale open topology of $K$ is induced by a field topology on $K$ if and only if it respects products of varieties in the obvious sense. Now suppose $K$ is an $\aleph_1$-saturated \'ez field whose \'etale open topology is induced by a field topology; we will see later that $K$ forms a Hausdorff geometric structure when equipped with the \'etale open topology on $K$.
\begin{remark} Note that it is still open whether many of our results can be proven for \textit{arbitrary} \'ez fields (where the \'etale open topology might not respect products). This case could still be useful, as it could have implications for relics of pseudo algebraically closed fields (in particular pseudo-finite fields).
\end{remark}

\begin{assumption}\textbf{From now until the end of section 7, fix a Hausdorff geometric structure $(\mathcal K,\tau)$. All sets $A,B,...$ of parameters are now assumed to be countable.}
\end{assumption}

In a previous draft of this paper, we proved several basic properties of Hausdorff geometric structures that are not used in the sequel. Nevertheless, we list a few of them below, as they may give some intuition on the interaction between the topology and dimension theory. The interested reader may want to check these facts as an easy exercise.

\begin{lemma}\label{L: HGS basic properties} Let $X\subset K^n$ be definable over $A$, and let $a\in X$ be generic over $A$.
\begin{enumerate}
    \item $X$ is locally closed in a neighborhood of $a$ -- that is, there is a neighborhood $U$ of $a$ in $K^n$ such that $U\cap X$ is relatively closed in $U$.
     \item Any first-order property over $A$ which holds of $a$, holds for all points in some relative neighborhood of $a$ in $X$.
   \item If $Y\subset X$ is definable and contains a neighborhood of $a$ in $X$, then $\dim(Y)=\dim(X)$.
   \item In particular, no infinite definable subset of $K^n$ is discrete.
    \end{enumerate} 
\end{lemma}

Before moving on, we make some further comments on Definition \ref{D: easy axioms}.

\begin{remark}\label{R: non-discrete} The Baire category axiom should be viewed as a substitute for the more common statement `acl dimension equals topological dimension.' In particular, this axiom is needed for Lemma \ref{L: HGS basic properties}(3)-(4). The reason for choosing our version of the axiom is that the latter becomes less useful when the topology is not definable.
\end{remark}

\begin{remark}\label{R: Baire} On a similar note to above, suppose $\tau$ has a definable basis. Then (modulo the generic local homeomorphism property and compactness) the Baire category axiom follows from the simpler statement `every definable open subset of $K^n$ has dimension $n$'. Moreover, the latter condition holds automatically, provided $K$ has no isolated points. Indeed, in a Hausdorff topology with no isolated points, all open sets are infinite, so -- when definable in a geometric structure -- one dimensional. Since the topology on $K^n$ is the product topology, and dimension in geometric structures is additive in Cartesian products, definable open sets contain $n$-dimensional boxes, which implies the desired conclusion. 

On the other hand, in the complex field with the analytic topology, the Baire category axiom follows easily from Baire's theorem -- hence the name for the axiom. 
\end{remark}

\begin{remark}\label{R: frontier inequality} One typically thinks of the `frontier inequality' as stipulating that $a\in\operatorname{Fr}(X)$ implies $\dim(a/A)<\dim(X)$. Our `strong' version (condition (3) above) also applies to some elements of $X$; however, the two conditions are equivalent if the closure of a definable set is always definable. As written, the strong frontier inequality essentially amounts to `generic local closedness' (see Lemma \ref{L: HGS basic properties}(2)).
\end{remark}

\subsection{Germs of Type-Definable Sets} In the next subsections, we will need to talk about the local behavior of definable sets and functions near a given point. An ideal way of doing this would be to use `infinitesimal neighborhoods'. If the topology $\tau$ is definable, one can generate an infinitesimal neighborhood of a point $x$ by intersecting infinitely many definable open neighborhoods of $x$. The result is a \textit{type-definable set}. An example of particular interest is the set of realizations of $\tp(x/A)$ for a small set $A$ (indeed, it follows by the frontier inequality that every $\pi(x,a)\in\tp(x/A)$ holds on a neighborhood of $x$).

In our case, the topology might not be definable; however, certain type-definable sets -- including complete types as above -- still give some intuition of infinitesimal neighborhoods (essentially, because the frontier inequality still applies as above). Our goal now is to develop a general notion of \textit{germed type-definable sets} which captures this idea. For instance, the frontier inequality will imply that all complete types over small sets are germed at every point. Importantly, the class of germed type-definable sets is richer than that of complete types, as new examples can be generated by closing under intersections and products (Lemma \ref{L: germed preservation}). This allows us to talk about germs of certain more intricate configurations when defining the key property of enough open maps (Definition \ref{D: enough open maps}) -- which ultimately makes this notion much easier to state.

To start, let us clarify what we mean by type-definable sets in this paper:

\begin{definition}
    Let $X\subset K^n$, and let $A$ be a countable parameter set.
    \begin{enumerate}
        \item $X$ is \textit{type-definable over $A$} if $X$ is a (potentially infinite) intersection of $A$-definable sets.
        \item In general, $X$ is \textit{type-definable} if it is type-definable over some countable set $B$.
        \item $X$ is a \textit{complete type over $A$} if $X$ is type-definable over $A$ and whenever $Y$ is definable over $A$, we either have $X\subset Y$ or $X\cap Y=\emptyset$.
        \item If $X$ is type-definable, we call $X$ a \textit{complete type} if it is a complete type over some countable set.
        \item A \textit{type-definable function} is a function $f:Y\rightarrow Z$ such that $Y$, $Z$, and the graph of $f$, are all type-definable (so we include $Y$ and $Z$ in the data of the function). Note that, by compactness, this is the same as the restriction of a definable function to a type-definable domain and target.
    \end{enumerate}
\end{definition}

By compactness (and the fact that $(\mathcal K,\tau)$ is $\aleph_1$-saturated), note that we have the following useful fact:

\begin{fact}\label{F: type def cofinal} Let $A$ be countable, and let $X=\bigcap\mathcal Y$ where $\mathcal Y$ is a collection of $A$-definable sets. Assume that $\mathcal Y$ is closed under finite intersections. Then $\mathcal Y$ is \textit{cofinal} in the definable supersets of $X$. That is, if $Z\supset X$ is any definable set, then there is $Y\in\mathcal Y$ with $Z\supset Y$.
\end{fact}

As is well-known in model theory, notions of dimension and genericity transfer naturally to type-definable sets:

\begin{definition}
    Let $X$ be type-definable over $A$. By $\dim(X)$, we mean any of the following values (which are all the same by compactness and Fact \ref{F: type def cofinal}):
    \begin{itemize}
        \item the smallest dimension of a definable set containing $X$
        \item the smallest dimension of an $A$-definable set containing $X$
        \item the largest value of $\dim(a/A)$ for $a\in X$.
    \end{itemize}
    If $a\in X$ with $\dim(a/A)=\dim(X)$, we say that $a$ is \textit{generic in $X$ over $A$}.
\end{definition}

Now let us discuss germs.

\begin{definition}\label{D: germs}
    Let $X$ be type-definable, and $a\in X$. We say that $X$ is \textit{germed at $a$} if there is a definable $Y\supset X$ such that, for all definable $Z$ with $X\subset Z\subset Y$, $Y$ and $Z$ agree in a neighborhood of $a$. If $f:X\rightarrow Y$ is a type-definable projection (i.e. a projection between type-definable sets), we say that $f$ is germed at $a$ if $X$ is germed at $a$ and $Y$ is germed at $f(a)$.
\end{definition}

\begin{example} For intuition, we give an example of a type-definable set which is not germed. Suppose $(\mathcal K,\tau)$ is $\mathbb C$ with the analytic topology. Now let $X\subset\mathbb C^2$ be the union of all lines $y=ax$ where $a\in\mathbb C-\mathbb Q$ -- in other words, the complement of the non-zero parts of all lines through $(0,0)$ with rational slope. Clearly, $X$ is type-definable over $\emptyset$. But a definable approximation to $X$ will only remove finitely many lines with rational slope, and thus does not give the correct germ at $(0,0)$.
\end{example}

One now easily checks the following:

\begin{lemma}\label{L: germed preservation} Let $X$, $Y$, and $Z$ be type-definable, $a\in X\cap Y$, and $b\in Z$.
    \begin{enumerate}
        \item If $X$ is definable, then $X$ is germed at $a$.
        \item If $X$ is type-definable over $A$, and $a\in X$ is generic over $A$, then $X$ is germed at $a$.
        \item If $X$ is a complete type, then $X$ is germed at $a$.
        \item If $X$ is germed at $a$ and $Y$ is germed at $a$, then $X\cap Y$ is germed at $a$.
        \item If $X$ is germed at $a$ and $Z$ is germed at $b$, then $X\times Z$ is germed at $(a,b)$.
    \end{enumerate}
\end{lemma}
\begin{proof}
    (1) is trivial. (2) uses the strong frontier inequality. (3) is an application of (2). (4) and (5) use Fact \ref{F: type def cofinal}.
\end{proof}

Note that if $X$ is germed at $a$, then $X$ determines a unique germ of open neighborhoods of $a$. We call this the \textit{germ of $X$ at $a$} (here we treat germs as equivalence classes of sets, where two sets are equivalent if they agree on a neighborhood of $a$). We will use germs to talk about local properties of $X$. Let us be more precise:

\begin{definition}\label{D: d-approximation}
    Let $f:X\rightarrow Y$ be a type-definable projection, $a\in X$, and suppose $f$ is germed at $a$.
    \begin{enumerate}
    \item A \textit{d-approximation} of $X$ at $a$ is a definable set $X'\supset X$ such that $\dim(X)=\dim(X')$ and $X$ and $X'$ have the same germ at $a$.
    \item A \textit{d-approximation} of $f$ at $a$ is a projection $f':X'\rightarrow Y'$, where $X'$ is a d-approximation of $X$ at $a$ and $Y'$ is a d-approximation of $Y$ at $f(a)$.
    \end{enumerate}
\end{definition}

\begin{definition}\label{D: d-local}
    Let $P$ be a property of a set with a distinguished point. We say that $P$ is \textit{d-local} if whenever $X$ is a definable set, $a\in X$, and $X'$ is a d-approximation of $X$ at $a$, then $P(X,a)$ holds if and only if $P(X',a)$ holds.

    We also make the analogous definition of $d$-local properties of projections (using Definition \ref{D: d-approximation}(2) instead of (1)).
\end{definition}

The main point of d-local properties is that they extend naturally from definable sets to germed type-definable sets. Namely, the following is well-defined:

\begin{definition}
    Let $P$ by a d-local property of a set with a distinguished point. Let $X$ be a type-definable set, let $a\in X$, and assume that $X$ is germed at $a$. We say that $X$ \textit{satisfies $P$ near $a$} if some (equivalently any) d-approximation of $X$ satisfies $P$ with the distinguished point $a$.

    As in Definition \ref{D: d-local}, we make the analogous definition for type-definable projections.
\end{definition}

In particular, the following are clear:

\begin{lemma}
    For a projection $f:X\rightarrow Y$ and a point $a$, each of the following properties is d-local:
    \begin{enumerate}
        \item $f$ is finite-to-one on some neighborhood of $a$.
        \item $f$ is \textit{locally open at $a$}: there are neighborhoods $U$ of $a$ in $X$, and $V$ of $f(a)$ in $Y$, such that $f$ restricts to an open map $U\rightarrow V$.
    \end{enumerate}
\end{lemma}

Before moving on, we also show the following useful facts:

\begin{lemma}\label{L: germ at generic} Suppose $X\subset K^n$ is definable over $A$, and $x\in X$. The following are equivalent:
\begin{enumerate}
    \item $X$ is a d-approximation of $\tp(a/A)$ at $a$.
    \item $a$ is generic in $X$ over $A$.
\end{enumerate}
\end{lemma}
\begin{proof} Suppose (1) holds. Then by definition, $\dim(X)=\dim(a/A)$, so (2) holds.

Now suppose (2) holds. Then $\dim(X)=\dim(a/A)$. So to show that $X$ is a d-approximation of $\tp(a/A)$, it suffices to show that $X$ has the same germ as $\tp(a/A)$ at $a$. To do this, let $Y$ be any d-approximation of $\tp(a/A)$ at $a$. We want to show that $X$ and $Y$ agree in a neighborhood of $a$. Since $Y$ is a d-approximation of $\tp(a/A)$, some neighborhood of $a$ in $Y$ is contained in $X$. So we need to show the same statement with $X$ and $Y$ reversed.

Shrinking $Y$ if necessary, by Fact \ref{F: type def cofinal}, we may assume $Y$ is definable over $A$. If no neighborhood of $a$ in $X$ is contained in $Y$, then $a$ is in the frontier of $Y-X$. But $Y-X$ is definable over $A$, so the strong frontier inequality gives $$\dim(a/A)<\dim(Y-X)\leq\dim(Y)=\dim(a/A),$$ a contradiction.
\end{proof}

\begin{lemma}\label{L: independent germs} Suppose that $x\in K^n$ and $A\subset B$ are parameter sets. If $x$ is independent from $B$ over $A$, then $\tp(x/A)$ and $\tp(x/B)$ have the same germ at $x$.
\end{lemma}
\begin{proof} Let $X$ be a d-approximation of $\tp(x/A)$. By Lemma \ref{L: germ at generic}, $x$ is generic in $X$ over $A$. Since $x$ is independent from $B$ over $A$, $x$ is also generic in $X$ over $B$. So by Lemma \ref{L: germ at generic} again, $X$ is a d-approximation of $\tp(x/B)$. Thus $\tp(x/A)$ and $\tp(x/B)$ have a common d-approximation at $x$, and so they have the same germ at $x$.
\end{proof}

\subsection{Enough Open Maps}

The main technical result of the paper concerns Hausdorff geometric structures satisfying an additional axiom about the openness of certain definable maps (called \textit{enough open maps}). Because this axiom is rather technical to state, we give it a separate treatment. Let us start by giving an example to motivate the definition. 

Later in the paper, we will show that non-locally modular definable strongly minimal $\mathcal K$-relics can reconstruct a fragment of the topology $\tau$, assuming only that $(\mathcal K,\tau)$ has enough open maps. To sketch the setting, suppose $(\mathcal K,\tau)$ is an algebraically closed valued field (so $\tau$ is the valuation topology), and $\mathcal M=(K,...)$ is a strongly minimal relic with universe $K$ (for example, $\mathcal M$ could be the pure field structure). Let $X\subset K^2$ be a plane curve with a frontier point at $(a,b)$. If $X$ happens to be definable in $\mathcal M$, say over a tuple $t$, we want to `recognize' the frontier point $(a,b)$ using only the language of $\mathcal M$ (in an ideal world, this would mean showing $(a,b)\in\acl_{\mathcal M}(t)$). Our only tool for doing this is the existence of a Morley rank 2 family of plane curves in $\mathcal M$. For illustration, suppose this family is the family of all lines $y=\alpha x+\beta$ in $K^2$ (parametrized by $(\alpha,\beta)\in K^2$).

Now after potentially tweaking the setup (e.g. if addition is definable, we may want to translate $X$ by a generic), our strategy is as follows: we fix an independent generic slope $\alpha$, which determines a line $l$ of slope $\alpha$ through $(a,b)$. Suppose $l$ intersects $X$ in $m$ points. One can show, using the genericity of $\alpha$ (and the aforementioned `tweaks'), that $X$ and $l$ are not tangent at any of these $m$ points (that is, $X$ and $l$ are \textit{transverse} at each point). It then follows that if we translate $l$ to a nearby line $l'$ (still of slope $\alpha$), none of these $m$ points can `disappear' -- that is, $X\cap l'$ will contain at least $m$ points, one near each point of $X\cap l$. Meanwhile, since $(a,b)\in\operatorname{Fr}(X)$, infinitely many such translations also produce an $(m+1)$-st point of $X\cap l'$, this time near $(a,b)$. We conclude (by strong minimality) that cofinitely many of the lines of slope $\alpha$ intersect $X$ in at least $m+1$ points, and thus that $l\in\acl_{\mathcal M}(t\alpha)$. It is now an exercise in forking calculus (using the independence of $\alpha$ from $tab$) to show that this implies $(a,b)\in\acl_{\mathcal M}(t)$.

In the abstract topological setting, we need a replacement for the `transversality' of $X$ and $l$ above: that is, we need an axiom saying that the $m$ given intersection points of $X\cap l$ will not disappear when we perturb $l$. This is what we wish to define. It will be most natural to work at the level of types and germs as in the previous subsection.\\

To start, we define:

\begin{definition}\label{D: transversality config} Let $x\in K^m$ and $y\in K^n$. Let $A$, $B\supset A$, and $C\supset A$ be parameter sets. We call $(x,y,A,B,C)$ a \textit{transversality configuration} if the following hold:
\begin{enumerate}
    \item $\dim(x/B)=\dim(y/C)$.
    \item $x\in\acl(By)$ and $y\in\acl(Cx)$.
    \item $x$ is independent from $C$ over each of $A$ and $Ay$.
\end{enumerate}
    If $(x,y,A,B,C)$ is a transversality condition, we further define the \textit{intersection family} of $(x,y,A,B,C)$ to be the $(BC)$-type-definable set $W$ of all $(x',y')$ such that $(x',y')\models\tp(x,y/C)$ and $x'\models\tp(x/B)$.
\end{definition}

To illustrate Definition \ref{D: transversality config}, we revisit the example given above. In the language of the example, one should think of the following analogs:

\begin{itemize}
    \item $x$ is analogous to a point of $X\cap l$.
    \item $y$ is analogous to the parameter for the line $l$.
    \item $A=\emptyset$, and $\tp(x/A)$ is analogous to $K^2$ (or rather, $K^2$ is analogous to a d-approximation of $\tp(x/A)$).
    \item $B$ is a parameter defining $X$, and $\tp(x/B)$ is analogous to $X$.
    \item $C=\alpha$.
    \item $\tp(x/Ay)$ is analogous to $l$ (or rather, the germ of $l$ at $x$).
    \item The assumption that $\dim(x/B)=\dim(y/C)$ replaces the strong minimality of the set of lines of slope $\alpha$. 
    \item The assumption that $x\in\acl(By)$ replaces the finiteness of $X\cap l$.
    \item The assumption that $y\in\acl(Cx)$ replaces the fact that $l$ is determined by its slope and one point.
    \item The independence of $x$ and $C$ over $A$ and $Ay$ replaces the fact that $\alpha$ is chosen independently of the other data. In particular, this means that $\tp(x/Cy)$ is also analogous to $l$.
    \item If $W$ is the intersection family of $(x,y,A,B,C)$, then the fibers of the projection $W\rightarrow\tp(y/C)$ correspond to the intersections of $X$ with translates of $l$.
\end{itemize}

Now suppose $(x,y,A,B,C)$ is any transversality configuration, with intersection family $W$. The general idea is to think of $\tp(y/C)$ as parametrizing the family of conjugates of $\tp(x/Cy)$, and the projection $W\rightarrow\tp(y/C)$ as recording the intersection points of each such conjugate with $\tp(x/B)$. The desired conclusion (at least in the relevant cases) is that the point $(x,y)\in W$ `moves' when $y$ does. Abstractly, we want the projection $W\rightarrow\tp(y/C)$ to be \textit{open at $(x,y)$}.

Before giving a precise statement, we need to check that $W$ is germed at $(x,y)$:

\begin{lemma}\label{L: W germed} Let $(x,y,A,B,C)$ be a transversality configuration, with family of intersections $W$.
\begin{enumerate}
    \item $W$ is germed at $(x,y)$.
    \item $\dim(W)=\dim(x/B)=\dim(y/C)$.
   
\end{enumerate}
\end{lemma}
\begin{proof}
    \begin{enumerate}
        \item By repeated applications of instances of Fact \ref{L: germed preservation}. Namely, $W$ is the intersection of the complete type $\tp(x,y/C)$ with the set of $(x',y')$ satisfying $x'\models\tp(x/B)$; and the latter is the product of $\tp(x/B)$ with (the definable set) $K^n$ (where $n$ is the length of $y$).
        
        \item First, we check that $\dim(W)\leq\dim(x/B)$. Indeed, let $(x',y')\in W$. Then $y'\in\acl(Cx')$ and $\dim(x'/BC)\leq\dim(x'/B)$, so $\dim(x'y'/BC)\leq\dim(x/B)$.

        Now we show that $\dim(W)\geq\dim(x/B)$. It suffices to show that every d-approximation of $W$ has dimension at least $\dim(x/B)$. So let $W'$ be such a d-approximation. Using Fact \ref{F: type def cofinal}, after shrinking if necessary, we can write $W'$ as the set of $(x',y')$ satisfying $(x',y')\in Z$ and $x'\in X_1$, where $X_1$ is a d-approximation of $\tp(x/B)$ and $Z$ is a d-approximation of $\tp(xy/C)$. Let us also fix a d-approximation $X$ of $\tp(x/A)$. By Fact \ref{F: type def cofinal}, we can assume that $X$, $X_1$, and $Z$ are definable over $A$, $B$, and $C$, respectively.
        
        Now let $S$ be the set of $x'$ such that $(x',y')\in Z$ for some $y'$. Then $S$ is definable over $C$ and contains all realizations of $\tp(x/C)$. So the germ of $S$ at $x$ contains the germ of $\tp(x/C)$ at $x$ (this means that there are a neighborhood $V$ of $x$, and a $C$-definable set $Y$ containing $x$, such that $Y\cap V\subset S\cap V$). On the other hand, since $x$ and $C$ are independent over $A$, and by Lemma \ref{L: independent germs}, the germ of $\tp(x/C)$ at $x$ is the same as the germ of $\tp(x/A)$ at $x$ -- that is, the germ of $X_1$ at $x$. In other words, we have shown that the germ of $S$ at $x$ contains the germ of $X_1$ at $x$, and thus there is a neighborhood $U$ of $x$ such that $X_1\cap U\subset S\cap U$.
        
        Now by construction, $x$ is generic in $X_1$ over $B$. By the Baire category axiom, there is $x'\in U$ which is generic in $X_1$ over $BC$. Thus $x'\in S$, and so there is some $y'$ with $(x',y')\in Z$. Thus $(x',y')\in W'$, and moreover, $$\dim(x'y'/BC)\geq\dim(x'/BC)=\dim(X_1)=\dim(x/B),$$ which completes the proof.
        \end{enumerate}
\end{proof}

Now that we know $W$ is germed at $(x,y)$, we can define various versions of what it means for a transversality configuration to actually be transverse:

\begin{definition}\label{D: transverse}
    Let $(x,y,A,B,C)$ be a transversality configuration, with intersection family $W$.
    \begin{enumerate}
        \item $(x,y,A,B,C)$ is \textit{weakly transverse} if $W\rightarrow\tp(y/C)$ is finite-to-one near $(x,y)$.
         \item $(x,y,A,B,C)$ is \textit{transverse} if $W\rightarrow\tp(y/C)$ is locally open and finite-to-one near $(x,y)$. 
        \item $(x,y,A,B,C)$ is \textit{strongly transverse} if $x$ is independent from $C$ over $B$ (equivalently, if $(x,y)$ is generic in $W$ over $BC$).
    \end{enumerate}
\end{definition}

We often abbreviate `weakly transverse transversality configuration' by `weakly transverse configuration', and similarly for transverse and strongly transverse configurations.

One can check that strong transversality implies transversality, which implies weak transversality, though we will not need this.

We are finally ready to introduce the notion of a Hausdorff geometric structure having \emph{enough open maps}. The name refers to the local openness of the map $W\rightarrow\tp(y/C)$ above (see Definition \ref{D: transverse}) for a transversality configuration $(x,y,A,B,C)$. We would like to say that `enough of the time', this map is open at $(x,y)$ -- or in the language of Definition \ref{D: transverse}, `enough' transversality configurations are transverse.  

Our easiest desired conclusion would be that every transversality configuration is transverse. However, this seems too much to ask for. Instead, we need a more complicated statement involving all three transversality notions:

\begin{definition}\label{D: enough open maps} We say that $(\mathcal K,\tau)$ has \textit{enough open maps} if the following holds: let $(x,y,A,B,C)$ and $(x,y,A,B,C')$ be transversality configurations. If $(x,y,A,B,C)$ is weakly transverse, and $(x,y,A,B,C')$ is strongly transverse, then $(x,y,A,B,C)$ is transverse. 
\end{definition}

\begin{remark} A simpler statement would have been `every weakly transverse configuration is transverse'. However, the condition of weak transversality alone did not seem to be enough to imply transversality in examples. So one could think of Definition \ref{D: enough open maps} as saying that, in nice enough situations (i.e. when strong transversality is consistent), transversality can only fail for a very good reason (being non-weakly-transverse).
\end{remark}

\subsection{The Open Mapping Property and Enough Open Maps}

Due to the technical nature of the notion of $(\mathcal K,\tau)$ having enough open maps, we will now provide two other notions, each of which implies \textit{enough open maps} and is easier to verify in examples. The first notion below is to be thought of as an `algebraically closed' condition, and generalizes the open mapping theorem from complex analysis. The second is a `characteristic zero' condition, and essentially amounts to a version of Sard's Theorem.

Before giving either of these notions, we need to develop `smoothness' in Hausdorff geometric structures:

\begin{definition}\label{D: smooth} By a \textit{notion of smoothness} on $(\mathcal K,\tau)$, we mean a map $X\mapsto X^S$ sending each definable set $X\subset K^n$ (for all $n$) to a (not necessarily definable) subset $X^S\subset X$, satisfying the following:
\begin{enumerate}
    \item If $x$ is generic in $X$ over any set of parameters defining $X$, then $x\in X^S$.
    \item If $x\in X^S$ and $y\in Y^S$ then $(x,y)\in X^S\times Y^S$.
    \item If $x\in X^S$ and $\sigma$ is a coordinate-permutation, then $\sigma(x)\in(\sigma(X))^S$.
    \item If $Z\subset X\times Y$ are definable and $W=\{(x,x,y):(x,y)\in Z\}$, then for all $x,y$ we have $(x,y)\in Z^S$ if and only if $(x,x,y)\in W^S$.  
    \item Suppose $f:X\rightarrow Y$ is an $A$-definable projection, and $x\in X$ and $y=f(x)\in Y$ are both generic over $A$. Let $Z$ be definable over $B$ so that $y\in Z$ is generic over $B$ and $Z\cap U\subset Y\cap U$ for some neighborhood $U$ of $y$. Then $x\in(f^{-1}(Z\cap Y))^S$.
    \item The assertion `$x\in X^S$' is a d-local property of $x$ and $X$.
\end{enumerate}

If $x\in X^S$, we say that \textit{$x$ is smooth in $X$} or \textit{$X$ is smooth at $x$}.
\end{definition}

\begin{remark}
    Definition \ref{D: smooth}(5) is our attempt at capturing the \textit{submersion theorem} of differential geometry (asserting that every submersion of smooth manifolds is locally diffeomorphic to a coordinate projection of affine spaces) and its corollary on fibers (that preimages of smooth manifolds under submersions have a smooth structure). For technical reasons involving the Frobenius map, we could not find a more direct statement of this nature that holds in ACVF and is suitable for our needs. For the reader's intuition, we note the informal translation between our statement and the usual statement. Given our $f:X\rightarrow Y$, the genericity of $x\in X$ and $f(x)\in Y$ lets us informally visualize $f$ as a submersion near $x$. The submersion theorem would then allow us to visualize $X$ as $W\times Y$ for some $W$, and thus $x=(w,y)$ for some $w$ and $y$. The restriction of $X$ over $Z$ then corresponds to $W\times Z$, which remains smooth at $(w,y)$ by looking at $W$ and $Z$ separately and applying Definition \ref{D: smooth}(2) (where the smoothness of $Z$ comes from the genericity of $y\in Z$ over $B$).
\end{remark}

Thus, informally, a notion of smoothness is any d-local notion which holds generically and is closed under products, permutations, concatenations, and preimages under nice enough maps. In many natural examples of fields (including ACVF), there is a canonical such notion coming from the smooth locus of a variety (see Lemma \ref{L: ez smooth}): one takes $X^S$ to be the points that are locally open in the smooth locus of the Zariski closure of $X$. For o-minimal expansions of real closed fields, one can fix $n\geq 1$ and declare $x\in X^S$ if after restricting to a neighborhood of $x$, $X$ becomes a $C^n$-submanifold of the ambient space.

By d-locality, any notion of smoothness extends to germs of type-definable sets. Thus, if $X$ is type-definable and germed at $x$, the assertion `$x\in X^S$' is well-defined. We note the translations of Definition \ref{D: smooth} (1) to this context:

\begin{lemma}\label{L: type version of smooth} Let $X\mapsto X^S$ be a notion of smoothness on $(\mathcal K,\tau)$. Let $x\in K^n$, and let $A$ be a set of parameters. Then $\tp(x/A)$ is smooth at $x$.
\end{lemma}

In particular, we conclude:

\begin{lemma}\label{L: W smooth} Fix a notion of smoothness $X\mapsto X^S$ on $(\mathcal K,\tau)$. Let $(x,y,A,B,C)$ be a transversality configuration, with family of intersections $W$. Then $(x,y)\in W^S$.
\end{lemma}
\begin{proof} By Lemma \ref{L: W germed}, $W$ is germed at $(x,y)$, so the conclusion makes sense. Let $W'$ be a d-approximation of $W$ at $(x,y)$. By Fact \ref{F: type def cofinal}, we may assume $W'=\{(x,y)\in X:x\in Z\}$, where $X$ is some d-approximation of $\tp(x,y/C)$ and $Z$ is some d-approximation of $\tp(x/B)$. We may assume that $X$ and $Z$ are definable over $C$ and $B$, respectively.

Let $Y$ be a d-approximation of $\tp(x/C)$, which we may assume is definable over $C$. Shrinking $X$ if necessary, we may assume the projection of $X$ is contained in $Y$.

We want to show that $(x,y)$ is smooth in $W'$, by applying Definition \ref{D: smooth}(5). Almost all the hypotheses are already met. It remains only to show that the germ of $Z$ at $x$ is contained in the germ of $Y$ at $x$ -- that is, that the germ of $\tp(x/B)$ at $x$ is contained in the germ of $\tp(x/C)$ at $x$. But this follows from Lemma \ref{L: independent germs}. Indeed, since $x$ and $C$ are independent over $A$, $\tp(x/A)$ and $\tp(x/C)$ have the same germ at $x$. So it suffices to show that the germ of $\tp(x/B)$ at $x$ is contained in the germ of $\tp(x/A)$ at $x$. And this is clear, since $B\supset A$.
\end{proof}

We now give our first simplified condition:

\begin{definition}\label{D: open mapping thm} Let $(\mathcal K,\tau)$ be a Hausdorff geometric structure. We say that $(\mathcal K,\tau)$ has the \textit{open mapping property} if there is a notion of smoothness $X\mapsto X^S$ satisfying the following: suppose $X$ and $Y$ are definable of the same dimension, $f:X\rightarrow Y$ is a projection, $x\in X^S$, and $f(x)\in Y^S$. If $f$ is finite-to-one in a neighborhood of $x$, then $f$ is open in a neighborhood of $x$.
\end{definition}

As with Lemma \ref{L: type version of smooth}, we immediately get the following for germed sets:

\begin{lemma}\label{L: type version of open mapping} Assume $(\mathcal K,\tau)$ has the open mapping property, witnessed by the notion of smoothness $X\mapsto X^S$. Let $X$ and $Y$ be type-definable of the same dimension, and $f:X\rightarrow Y$ a projection which is germed at $x\in X$. If $f$ is finite-to-one near $x$, then $f$ is open near $x$.
\end{lemma}

The argument in \cite{CasACF0} crucially used that the complex field has the open mapping property, which follows from the usual open mapping theorem. Moreover, we will see later (Corollary \ref{C: acvf axioms}) that ACVF has the open mapping property in all characteristics. Now we show:

\begin{proposition}\label{P: open thm implies open maps} Assume that $(\mathcal K,\tau)$ has the open mapping property. Then every weakly transverse configuration is transverse. Thus, $(\mathcal K,\tau)$ has enough open maps.
\end{proposition}
\begin{proof} Let $(x,y,A,B,C)$ be a weakly transverse configuration, with family of intersections $W$. By Lemma \ref{L: W germed}, $W$ is germed and has the same dimension as $\tp(y/C)$. By assumption, $W\rightarrow\tp(y/C)$ is finite-to-one near $(x,y)$. Thus, by Definition \ref{D: open mapping thm}, $W\rightarrow\tp(y/C)$ is open near $(x,y)$. Thus $(x,y,A,B,C)$ is transverse.
\end{proof}

\subsection{Differentiability and Enough Open Maps}

We now introduce \textit{differentiability}, which complements the open mapping property as another simplified condition implying enough open maps. Roughly speaking, we want to say that $(\mathcal K,\tau)$ can be equipped with a tangent space functor satisfying versions of the Inverse Function Theorem and Sard's Theorem. 

\begin{definition}\label{D: differentiable} We say that $(\mathcal K,\tau)$ is \textit{differentiable} if there are a notion of smoothness $X\mapsto X^S$ on $(\mathcal K,\tau)$, and a covariant functor $\mathcal F:\mathcal C\rightarrow\mathcal D$, such that the following hold:
\begin{enumerate}
    \item $\mathcal D$ is the category of $F$-vector spaces for some field $F$.
    \item The objects of $\mathcal C$ are smooth-pointed definable sets in $\mathcal K$: that is, pairs $(X,x)$ where $X\subset K^n$ is definable and $x\in X^S$.
    \item For each object $(X,x)$ in $\mathcal C$, the dimension of $\mathcal F(X,x)$ as an $F$-vector space is precisely $\dim X$.
    \item The morphisms of $\mathcal C$ are compositions of inclusions and projections.
    \item $\mathcal F$ sends inclusions in $\mathcal C$ to inclusions in $\mathcal D$.
    \item If $x\in X^S$ and $y\in Y^S$, then $\mathcal F(X\times Y,(x,y))$ is the direct product of $\mathcal F(X,x)$ and $\mathcal F(Y,y)$. Moreover, $\mathcal F$ sends projections in $\mathcal C$ to the corresponding projections in $\mathcal D$.
    \item $\mathcal F$ sends constant functions in $\mathcal C$ to the zero map in $\mathcal D$.
    \item (Weak Inverse Function Theorem) If $f:(X,x)\rightarrow(Y,y)$ is a morphism in $\mathcal C$, and $\mathcal F(f)$ is an isomorphism, then $f$ is an open map in a neighborhood of $x$.
    \item (Sard's Theorem) If a morphism $f:(X,x)\rightarrow (Y,y)$ in $\mathcal C$ is definable over $A$, $x$ is generic in $X$ over $A$, and $y$ is generic in $Y$ over $A$, then $\mathcal F(f)$ is surjective. 
\end{enumerate}
\end{definition}

We will see later (Theorem \ref{T: ominimal} and Lemma \ref{L: ez differentiable}) that o-minimal expansions of fields, as well as pure Henselian fields of characteristic $0$, are differentiable. We now aim toward showing that differentiability implies enough open maps.\\

\textbf{For the rest of this subsection, we assume that $(\mathcal K,\tau)$ is differentiable, and fix witnessing data $X\mapsto X^S$, $\mathcal C$, $\mathcal D$, and $\mathcal F$.}

\begin{remark} Following the intuition of tangent spaces, we will denote the space $\mathcal F(X,x)$ by $T_x(X)$, and call it the \textit{tangent space to $X$ at $x$}.
\end{remark}

First we check:

\begin{lemma}\label{L: tangent space local} Tangent spaces are d-local. That is, suppose $x\in X\subset Y$, where $X$ is definable, $x\in X^S$, and $Y$ is a d-approximation of $X$ at $x$. Then $T_x(X)=T_x(Y)$.
\end{lemma}
\begin{proof}
    Since smoothness is d-local (Definition \ref{D: smooth}(6)), we have $x\in Y^S$, so $T_x(Y)$ is well-defined. Now by Definition \ref{D: differentiable}(5), the inclusion $X\xhookrightarrow{} Y$ induces an inclusion $T_x(X)\xhookrightarrow{} T_x(Y)$. On the other hand, since $Y$ is a d-approximation of $X$, we have $\dim(X)=\dim(Y)$, and thus $\dim(T_x(X))=\dim(T_x(Y))$ (by Definition \ref{D: differentiable}(3)). So $T_x(X)=T_x(Y)$. 
\end{proof}

By Lemma \ref{L: tangent space local}, we can extend tangent spaces to germed type-definable sets at smooth points. That is, if $X$ is type-definable and germed at $x$, and $x\in X^S$, then $T_x(X)$ is well-defined.

\begin{notation}
    If $\tp(x/A)$ is a complete type, we abbreviate $T_x(\tp(x/A))$ by $T(x/A)$.
\end{notation}

In the language of types, we conclude:

\begin{lemma}\label{L: type tangent space} Let $x$ and $y$ be tuples, and $A\subset B$ parameter sets. Let $0_y$ denote the identity of $T(y/A)$.
    \begin{enumerate}
        \item $\dim(T(x/A))=\dim(x/A)$.
        \item If $x$ and $B$ are independent over $A$, then $T(x/B)=T(x/A)$.
        \item $T(xy/A)\rightarrow T(y/A)$ is surjective, with kernel $T(x/Ay)\times\{0_y\}$.
    \end{enumerate}
\end{lemma}
\begin{proof} (1) follows from Definition \ref{D: differentiable}(3). (2) follows from Lemma \ref{L: independent germs}. For (3), the surjectivity follows first from Sard's Theorem. Now let $V_1=T(x/Ay)\times\{0_y\}$, and let $V_2$ be the kernel of $T(xy/A)\rightarrow T(y/A)$. We want to show that $V_1=V_2$. Notice that there is an inclusion $\tp(x/Ay)\times\{y\}\xhookrightarrow{}\tp(xy/A)$. By Definition \ref{D: differentiable}(5), there is an induced inclusion $V_1\xhookrightarrow{}T(xy/A)$, whose image is contained in $V_2$ (by composing with $T(xy/A)\rightarrow T(y/A)$ and applying Definition \ref{D: differentiable}(7)). Thus $V_1\subset V_2$, and it suffices to show that $\dim(V_1)=\dim(V_2)$. For this, since $T(xy/A)\rightarrow T(y/A)$ is surjective, and using (1), we have

$$\dim(V_2)=\dim(T(xy/A))-\dim(T(y/A))$$ $$=\dim(xy/A)-\dim(y/A)=\dim(x/A)=\dim(T(x/A))=\dim(V_1).$$
\end{proof}

The key observation in showing enough open maps is:

\begin{lemma}\label{L: kernel doesn't depend on C} Let $(x,y,A,B,C)$ be a transversality configuration, with family of intersections $W$. Then the kernel of $T_{(x,y)}(W)\rightarrow T(y/C)$ is precisely $(T(x/Ay)\cap T(x/B))\times\{0_y\}$. In particular, it does not depend on $C$.
\end{lemma}
\begin{proof} By Lemma \ref{L: type tangent space}(3), $T(xy/C)\rightarrow T(x/C)$ is surjective. But $y\in\acl(Cx)$ by assumption, so $\dim(xy/C)=\dim(x/C)$, thus $\dim(T(xy/C))=\dim(T(x/C))$, and thus $T(xy/C)\rightarrow T(x/C)$ is an isomorphism. Next, by Lemma \ref{L: type tangent space}(2), we have $T(x/C)=T(x/A)$. Now there is an inclusion $\tp(x/B)\xhookrightarrow{}\tp(x/A)$, so there is an induced inclusion $T(x/B)\xhookrightarrow{}T(x/A)=T(x/C)$. In particular, we have $T(x/B)\subset T(x/C)$. 

Let $V$ be the preimage of $T(x/B)$ under the isomorphism $T(xy/C)\rightarrow T(x/C)$. Note that there is an inclusion of $T_{(x,y)}(W)$ into $V$ (induced by the embedding $W\xhookrightarrow{}\tp(xy/C)$). On the other hand, by Lemma \ref{L: W germed}, we have $$\dim(T_{(x,y)}(W))=\dim(W)=\dim(x/B)=\dim(T(x/B))=\dim(V).$$ So in fact $T_{(x,y)}(W)=V$ (that is, $T_{(x,y)}(W)$ is formed by restricting $T(xy/C)$ to those points whose $T_x(C)$ coordinate lies in $T_x(B)$).

In particular, the kernel of $T_{(x,y)}(W)\rightarrow T(y/C)$ is now the restriction of the kernel of $T(xy/C)\rightarrow T(y/C)$ to those elements whose $T_x(C)$ coordinate lies in $T_x(B)$. By Lemma \ref{L: type tangent space}(3), this is the same as $(T(x/Cy)\cap T(x/B))\times\{0_y\}$. Finally, since $x$ and $C$ are independent over $Ay$, Lemma \ref{L: type tangent space}(2) gives that $T(x/Cy)=T(x/Ay)$, and we are done.
\end{proof}

Finally, we conclude:

\begin{proposition}\label{P: differentiable implies enough opens} Under our assumption that $(\mathcal K,\tau)$ is differentiable, $(\mathcal K,\tau)$ has enough open maps.
\end{proposition}
\begin{proof} Let $(x,y,A,B,C)$ and $(x,y,A,B,C')$ be transversality configurations, with $(x,y,A,B,C)$ weakly transverse and $(x,y,A,B,C')$ strongly transverse. Let $W$ and $W'$ be their respective families of intersections. Let $f$ and $f'$ denote the maps $T_{(x,y)}(W)\rightarrow T(y/C)$ and $T_{(x,y)}(W')\rightarrow\tp(y/C')$, respectively. By Lemma \ref{L: W germed}, each of these maps goes between two vector spaces of the same dimension.

Since $(x,y,A,B,C')$ is strongly transverse, $(x,y)$ is generic in $W'$ over $BC'$ and $y$ is generic in $\tp(y/C)$ over $BC'$. Thus, by Sard's Theorem, $f'$ is surjective. By dimension considerations, $f'$ thus has trivial kernel. Then, by Lemma \ref{L: kernel doesn't depend on C}, $f$ also has trivial kernel. Finally, by dimension considerations again, $f$ is surjective, and thus an isomorphism. It now follows by the weak inverse function theorem that $W\rightarrow\tp(y/C)$ is open near $(x,y)$, and thus $(x,y,A,B,C)$ is transverse.
\end{proof}

\section{Strongly Minimal Relics}\label{s: relics}

We now study the behavior of strongly minimal structures definable in $\mathcal K$. Our goal is to make progress toward the Zilber trichotomy for such structures, by showing that (assuming non-local modularity) they are able to partially reconstruct the topology $\tau$. Thus, we now fix:

\begin{assumption}\label{A: K and M}
    \textbf{From now until the end of section 7, we fix a non-locally modular $\emptyset$-definable strongly minimal $\mathcal K$-relic $\mathcal M=(M,...)$. In other words, $\mathcal M$ is a non-locally modular strongly minimal structure; the set $M$ is a subset of some $K^d$; and every $\emptyset$-definable set in $\mathcal M$ is also $\emptyset$-definable in $\mathcal K$. We also assume the language of $\mathcal M$ is countable; that $\acl_{\mathcal M}(\emptyset)$ is infinite; and that there is a $\emptyset$-definable \textit{excellent family} of plane curves in $\mathcal M$ (see Definition \ref{D: excellent}).}
\end{assumption}

\begin{remark} Similarly to Remark \ref{R: ctble expansion}, one might be concerned about the assumptions above regarding $\emptyset$-definability and $\acl_{\mathcal M}(\emptyset)$. We point out that all of these conditions can be arranged harmlessly by adding a countable set of constants to the languages of $\mathcal M$ and $\mathcal K$.
\end{remark}

\begin{remark} Since $\mathcal K$ is $\aleph_1$-saturated and the language of $\mathcal M$ is countable, one checks easily that $\mathcal M$ is also $\aleph_1$-saturated. We use this throughout.
\end{remark}

\subsection{Notational Conventions}

Several confusions can arise from working simultaneously with two geometric structures. Before proceeding, we clarify a few issues:

\begin{convention} Following \cite{CasACF0}, we make the following notational conventions to distinguish between $\mathcal K$ and $\mathcal M$:
\begin{enumerate}
\item Unless otherwise stated, all tuples are assumed to be taken in $\mathcal M^{\textrm{eq}}$, and all parameter sets are assumed to be countable sets of tuples from $\mathcal M^{\textrm{eq}}$. 
		\item The term \textit{rank}, and the notation rk, will always refer to the notion of dimension in $\mathcal M$. 
        \item We use the notation $\mathcal M(A)$-\textit{definable} to refer to sets definable in $\mathcal M$ over $A$. We similarly use $\mathcal K(A)$-definable, $\mathcal M(A)$-interpretable, and $\mathcal K(A)$-interpretable.
		\item The term \textit{plane curve}, and all properties of plane curves and families of plane curves, are interpreted in the sense of $\mathcal M$. Similarly, the terms \textit{stationary}, \textit{stationary component}, and \textit{canonical base} always refer to $\mathcal M$-definable objects, and are interpreted in the sense of $\mathcal M$. 
		\item When referring to the dimension functions of $\mathcal K$, we will use the term \textit{dimension}, and the notation dim.
		\item The terms \textit{generic}, \textit{independent}, and \textit{algebraic}, and the notation acl(A), are always interpreted in the sense of $\mathcal K$. We will refer to the corresponding terms in $\mathcal M$ with the prefix $\mathcal M$ (e.g. $\mathcal M$-generic), and algebraic closure in $\mathcal M^{\textrm{eq}}$ will be denoted $\operatorname{acl}_{\mathcal M}(A)$. 
	\end{enumerate}
\end{convention}


\subsection{Stating the Main Result}

Let us begin our work with $\mathcal M$ by stating our first main goal. As stated above, the idea is to show that $\mathcal M$ detects $\tau$ in a precise sense. Following \cite{CasACF0}, we make the following definitions:

\begin{definition} Let $x=(x_1,...,x_n)\in M^n$ be a tuple. We say that $x$ is \textit{coordinate-wise generic} if each $x_i$ is generic in $M$ over $\emptyset$ (that is, in the sense of $\mathcal K$).
\end{definition}

\begin{definition}\label{independent projections} Let $X\subset M^n$ be $\mathcal M(A)$-definable, and let $\pi_i,\pi_j:M^n\rightarrow M$ be the $i$th and $j$th projections. We say that $\pi_i$ and $\pi_j$ are \textit{independent on} $X$ if for all  $x=(x_1,...,x_n)\in X$ generic over $A$, the elements $x_i$ and $x_j$ are independent over $A$.
\end{definition}

\begin{remark} It is easy to see that Definition \ref{independent projections} is unchanged if `generic' and `independent' are replaced by their counterparts in $\mathcal M$, and that moreover the definition as a whole does not depend on the particular parameter set $A$.
\end{remark}

\begin{definition}\label{D: weak detection of closures} We say that $\mathcal M$ \textit{weakly detects closures} if the following holds: Suppose $X\subset M^n$ is $\mathcal M(A)$-definable of rank $r\geq 0$, and let $x=(x_1,...,x_n)\in\overline X$ be a coordinate-wise generic point. Then $\operatorname{rk}(x/A)\leq r$. Moreover, one of the following holds:
\begin{enumerate}
    \item $\operatorname{rk}(x/A)<r$.
    \item For all $i\neq j$ such that the projections $\pi_i,\pi_j:M^n\rightarrow M$ are independent on $X$, the elements $x_i$ and $x_j$ are $\mathcal M$-independent over $A$.
\end{enumerate}
\end{definition}

\begin{example} To aid in understanding condition (2) above, we point out that it always holds for non-trivial plane curves. Indeed, let $X$, $A$, and $x$ be as in Definition \ref{D: weak detection of closures}, such that, moreover,  $X$ is a non-trivial plane curve. Then for generic $(y_1,y_2)\in X$, the $y_i$ are interalgebraic, and thus dependent, over $A$. So the two projections $X\rightarrow M$ are \textit{dependent}, and thus (2) is vacuous in this case.
\end{example}

Now our first main result is:

\begin{theorem}\label{T: closure thm} If $(\mathcal K,\tau)$ has enough open maps, then $\mathcal M$ weakly detects closures.
\end{theorem}

We postpone the proof to the next section.

\subsection{Coherence}

Before moving toward the proof of Theorem \ref{T: closure thm}, we briefly study the notion of \textit{coherence}, which allows us to move between the dim and rk functions smoothly, and in many instances to forgo computations with dim in favor of simpler computations with rk. The material in this subsection is analogous to \cite[\S 3.3]{CasACF0}.

We begin by noting the following basic facts, which are all easy and can be proven word for word as in \cite{CasACF0}:

\begin{lemma}\label{dim rk comparison} Let $X$ be $\mathcal M(A)$-interpretable, and let $a\in X$.
\begin{enumerate}
\item $\dim X=\rk X\cdot\dim M$.
\item $\dim(a/A)\leq\rk(a/A)\cdot\dim M$.
\item If $a$ is generic in $X$ over $A$, then $a$ is $\mathcal M$-generic in $X$ over $A$.
\end{enumerate}
\end{lemma}

In light of Lemma \ref{dim rk comparison}(2), the following makes sense:

\begin{definition} Let $a$ be a tuple and $A$ a set. We say that $a$ is \textit{coherent} over $A$ if $\dim(a/A)=\rk(a/A)\cdot\dim M$. We say that $a$ is \textit{coherent} if $a$ is coherent over $\emptyset$.
\end{definition}

We will aim to work with coherent tuples as much as possible, because computations with them tend to be significantly easier. The following facts are very helpful:

\begin{lemma}\label{L: coherent preservation} Let $X$ be $\mathcal M(A)$-interpretable, let $a\in X$, and let $b$ be any tuple.
\begin{enumerate}
\item $a$ is generic in $X$ over $A$ if and only if it is both $\mathcal M$-generic in $X$ over $A$ and coherent over $A$.
\item $ab$ is coherent over $A$ if and only if $a$ is coherent over $A$ and $b$ is coherent over $Aa$.
\item If $a$ is coherent over $A$ and $b\in\acl_{\mathcal M}(Aa)$, then $b$ is coherent over $A$ and $a$ is coherent over $Ab$.
\item If $a$ is coherent over $A$ and $a$ is independent from $B$ over $A$, then $a$ is coherent over $AB$, and $a$ is $\mathcal M$-independent from $B$ over $A$.
\item If $a$ is coherent over $A$ and $B$ is any set, then there is $b$ realizing $\operatorname{tp}_{\mathcal K}(a/A)$ with $b$ coherent over $AB$. 
\item There is a realization $a'$ of $\operatorname{tp}_{\mathcal M}(a/A)$ which is coherent over $A$.
\end{enumerate}
\end{lemma}
\begin{proof} (1), (2), and (3) are exactly as in \cite{CasACF0} (Lemmas 3.16 and 3.17). For (5), let $b$ be any relation of $\operatorname{tp}_{\mathcal K}(a/A)$ which is independent from $B$ over $A$, and apply (4).

For (6), choose an $\mathcal M(A)$-interpretable set $Y$ of minimal Morley rank and degree containing $a$, so that any $\mathcal M$-generic element of $Y$ over $A$ has the same $\mathcal M$-type as $a$ over $A$. Now let $a'$ be a generic element of $Y$ over $A$, and apply (1).

It remains only to show (4). So assume $a$ is coherent over $A$ and independent from $B$ over $A$. Then $\dim(a/A)=\dim(a/AB)=\rk(a/A)\cdot\dim(M)=k$, say. The statement of (4) would follow if we can show $k=\rk(a/AB)\cdot\dim(M)$. But $\rk(a/AB)\cdot\dim(M)\leq k$ follows since $\rk(a/AB)\leq\rk(a/A)$, and $\rk(a/AB)\cdot\dim(M)\geq k$ follows by Lemma \ref{dim rk comparison}(2).
\end{proof}

\section{The Proof of Theorem \ref{T: closure thm}}

Let us now begin the proof of Theorem \ref{T: closure thm}. Almost all of the proof is independent of the topology on $\mathcal K$, and uses only dimension computations in $\mathcal K$ and $\mathcal M$. In particular, most of the proof can be copied word-for-word from \cite{CasACF0}. The key difference occurs in an important geometric step in the proof, where the assumption of \emph{enough open maps} (Definition \ref{D: enough open maps}) replaces some application of complex geometry in the original proof. 

Rather than duplicating twenty pages of material, we will largely give an outline of the proof, pointing out which parts can be copied and which ones involve the topological axioms. In the former case, we will point the reader to the identical steps in \cite{CasACF0} for details. We will then focus almost exclusively on the latter case.

As in Definition \ref{D: weak detection of closures}, we are given an $\mathcal M(A)$-definable set with coordinate-wise generic closure point $x$. Our goal is to show that $\rk(x/A)\leq r$, and that one of the two options in the conclusion of Definition \ref{D: weak detection of closures} holds. The overarching structure of the proof will be an induction on the value $r=\rk(X)$, with almost the whole proof contained in the inductive step. Thus, we make the following convention (as we will see later, the case $r=0$ is trivial):

\begin{assumption}\label{IH} Until otherwise stated, we assume that $r\geq 1$ is fixed and the statement of Theorem \ref{T: closure thm} holds for all $r'<r$.
\end{assumption}

\subsection{The Main Argument}

As in \cite{CasACF0}, we deduce the inductive step of Theorem \ref{T: closure thm} from a different, more technical statement (Proposition \ref{P: main closure step} below), whose proof contains the main geometric idea of the argument as a whole. This proposition will apply to generic members of almost faithful families of non-trivial hypersurfaces. Let us first recall what this means:

\begin{definition}\label{D: hypersurface} An $r$-\textit{hypersurface} is an $\mathcal M$-definable subset of $M^{r+1}$ of rank $r$. An $r$-hypersurface $H$ is \textit{non-trivial} if for  $(x_0,...,x_r)\in H$ generic over the parameters defining $H$ in $\mathcal M$, any $r$ of $x_0,...,x_r$ have rank $r$ over the same parameters.
\end{definition}

\begin{remark}\label{R: nontrivial plane curve} Note that a plane curve is the same as a 1-hypersurface; in particular, one easily checks that Definition \ref{D: hypersurface} generalizes non-triviality for plane curves (introduced in Subsection 2.2). 
\end{remark}

As discussed in Section \ref{ss: sm}, we say that an $\mathcal M$-definable family $\{H_t:t\in T\}$ of $r$-hypersurfaces is \textit{almost faithful} if for each $t\in T$, there are only finitely many $t'\in T$ with $\rk(H_t\cap H_{t'})=r$. We will (later on) use the existence of such families. Namely, let $H$ be a stationary non-trivial hypersurface. Then, as also discussed in Section \ref{ss: sm}, there are an $\CM(\0)$-definable almost faithful definable family of (without loss of generality non-trivial) hypersurfaces $\{H_t:t\in T\}$, and an $\CM$-generic $\hat t\in T$, such that $H$ is almost equal to $H_{\hat t}$.  See \cite[Lemma 2.25]{CasACF0} as well as the discussion following Definition 8.11 \textit{loc. cit.} for details. 
Moreover, we may assume $T\sub M^n$ For some $n$, and if $H$ is defined over a coherent parameter (over $\emptyset$), then $\hat t$ will be generic in $T$ (i.e. in the sense of $\mathcal K$). \\

Now, the main technical step of Theorem \ref{T: closure thm} (analogous to \cite[Proposition 8.15]{CasACF0}) is:

\begin{proposition}\label{P: main closure step} Let $\mathcal H=\{H_t:t\in T\}$ be an almost faithful family of non-trivial $r$-hypersurfaces of rank $k>(r+1)\cdot\dim M$, and let $\mathcal C=\{C_s:s\in S\}$ be an excellent family of plane curves. Assume that each of $\mathcal H$ and $\mathcal C$ is $\mathcal M(A)$-definable for some countable set $A$. Let $\hat t\in T$ and $\hat x=(\hat x_0,...\hat x_r)\in M^{r+1}$ each be generic over $A$, and assume that $\hat x\in\overline{H_{\hat t}}$. Then at least one of the following happens:
		\begin{enumerate}
			\item $\rk(\hat x/A\hat t)<r$.
			\item There is some $x\in H_{\hat t}$ such that for each $i=1,...,r$, there are infinitely many $s\in S$ with $(\hat x_0,\hat x_i),(x_0,x_i)\in C_s$. 
			\item There is some $i\geq 2$ such that $\hat x_i\in\operatorname{acl}_{\mathcal M}(A\hat t\hat x_0...\hat x_{i-1})$.
		\end{enumerate}
	\end{proposition}

\begin{proof}
		For ease of notation, we will assume $A=\emptyset$. 
		
		Arguing by contradiction, assume that each of (1), (2), and (3) fails. As in \cite{CasACF0}, we will proceed by studying the pairwise intersections of $H_{\hat t}$ with a certain family of curves in $M^{r+1}$. We begin by defining the various objects we will need for the argument.

        \begin{notation}
            For the rest of the proof of Proposition \ref{P: main closure step}, we fix the following:
            \begin{itemize}
                \item Let $H\subset M^{r+1}\times T$ be the graph of $\mathcal H$.
                \item Let $U=S^r$.
                \item For $u=(s_1,...,s_r)\in U$, set 
                \[D_u=\{(x_0,...,x_r)\in M^{r+1}:(x_0,x_i)\in C_{s_i}\textrm{ for each }i=1,...,r\}.\]
                \item Let $\mathcal D$ be the family of $D_u$ as $u$ varies in  $U$, and let $D\subset M^{r+1}\times U$ denote its graph.
                \item Let $\hat p=(\hat p_1,...,\hat p_r)$ be a fixed element of $M^r$ generic over $\hat t\hat x$.
			\item Let $U_{\hat p}$ be the set of all $u=((p_1,q_1),...,(p_r,q_r))\in U$ such that $p_i=\hat p_i$ for each $i=1,...,r$.
				\item Let $\mathcal Y$ be the subfamily of $\mathcal D$ indexed by $U_{\hat p}$, with graph $Y\subset M^{r+1}\times U_{\hat p}$.
            \item Let $I$ be the set of $(x,u)\in Y$ with $x\in H_{\hat t}$.
            \item For $x\in M^{r+1}$, we let ${_xD}:=\{u\in U:x\in D_u\}$ and ${_xY}:={_xD}\cap U_{\hat p}$.
            \end{itemize}
        \end{notation}

        \begin{remark} A minor point of departure from \cite{CasACF0} is the definition of $I$. In \cite{CasACF0}, $I$ also had a $T$-coordinate, so one took triples $(x,t,u)$ with $x\in H_t\cap D_u$. Our definition essentially prohibits $t$ from varying. This will make the current version of the argument easier, though the role played by $I$ will be the same.
        \end{remark}

        The following facts are all now easy to see. Note that (5) is Lemma 8.21 of \cite{CasACF0}.

        \begin{fact}\label{F: D and U facts} For the various objects specified above, we have:
        \begin{enumerate}
            \item $U$ is stationary of rank $2r$, and $U_{\hat p}$ is stationary of rank $r$.
            \item $\mathcal D$ and $\mathcal Y$ are families of rank one subsets of $M^{r+1}$. 
            \item For generic $x\in M^{r+1}$, we have $\rk({_xD})=r$ and $\rk({_xY})=0$. 
            \item For independent generic $x,x'\in M^{r+1}$, we have $\rk({_xD}\cap{_{x'}D})=0$.
            \item The projection $Y\rightarrow M^{r+1}$ is finite-to-one and almost surjective.
            \item $\hat x$ is generic in $M^{r+1}$ over $\hat p$.
        \end{enumerate}
        \end{fact}
	
	By Fact \ref{F: D and U facts} (5), (6) and the genericity of $\hat x$, the set ${_{\hat x}Y}$ is non-empty and finite. 

    \begin{notation}
		For the rest of the proof of Proposition \ref{P: main closure step}, fix $\hat u\in{_{\hat x}Y}$.
	\end{notation}

The following facts about $\hat u$ are also straightforward (see \cite[Lemma 8.23, Lemma 8.24]{CasACF0} for the proofs):

	\begin{fact}\label{hat u facts} $\hat u$ is generic in each of the following senses:
        \begin{enumerate}
            \item in $U$ over $\emptyset$; in particular, $\hat u$ is coordinate-wise generic.
            \item in $U_{\hat p}$ over $\hat p$.
            \item in ${_{\hat x}D}$ over $\hat t\hat x$.
        \end{enumerate}
      \end{fact}

      Later, on multiple occasions, we will apply the inductive hypothesis of Theorem \ref{T: closure thm} with $\hat u$ as the distinguished closure point of a certain set. Fact \ref{hat u facts}(1) says that these applications will be valid.

	Now that we have the main pieces in place, let us outline the proof. The argument can be organized into the following four main steps:\\

        \textbf{Step I:} $\hat u$ is $\mathcal M$-generic in $U_{\hat p}$ over $\hat t\hat p$.\\
        
        \textbf{Step II:} If $x\in H_{\hat t}\cap D_{\hat u}$, then $x$ is generic in $H_{\hat t}$ over $\hat t\hat x$, in $M^{r+1}$ over $\hat x$, and in $D_{\hat u}$ over $\hat u\hat x$.\\
        
        \textbf{Step III:} If $x\in H_{\hat t}\cap D_{\hat u}$, then the projection $I\rightarrow U_{\hat p}$ is open in a neighborhood of $(x,\hat t,\hat u)$.\\
        
        \textbf{Step IV:} $\hat u$ is not $\mathcal M$-generic in $U_{\hat p}$ over $\hat t\hat p$.\\

        Note that Steps I and IV are contradictory, so this will conclude the proof. Step I is done in isolation from the other three. Meanwhile, II is used to prove III, and III is used to prove IV.

        Steps I and II have nothing to do with the topology, and are proven entirely by rank and dimension computations. They can be copied word-for-word from \cite{CasACF0}. We will not reproduce the details here.

        Step III (or rather an analogous statement) is done in \cite{CasACF0} using complex geometry. In that setting, II allows one to reduce III to a problem about smooth complex spaces and holomorphic maps, and apply an open mapping theorem from complex analysis. In the abstract setting, we will prove III using the assumption that $(\mathcal K,\tau)$ has enough open maps. Indeed, most of the required hypotheses are straightforward to verify, with the most difficult ones being given exactly by II.

        Step IV is the key topological argument. Roughly, we show using III that $\hat u$ belongs to the closure of an $\mathcal M(\hat t\hat p)$-definable set of small rank, and conclude using the inductive hypothesis. This argument is largely unchanged from \cite{CasACF0}, but we still give the details because it is the most crucial step in the proof of Theorem \ref{T: closure thm}.

        \begin{center}\textbf{Step I}\end{center}
        
        \begin{fact}\label{u gen over t} $\hat u$ is $\mathcal M$-generic in $U_{\hat p}$ over $\hat t\hat p$.
        \end{fact}
        \begin{proof} Exactly the same as Lemma 8.25 of \cite{CasACF0}.
        \end{proof}

        \begin{center}\textbf{Step II}\end{center}

        \begin{fact}\label{F: x gen over t and u} If $x\in H_{\hat t}\cap D_{\hat u}$, then $x$ is generic in each of the following senses:
        \begin{enumerate}
            \item in $H_{\hat t}$ over $\hat t\hat x$.
            \item in $M^{r+1}$ over $\hat x$.
            \item in $D_{\hat u}$ over $\hat u\hat x$.
        \end{enumerate}
        \end{fact}
        \begin{proof} Copy Lemmas 8.29, 8.34, 8.39, 8.40, 8.41, and 8.44 from \cite{CasACF0}.
        \end{proof}
	
        \begin{center}\textbf{Step III}\end{center}

        First, we need the following, which is also used in Step IV:
	
	\begin{lemma}\label{ab finite intersection} The intersection $H_{\hat t}\cap D_u$ is finite for all $u\in U_{\hat p}$ in some neighborhood of $\hat u$.
		\end{lemma}
	\begin{proof} Let $J$ be the set of $u\in U_{\hat p}$ such that $H_{\hat t}\cap D_u$ is infinite. So $J$ is $\mathcal M$-definable over $\hat t\hat p$, and if the lemma fails, then $\hat u\in\overline J$. We want to apply the inductive hypothesis to $J$. Thus we check: 
 \begin{claim} $\rk(J)<r$.
 \end{claim}
 \begin{claimproof} If not, there is $u\in J$ with $\rk(u/\hat t\hat p)=r$. Since $u\in J$, there is $x\in H_{\hat t}\cap D_u$ with $\rk(x/\hat t\hat pu)\geq 1$. So $\rk(xu/\hat t\hat p)\geq r+1$. On the other hand, $\rk(x/\hat t)\leq r$ since $x\in H_{\hat t}$, and $\rk(u/\hat px)=0$ since $u\in{_xY}$. Thus $\rk(xu/\hat t\hat p)\leq r$, a contradiction.
		\end{claimproof}
  Now suppose the lemma fails, so $\hat u\in\overline J$. Then by the inductive hypothesis we get $\rk(\hat u/\hat t\hat p)\leq\rk(J)<r$, contradicting Lemma \ref{u gen over t}.
  \end{proof}
	
	Recall that our aim in this step is to show that if $x\in H_{\hat t}\cap D_{\hat u}$, then the projection $I\rightarrow U_{\hat p}$ is open in a neighborhood of $(x,\hat u)$. So for the rest of Step III, let us fix $x\in H_{\hat t}\cap D_{\hat u}$. We will deduce the desired openness of  $I\rightarrow U_{\hat p}$ using an application of the assumption on enough open maps for $(\mathcal K,\tau)$. Namely, in the terminology of transversality configurations from section 3, we  show:
 
 \begin{itemize}
     \item $(x,\hat u,\emptyset,\hat t,\hat p)$ and $(x,\hat u,\emptyset,\hat u,\hat x)$ are transversality configurations.
     \item $(x,\hat u,\emptyset,\hat t,\hat p)$ is weakly transverse.
     \item $(x,\hat u,\emptyset,\hat t,\hat x)$ is strongly transverse.
 \end{itemize}

 It will follow from enough open maps that $(x,\hat u,\emptyset,\hat t,\hat p)$ is transverse. Thus $W_{\hat p}\rightarrow\tp(\hat u/\hat p)$ is open near $(x,\hat u)$, where $W_{\hat p}$ is the family of intersections of $(x,\hat u,\emptyset,\hat t,\hat p)$. Finally, we note that $I\rightarrow U_{\hat p}$ realizes the same germ as $W_{\hat p}\rightarrow\tp(\hat u/\hat p)$ at $(x,\hat u)$, so is also open at $(x,\hat u)$.

 Let us proceed. We begin by showing:

\begin{lemma}\label{L: the points make transversality configs} $(x,\hat u,\emptyset,\hat t,\hat p)$ and $(x,\hat u,\emptyset,\hat t,\hat x)$ are transversality configurations.
    \end{lemma}
    \begin{proof} The three properties in Definition \ref{D: transversality config} translate into the following three claims:
        \begin{claim}
            $\dim(x/\hat t)=\dim(\hat u/\hat p)=\dim(\hat u/\hat x)$.
        \end{claim}
        \begin{claimproof}
            By Fact \ref{F: x gen over t and u}(1), $\dim(x/\hat t)=\dim(H_{\hat t})$. By Fact \ref{hat u facts}(2), $\dim(\hat u/\hat p)=\dim(U_{\hat p})$. By Fact \ref{hat u facts}(3), $\dim(\hat u/\hat x)=\dim(_{\hat x}D)$. So it suffices to show that $\dim(H_{\hat t})=\dim(U_{\hat p})=\dim(_{\hat x}D)$. By Lemma \ref{dim rk comparison}, it moreover suffices to show that $\rk(H_{\hat t})=\rk(U_{\hat p})=\rk(_{\hat x}D)$. But $\rk(H_{\hat t})=r$ by construction, $\rk(U_{\hat p})=r$ by Fact \ref{F: D and U facts}(1), and $\rk(_{\hat x}D)=r$ by Fact \ref{F: D and U facts}(3).
        \end{claimproof}
        \begin{claim} $x\in\acl(\hat t\hat u)$, $\hat u\in\acl(\hat px)$, and $\hat u\in\acl(\hat xx)$.
        \end{claim}
        \begin{claimproof}
            Lemma \ref{ab finite intersection} gives $x\in\acl(\hat t\hat u)$ immediately. That $\hat u\in\acl(\hat px)$ follows from Fact \ref{F: D and U facts}(5). That $\hat u\in\acl(\hat xx)$ follows from Facts \ref{F: x gen over t and u}(2) and \ref{F: D and U facts}(4).
        \end{claimproof}
        \begin{claim}
            $x$ is independent from each of $\hat p$ and $\hat x$ over each of $\emptyset$ and $\hat u$.
        \end{claim}
        \begin{claimproof}
            That $x$ is independent from $\hat x$ over $\emptyset$ follows from Fact \ref{F: x gen over t and u}(2). That $x$ is independent from $\hat x$ over $\hat u$ follows from Fact \ref{F: x gen over t and u}(3). That $x$ is independent from $\hat p$ over $\hat u$ is automatic since $\hat p\in\acl(\hat u)$ (indeed $\hat p$ contains $\hat u$ among its coordinates).

            Finally, we check that $x$ is independent from $\hat p$ over $\emptyset$. First, it follows from Facts \ref{hat u facts}(2) and \ref{F: x gen over t and u}(3) that $(x,\hat u)$ is generic in $Y$ over $\hat p$. But by Fact \ref{F: D and U facts}, this implies that $x$ is generic in $M^{r+1}$ over $\hat p$, which is enough. 
        \end{claimproof}
    \end{proof}

    As $(x,\hat u,\emptyset,\hat t,\hat p)$ and $(x,\hat u,\emptyset,\hat t,\hat x)$  are transversality configurations, we can form the associated families of intersections,  $W_{\hat p}$ and $W_{\hat x}$. 
    Note that if $(x',u')\in W_{\hat p}$, then $\tp(x'/\hat t)=\tp(x/\hat t)$ and $\tp(x'u'/\hat p)=\tp(x\hat u/\hat p)$. In particular, $x'\in H_{\hat t}\cap D_{u'}$. The following is then immediate:

    \begin{lemma}\label{L: weakly transverse} $(x,\hat u,\emptyset,\hat t,\hat p)$ is weakly transverse.
    \end{lemma}
    \begin{proof}
        By Lemma \ref{ab finite intersection}.
    \end{proof}

    To finish setting up the application of enough open maps, we check:

    \begin{lemma}\label{L: strongly transverse} $(x,\hat u,\emptyset,\hat t,\hat x)$ is strongly transverse.
    \end{lemma}
    \begin{proof}
        An equivalent statement is that $x$ and $\hat x$ are independent over $\hat t$. This follows from Fact \ref{F: x gen over t and u}(1).
    \end{proof}

    We may now conclude:

    \begin{lemma}\label{L: transverse} $(x,\hat u,\emptyset,\hat t,\hat p)$ is transverse. That is, the projection $W_{\hat p}\rightarrow\tp(\hat u/\hat p)$ is open at $(x,\hat u)$.
    \end{lemma}
    \begin{proof} By the definition of enough open maps.
    \end{proof}

    To finish step III, it remains to relate Lemma \ref{L: transverse} back to the projection $I\rightarrow U_{\hat p}$. To that end, we show:

    \begin{lemma}\label{L: d approximation to open map} The projections $I\rightarrow U_{\hat p}$ and $W_{\hat p}\rightarrow\tp(\hat u/\hat p)$ have the same germ at $(x,\hat u)$. Thus, $I\rightarrow U_{\hat p}$ is open at $(x,\hat u)$.
    \end{lemma}
    \begin{proof}
        By their definitions, we can write $$I=Y\cap(H_{\hat t}\times U_{\hat p}),\;\;\;\; W_{\hat p}=\tp(x\hat u/\hat p)\cap(\tp(x/\hat t)\times\tp(\hat u/\hat p)).$$  So to show equality of germs, it is enough to show that each set in the intersection on the left  hand side has the same germ as the corresponding set in the intersection on the right hand side.  That is, it suffices to show the following:
        
        \begin{claim} $(x,\hat u)$ is generic in $Y$ over $\hat p$, and thus $Y$ realizes the germ of $\tp(x\hat u/\hat p)$.
        \end{claim}
        \begin{claimproof} The genericity statement follows from Facts \ref{hat u facts}(2) and \ref{F: x gen over t and u}(3). Now apply Lemma \ref{L: germ at generic}.
        \end{claimproof}

        Since the germ of a Cartesian product is the Cartesian product of the germs, to show the equality of the germs of $H_{\hat t}\times U_{\hat p}$ and $\tp(x/\hat t)\times\tp(\hat u/\hat p)$ we show: 
        
        \begin{claim} $x$ is generic in $H_{\hat t}$ over $\hat t$, and thus $H_{\hat t}$ realizes the germ of $\tp(x/\hat t)$.
        \end{claim}
        \begin{claimproof} The genericity statement follows from Fact \ref{F: x gen over t and u}(1). Now apply Lemma \ref{L: germ at generic}.
        \end{claimproof}
        
        \begin{claim} $\hat u$ is generic in $U_{\hat p}$, and thus $U_{\hat p}$ realizes the germ of $\tp(\hat u/\hat p)$.
        \end{claim}
        \begin{claimproof} The genericity statement follows from Fact \ref{hat u facts}(2). Now apply Lemma \ref{L: germ at generic}.
        \end{claimproof}
        Concluding the proof of the lemma. 
        \end{proof}
This finishes the proof of Step III.

\begin{center}\textbf{Step IV}\end{center}
To obtain a contradiction to Step I it remains to show: 
 \begin{lemma}\label{open is enough} $\hat u$ is not $\mathcal M$-generic in $U_{\hat p}$ over $\hat t\hat p$.
		\end{lemma}
	\begin{proof} Let $w_1,...,w_l$ be the distinct intersection points of $H_{\hat t}$ and $D_{\hat u}$ ($l$ is finite by Lemma \ref{ab finite intersection}). If $\hat x$ is among $w_1,...,w_l$, then (2) in the statement of Proposition \ref{P: main closure step} holds with $x=\hat x$. Thus, we assume the $l+1$ points $\hat x,w_1,...,w_l$ are distinct.

    Now let $Z$ be the set of all $u\in U_{\hat p}$ such that $H_{\hat t}\cap D_u$ is not of size $l$. So $Z$ is $\mathcal M$-definable over $\hat t\hat p$. Then the main point is the following:
	
	\begin{claim}\label{closure of big intersections} $\hat u\in\overline Z$.
	\end{claim}
	\begin{claimproof} Let $V$ be any neighborhood of $\hat u$ in $U_{\hat p}$. We will find some $u\in V\cap Z$.

 By Lemma \ref{L: d approximation to open map}, the projection $I\rightarrow U_{\hat p}$ is open in a neighborhood of each $(w_i,\hat u)$. Moreover, the projection $Y\rightarrow M^{r+1}$ is open in a neighborhood of $(\hat x,\hat u)$, by the generic local homeomorphism property of $(\mathcal K,\tau)$ (see Fact \ref{F: D and U facts} (5) and (6)). Then using each of these instances of local openness, and after applying various shrinkings, we can choose sets $W_0,...,W_l$ such that:
		\begin{enumerate}
			\item $W_0$ is a neighborhood of $\hat x$ in $M^{r+1}$, and for each $i\geq 1$, $W_i$ is a neighborhood of $w_i$ in $M^{r+1}$.
			\item The sets $W_0,...,W_l$ are pairwise disjoint (this can be arranged since the topology $\tau$ is Hausdorff).
			\item For each $x\in W_0$ there is some $u\in V$ such that $(x,u)\in Y$.
			\item For each $u\in V$ and $i\geq l$ there is some $x\in W_i$ with $(x,u)\in I$.
		\end{enumerate}
	Now using (1) and the fact that $\hat x\in\overline{H_{\hat t}}$, we can choose $x\in W_0\cap H_{\hat t}$, and then choose $u$ as in (3). Then, by (4), the intersection $H_{\hat t}\cap D_u$ contains $x$, in addition to one point in each $W_i$. By (2) all of these points are distinct, so $|H_{\hat t}\cap D_u|\geq l+1$, and thus $u\in V\cap Z$.
		\end{claimproof}
	
	We now finish the proof of Lemma \ref{open is enough} in two cases:
	
	\begin{itemize}
		\item First, suppose that $Z$ is generic in $U_{\hat p}$. Then since $U_{\hat p}$ is stationary (see Fact \ref{F: D and U facts}(1)), $U_{\hat p}-Z$ is non-generic in $U_{\hat p}$. Since $\hat u\in U_{\hat p}-Z$, we are done.
		\item Now suppose that $Z$ is non-generic in $U_{\hat p}$. So $\rk(Z)<r$. By Claim \ref{closure of big intersections} and the inductive hypothesis, we get $$\rk(\hat u/\hat t\hat p)\leq\rk(Z)<r,$$ so we are again done.
	\end{itemize}
	\end{proof}

Finally, by Steps I and IV (i.e. Fact \ref{u gen over t} and Lemma \ref{open is enough}), the proof of Proposition \ref{P: main closure step} is complete.
\end{proof}

\subsection{Generic Non-trivial $r$-Hypersurfaces}
	
	The next three subsections complete the proof of Theorem \ref{T: closure thm}, using Proposition \ref{P: main closure step} as the main tool. They are analogous to subsections 8.4-8.6 of \cite{CasACF0}. These remaining cases are essentially word-for-word copies of their analogs in \cite{CasACF0}, so we will be brief. As in \cite{CasACF0}, we use the following:
    
	\begin{definition}\label{*} Let $x=(x_0,...,x_r)\in M^{r+1}$ be a tuple, and $A$ a set of parameters. We say that $x$ \textit{satisfies} $*$ \textit{over} $A$, also denoted $*(x,A)$, if the following hold:
	\begin{enumerate}
		\item $\rk(x/A)\leq r$.
		\item If $\rk(x/A)=r\geq 2$ then for all $i\neq j\in\{0,...,r\}$ we have $\rk(x_ix_j/A)=2$.
	\end{enumerate}
	\end{definition}

	As in \cite{CasACF0}, we have the following properties: 
	
	\begin{lemma}\label{* preservation} Let $x$ and $y$ be tuples in $M^{r+1}$, and $A$ and $B$ sets of parameters.
		\begin{enumerate}
			\item If $x$ is $\mathcal M$-independent from $B$ over $A$, then $*(x,A)$ and $*(x,B)$ are equivalent.
			\item If $x$ and $y$ are coordinate-wise $\mathcal M$-interalgebraic over $A$, then $*(x,A)$ and $*(y,A)$ are equivalent.
			\item $*(x,A)$ is equivalent to $*(x,\operatorname{acl}_{\mathcal M}(A))$.
			\item If $*(x,A)$ and $B\supset A$ then $*(x,B)$.
		\end{enumerate}
	\end{lemma}

        Let us now proceed. First we show that in the context of Theorem \ref{P: main closure step}, the full statement of Theorem \ref{T: closure thm} follows from several applications of Proposition \ref{P: main closure step}:
	
	\begin{proposition}\label{generic non-trivial hypersurfaces}
		Let $\mathcal H=\{H_t:t\in T\}$ be an almost faithful family of non-trivial $r$-hypersurfaces of rank $k>(r+1)\cdot\dim M$, and assume $\mathcal H$ is $\mathcal M$-definable over a set $A$. Let $\hat t\in T$ and $\hat x=(\hat x_0,...,\hat x_r)\in M^{r+1}$ each be generic over $A$, and assume that $\hat x\in\overline{H_{\hat t}}$. Then $*(\hat x,A\hat t)$ holds.
	\end{proposition}
	\begin{proof} Exactly as in  \cite[Proposition 8.50]{CasACF0}.				\end{proof}

	\subsection{Stationary Non-trivial $r$-Hypersurfaces}
	We next prove Theorem \ref{T: closure thm} for \textit{all} stationary non-trivial hypersurfaces and \textit{all}  coordinate-wise generic points, without the added genericity assumptions of Proposition \ref{generic non-trivial hypersurfaces}. As in \cite[\S 8.5]{CasACF0}, the trick is to use a sequence of independent plane curves to `translate' the initial setup into a more generic one. That this is possible is provided by the local homeomorphism property of Definition \ref{D: easy axioms}.
 
	\begin{proposition}\label{stationary non-trivial hypersurfaces} Let $X$ be a stationary non-trivial $r$-hypersurface which is $\mathcal M$-definable over a set $A$. Let $\hat x=(\hat x_0,...,\hat x_r)$ be coordinate-wise generic, and assume that $\hat x\in\overline X$. Then $*(\hat x,A)$ holds.
	\end{proposition}
	\begin{proof}
		As described above, the first step is to `translate' $X$ and $\hat x$ along a sequence of `independent plane curves'. The following is straightforward and independent of the topology, and is proven in \cite{CasACF0} using pure rank and dimension computations:

		\begin{fact}\label{translators} There are plane curves $C_0,...,C_r$, tuples $c_0,...,c_r$ from $\mathcal M^{\textrm{eq}}$, a positive integer $k>(r+1)\cdot\dim M$, and an element $\hat y=(\hat y_0,...,\hat y_r)\in M^{r+1}$, which satisfy the following:
			\begin{enumerate}
				\item Each $C_i$ is stationary and non-trivial.
				\item Each $C_i$ is $\mathcal M$-definable over $c_i$, and each $c_i=\operatorname{Cb}(C_i)$.
				\item The $c_i$ are independent over $\emptyset$, and each individual $c_i$ is coherent of rank $k$ over $\emptyset$.
				\item The tuple $(c_0,...,c_r)$ is independent from $A\hat x$ over $\emptyset$.
                    \item Each $(\hat x_i,\hat y_i)$ is a generic element of $C_i$ over $c_i$. In particular, $\hat x$ and $\hat y$ are coordinate-wise $\mathcal M$-interalgebraic over $Ac_0...c_r$.
                    \item $\hat y$ is generic in $M^{r+1}$ over $A\hat x$.
				\end{enumerate} 
			\end{fact}
		\begin{proof} See Lemmas 8.56 and 8.59 of \cite{CasACF0}. In particular, (5) is by the proof of Lemma 8.59.
            \end{proof}
    
		Fix $C_0,...,C_r,c_0,...,c_r,k,\hat y$ as in Fact \ref{translators}. Let $c=(c_0,...,c_r)$. The tuple $\hat y$ is our `translated' $\hat x$. We now `translate' $X$ compatibly:  
		
		\begin{notation}\label{D'} We set $$Y=\{(y_0,...,y_r):\textrm{ for some }(x_0,...,x_r)\in X\textrm{ we have that each }(x_i,y_i)\in C_i\}.$$
			\end{notation}
		
		So $Y$ is $\mathcal M$-definable over $Ac$. It is easy to verify that $Y$ is also a non-trivial $r$-hypersurface. Moreover, the assumption that $\hat x \in \overline{X}$ transfers to $\hat y$ and $Y$:
		
		\begin{lemma} $\hat y\in\overline Y$.
		\end{lemma}
		\begin{proof}
		Let $V=V_0\times...\times V_r$ be any neighborhood of $\hat y$ in $M^{r+1}$. We will find an element $y\in V\cap Y$. Since each $(\hat x_i,\hat y_i)$ is generic in $C_i$ over $c_i$, the generic local homeomorphism property (see Definition \ref{D: easy axioms}) implies that the projection $\pi:C_i\rightarrow M$, $(x,y)\mapsto x$, is a homeomorphism near $(\hat x_i,\hat y_i)$. After applying various shrinkings, it follows that we can find a neighborhood $U=U_0\times...\times U_r$ of $\hat x$ in $M^{r+1}$, so that for each $i$ and each $u\in U_i$ there is some $v\in V_i$ with $(u,v)\in C_i$. Since $\hat x\in\overline X$, there is some $x=(x_0,...,x_r)\in U\cap X$. We thus obtain $y=(y_0,...,y_r)\in V$ with each $(x_i,y_i)\in C_i$, which implies that $y\in Y$, as desired.
		\end{proof}

        The rest of the proof of Proposition \ref{stationary non-trivial hypersurfaces} is identical to the analogous proof in\cite{CasACF0}, so we just give a brief outline.

        Since $\hat y\in\overline Y$, it belongs to the closure of one of the stationary components of $Y$. So let $Z\subset Y$ be a stationary component, $\mathcal M$-definable over $\operatorname{acl}_{\mathcal M}(Ac)$, with $\hat y\in\overline Z$. Let $z=\operatorname{Cb}(Z)$. The main point is:

        \begin{fact}\label{F: rk z big} $z$ is coherent of rank at least $k$ over $A$.
        \end{fact}
        \begin{proof} Exactly as in Lemmas 8.61, 8.62 of \cite{CasACF0}.
        \end{proof}

        By Fact \ref{F: rk z big}, we can realize $Z$ up to almost equality as a generic member of a large family. That is, let $\mathcal H=\{H_t:t\in T\}$ be an almost faithful $\CM(A)$-definable family of non-trivial $r$-hypersurfaces, and let $t\in T$ be generic over $A$, so that $Z$ is almost equal to $H_t$.

        \begin{lemma}\label{* y} Keeping the above notation, $*(\hat y,Azct)$ holds.
        \end{lemma}
        \begin{proof} If $\hat y\in\overline{H_{\hat t}}$, then Proposition \ref{generic non-trivial hypersurfaces} gives $*(\hat y,At)$, and we conclude by Lemma \ref{* preservation}(4). If $\hat y\notin\overline{H_{\hat t}}$, then $\hat y\in\overline{Z-H_{\hat t}}$. By assumption $\rk(Z-H_{\hat t})<r$, so by the inductive hypothesis, $\rk(\hat y/Azt)<r$, and we again conclude by Lemma \ref{* preservation}. (To use the inductive hypothesis, one most note that $\hat y$ is coordinate-wise generic; this is by Fact \ref{translators}(6)).
        \end{proof}

	To end the proof, we repeatedly use Lemma \ref{* preservation}. First, by the definition of $Z$ and the almost faithfulness of $\mathcal H$, we have $z,t\in\operatorname{acl}_{\mathcal M}(Ac)$, so we get $*(\hat y,Ac)$. Since $\hat x$ and $\hat y$ are coordinate-wise $\mathcal M$-interalgebraic over $Ac$, this implies $*(\hat x,Ac)$. Finally, since $c$ is independent from $A\hat x$, this implies $*(x,A)$.
			
	\end{proof}
	\subsection{The Remaining Cases}
	
	The rest of the proof of Theorem \ref{T: closure thm} is quite straightforward, and does not use any properties of the topology $\tau$ other than those true in all topological spaces. Thus, we do not repeat the whole proof here.
	
	\begin{proposition}\label{inductive step} Theorem \ref{T: closure thm} holds whenever $X\subset M^n$ has rank $r$ and $n\geq r+1$. 
		\end{proposition}
	\begin{proof} Exactly as in Propositions 8.68, 8.70, and 8.75 of \cite{CasACF0}. 
        \end{proof}

        We now drop the inductive assumption, i.e. Assumption \ref{IH}, and summarize the completed proof of the theorem:
        
			\begin{proof}[Proof of Theorem \ref{T: closure thm}] Let $X$, $n$, $A$, and $x=(x_1,...,x_n)$ be as in Definition \ref{D: weak detection of closures}. It is easy to see that the statement of the theorem holds if $x\in X$.

            We now induct on $r=\rk(X)$. If $r=0$ then $x\in X$, so we are done. Now assume $r\geq 1$ and the statement holds for all $r'<r$. Clearly $n\geq r$. If $n=r$ then, since $M^n$ is stationary, we either have $x\in X$ or $\rk(x/A)<r$. In either case, we are done. So assume $n\geq r+1$. Then by we done by Proposition \ref{inductive step}.
				\end{proof}

\section{Consequences of Weak Detection of Closures}\label{S: 1-dim}

We now give applications of weak detection of closures to trichotomy problems. The main point is that weak detection of closures allows for the detection of ramification points of certain maps, and that this can be used in some cases to detect tangency of plane curves. As a corollary, assuming an appropriate `purity of ramification' statement, we show that $\dim(M)=1$.

\subsection{Weakly Generic Intersections}

In sections 6 and 7, we will frequently consider geometric properties of pairwise intersections between definable families of sets. Many natural properties become more difficult to state in the abstract setup, because notions such as `variety' are not accessible. In order to make the presentation smoother, we begin by establishing some terminology. The following will be used extensively:

\begin{definition}\label{D: generic intersections} Let $\mathcal Y=\{Y_t:t\in T\}$ and $\mathcal Z=\{Z_u:u\in U\}$ be $\mathcal K(A)$-definable families of subsets of a $\mathcal K(A)$-definable set $X$.
\begin{enumerate}
      \item By a \textit{weakly generic $(\mathcal Y,\mathcal Z)$-intersection over $A$}, we mean a tuple $(x,t,u)$ such that $x\in X$, $t\in T$, $u\in U$, $(x,t)\in Y$, and $(x,u)\in Z$, are all generic over $A$.
      \item Let $(x,t,u)$ be a weakly generic $(\mathcal Y,\mathcal Z)$-intersection over $A$. We call $(x,t,u)$ a \textit{strongly generic $(\mathcal Y,\mathcal Z)$-intersection over $A$} if $t$ and $u$ are independent over both $A$ and $Ax$.
      \item Let $(x,t,u)$ be a weakly generic $(\mathcal Y,\mathcal Z)$-intersection over $A$. We call $(x,t,u)$ \textit{strongly approximable over $A$} if every neighborhood of $(x,t,u)$ contains a strongly generic $(\mathcal Y,\mathcal Z)$-intersection over $A$.
 \end{enumerate}
\end{definition}

\begin{remark} This is a slight abuse of notation, as Definition \ref{D: generic intersections} depends on $X$, $T$, and $U$. In our applications, these sets will be clear from context, so we just write $(\mathcal Y,\mathcal Z)$.
\end{remark}

The idea is that a weakly generic $(\mathcal Y,\mathcal Z)$-intersection is an intersection $x\in Y_t\cap Z_u$ which is `as generic as possible' when looking at each family individually; while a strongly generic $(\mathcal Y,\mathcal Z)$-intersection is as generic as possible when looking at the two families simultaneously.

Note that for $\mathcal Y$ and $\mathcal Z$, weakly generic $(\mathcal Y,\mathcal Z)$-intersections need not exist (e.g. if one of $Y$ and $Z$ concentrates only on a small portion of $X$). Moreover, there might be weakly generic $(\mathcal Y,\mathcal Z)$-intersections but no strongly generic $(\mathcal Y,\mathcal Z)$-intersections (e.g. if $\mathcal Y=\mathcal Z$ and the sets in $\mathcal Y$ are pairwise disjoint).

On the other hand, for families of plane curves in $\mathcal M$, one can always find both types of intersections:

\begin{lemma}\label{L: generic intersections exist} Let $\mathcal C=\{C_t:t\in T\}$ and $\mathcal D=\{D_u:u\in U\}$ be almost faithful $\mathcal M(A)$-definable families of plane curves in $\mathcal M$, each of rank at least 2.
\begin{enumerate}
    \item Suppose $t\in T$ and $u\in U$ are generic over $A$, and $x$ is generic in both $C_t$ over $At$, and $D_u$ over $Au$. Then $(x,t,u)$ is a weakly generic $(\mathcal C,\mathcal D)$-intersection over $A$.
    \item There is a tuple $(x,t,u)$ satisfying the assumptions of (1) above. In particular, there is a weakly generic $(\mathcal C,\mathcal D)$-intersection over $A$.
    \item Every weakly generic $(\mathcal C,\mathcal D)$-intersection over $A$ is strongly approximable over $A$. In particular, there is a strongly generic $(\mathcal C,\mathcal D)$-intersection over $A$.
\end{enumerate}
\end{lemma}
\begin{proof} To simplify notation, throughout, we assume $A=\emptyset$.
\begin{enumerate}
    \item Everything is clear except the genericity of $x\in M^2$. But by Lemma \ref{L: coherent preservation}, $x$ is coherent (because $t$ is coherent and $x$ is coherent over $t$). So it suffices to show that $x$ is $\mathcal M$-generic in $M^2$. This is well known easy (see  \cite[Lemma 2.40]{CasACF0}).
    \item Let $I$ be the set of $(x,t,u)$ with $x\in C_t\cap D_u$. Fix $(x,t,u)\in I$ generic. Then $(x,t,u)$ is such a tuple (this also follows from  \cite[Lemma 2.40]{CasACF0}).
    \item Let $(x,t,u)$ be a weakly generic $(\mathcal C,\mathcal D)$-intersection. Let $V$ be any neighborhood of $u$. It will suffice to find a strongly generic $(\mathcal C,\mathcal D)$-intersection of the form $(x,t,u')$ with $u'\in V$.
    
    Let ${_x}D$ be the set of $u'$ with $x\in D_{u'}$. Then $u$ is generic in ${_x}D$ over $x$. By the Baire category axiom, there is some $u'\in V$ which is generic in ${_x}D$ over $tx$. It follows easily that $(x,t,u')$ is a strongly generic $(\mathcal C,\mathcal D)$-intersection (again, using \cite[Lemma 2.40]{CasACF0}).
\end{enumerate}
\end{proof}

\subsection{Topological Ramification and Multiple Intersections}

We will also extensively study the related notions of `ramification' and `multiple intersections' in sections 6 and 7. Let us now make these precise.

\begin{definition} Let $f:X\rightarrow Y$ be any function, where $X\subset K^m$ and $Y\subset K^n$. We say that $x_0\in X$ is a \textit{topological ramification point} of $f$ if $(x_0,x_0)$ belongs to the closure of the set of $(x,x')\in X^2$ with $x\neq x'$ and $f(x)=f(x')$.
\end{definition}

\begin{definition} Let $\mathcal Y=\{Y_t:t\in T\}$ and $\mathcal Z=\{Z_t:t\in T\}$ be $\mathcal K(A)$-definable families of subsets of a $\mathcal K(A)$-definable set $X$. Let $(x,t,u)$ be a weakly generic $(\mathcal Y,\mathcal Z)$-intersection over $A$. We call $(x,t,u)$ a \textit{generic multiple $(\mathcal Y,\mathcal Z)$-intersection over $A$} if $(x,t,u)$ is a topological ramification point of the projection $I\rightarrow T\times U$, where $I$ is the set of $(x',t',u')$ with $(x',t')\in Y$ and $(x',u')\in Z$.
     \end{definition}

\subsection{Detection of Generic Multiple Intersections}

The first main goal of section 6 is to show that, if $\mathcal M$ weakly detects closures (Definition \ref{D: weak detection of closures}), then $\mathcal M$ can detect multiple intersections of plane curves in a precise sense. It will be convenient to only consider curves coming from \textit{standard families}. Recall (Definition \ref{D: excellent}):

\begin{definition}\label{D: standard family}
    A \textit{standard family of plane curves over $A$} is an almost faithful $\mathcal M(A)$-definable family of plane curves $\mathcal C=\{C_t:t\in T\}$ such that:
    \begin{enumerate}
        \item $\rk(T)\geq 1$.
        \item $T$ is a generic subset of $M^{\rk(T)}$.
        \item For each $x\in M^2$, the set ${_xC}=\{t:x\in C_t\}$ has rank at most $\rk(T)-1$.
    \end{enumerate}
\end{definition}

\begin{remark} Condition (3) does not matter to us. We include it for consistency with \cite{CasACF0}.
\end{remark}

Standard families always exist in the following sense:

\begin{lemma}\label{L: standard families exist} Let $C$ be a strongly minimal plane curve in $\mathcal M$, with canonical base $c$. Assume that $c$ is coherent of rank at least 1 over $A$. Then there are a standard family of plane curves $\mathcal C=\{C_t:t\in T\}$ over $A$, and a generic $t\in T$ over $A$, with $C$ almost contained in $C_t$.
\end{lemma}
\begin{proof} See Lemmas 2.42 and 3.21 of \cite{CasACF0}.
\end{proof}

We can now state what it means for $\mathcal M$ to detect multiple intersections:

    \begin{definition}\label{D: detects noninjectivities} We say that $\mathcal M$ \textit{detects multiple intersections} if whenever $\mathcal C=\{C_t:t\in T\}$ and $\mathcal D=\{D_u:u\in U\}$ are standard families of plane curves over $A$, and $(x,t,u)$ is a generic multiple $(\mathcal C,\mathcal D)$-intersection over $A$, the parameters $t$ and $u$ are $\mathcal M$-dependent over $A$.
\end{definition}

We now prove, in analogy to  \cite[Theorem 9.1]{CasACF0}:

\begin{theorem}\label{T: detecting ramification} If $\mathcal M$ weakly detects closures, then $\mathcal M$ detects multiple intersections.
\end{theorem}
\begin{proof}
Let $\mathcal C=\{C_t:t\in T\}$ and $\mathcal D=\{D_u:u\in U\}$ be standard families of plane curves over $A$. For ease of notation, we assume $A=\emptyset$. Let $\hat w=(\hat x,\hat t,\hat u)$ be a generic multiple $(\mathcal C,\mathcal D)$-intersection over $\emptyset$. Note that $\hat x$, $\hat t$, and $\hat u$ are generic in their respective powers of $M$, and thus $\hat w$ is coordinate-wise generic; this will guarantee that applications of weak detection of closures below are valid.
		
	Let $I$ be the set of $(x,t,u)$ with $x\in C_t\cap D_u$ -- so $\hat w$ is a topological ramification point of the projection $I\rightarrow T\times U$. Also let $Z\subset T\times U$ be the set of $(t,u)$ such that $C_t\cap D_u$ is infinite, and note that by almost faithfulness the projection $Z\rightarrow T$ is finite-to-one. Finally, let $r_{\mathcal C}=\rk(T)$ and $r_{\mathcal D}=\rk(U)$. So $\rk(\hat t)=r_{\mathcal C}$ and $\rk(\hat u)=r_{\mathcal D}$, and our goal is to show that $\rk(\hat t\hat u)<r_{\mathcal C}+r_{\mathcal D}$. 
		
		\begin{lemma} We may assume that $(\hat t,\hat u)\notin\overline Z$.
		\end{lemma}
		\begin{proof} Assume that $(\hat t,\hat u)\in\overline Z$. Then by weak detection of closures and the fact that $Z\rightarrow T$ is finite-to-one, we have $$\rk(\hat t\hat u)\leq\rk(Z)\leq\rk(T)=r_{\mathcal C}<r_{\mathcal C}+r_{\mathcal D},$$ which proves the theorem in this case.
		\end{proof}
		
		So, assume that $(\hat t,\hat u)\notin\overline Z$. Let $I'$ be the set of $(x,t,u)\in I$ such that $(t,u)\notin Z$. Then by assumption, it follows that $\hat w$ is still a topological ramification point of $I'\rightarrow T\times U$. Let $P$ be the set of $(x,x',t,u)$ such that $x\neq x'$ and $(x,t,u),(x',t,u)\in I'$ -- so the tuple $\hat z=(\hat x,\hat x,\hat t,\hat u)$ belongs to the frontier $\textrm{Fr}(P)$. Moreover, note by definition of $P$, that the projection $P\rightarrow T\times U$ is finite-to-one, which shows that $\rk(P)\leq r_{\mathcal C}+r_{\mathcal D}$. 
		
		Now we would like to apply weak detection of closures to $\hat z\in\operatorname{Fr}(P)$, but we need to `preen' $P$ first to get the relevant projections to be independent on $P$. Let $\pi_1,\pi_2:M^2\rightarrow M$ denote the two projections. Then we define the sets $$P_1=\{(x,x',t,u)\in P:\pi_1(x)\neq\pi_1(x')\},$$ $$P_2=\{(x,x',t,u)\in P:\pi_2(x)\neq\pi_2(x')\}.$$ It is evident from the definition of $P$ that $P=P_1\cup P_2$. Thus, we get $\hat z\in\operatorname{Fr}(P_i)$ for some $i=1,2$. Without loss of generality, we assume that $\hat z\in\operatorname{Fr}(P_1)$. We then define:
		
		\begin{definition} Let $x_1\neq x_1'\in M$. Then $x_1$ and $x_1'$ are \textit{extendable} if the preimage of $(x_1,x_1')$ in $P_1$ under the projection to the first and third coordinates has rank at least $r_{\mathcal C}+r_{\mathcal D}-1$. 
		\end{definition}
		
		Let $E$ be the set of extendable pairs. It is evident from the definition, and the fact that $\rk(P)\leq r_{\mathcal C}+r_{\mathcal D}$, that $\rk(E)\leq 1$. Then, denoting $\hat x$ as the pair $(\hat x_1,\hat x_2)$, we conclude the following:
		
		\begin{lemma}\label{frontier of extendable points} $(\hat x_1,\hat x_1)\notin\overline E$.
		\end{lemma}
		\begin{proof} If not, then $(\hat x_1,\hat x_1)\in\operatorname{Fr}(E)$, since the two coordinates of $(\hat x_1,\hat x_1)$ are equal. But since $\rk(E)\leq 1$, we have $\dim(E)\leq\dim M$; so if $(\hat x_1,\hat x_1)\in\operatorname{Fr}(E)$ then by the strong frontier inequality (see Definition \ref{D: easy axioms}) we get $\dim(\hat x_1,\hat x_1)<\dim M$, which contradicts that $\hat x$ is generic in $M^2$.
		\end{proof}
		
		Finally, let $P'$ be the set of all $(x,x',t,u)\in P_1$ such that $\pi_1(x)$ and $\pi_1(x')$ are not extendable. Note that $P'$ is $\mathcal M$-definable over $\emptyset$, because $P$ is. Moreover, by Lemma \ref{frontier of extendable points}, it follows that $\hat z\in\operatorname{Fr}(P')$. Then we note:
		
		\begin{lemma}\label{P projections independent} If $\rk(P')=r_{\mathcal C}+r_{\mathcal D}$, then the projections to the first and third $M$-coordinates are independent on $P'$.
		\end{lemma}
		\begin{proof} Assume that $\rk(P')=r_{\mathcal C}+r_{\mathcal D}$, and let $(x,x',t,u)\in P'$ be generic. Let $x_1$ and $x_1'$ be the first and third coordinates of $(x,x',t,u)$ -- that is, the first coordinate of each of $x$ and $x'$. Then by definition of $P'$, $x_1$ and $x_1'$ are not extendable. By definition of extendability, we conclude that $\rk(xx'tu/x_1x_1')\leq r_{\mathcal C}+r_{\mathcal D}-2$. But by the choice of $(x,x',t,u)$ we have $\rk(xx'tu)=r_{\mathcal C}+r_{\mathcal D}$. So by additivity, we obtain $\rk(x_1x_1')\geq 2$. Thus $x_1$ and $x_1'$ are $\mathcal M$-independent $\mathcal M$-generics in $M$ over $\emptyset$, which is enough to prove the lemma.	
		\end{proof}
		
		We now apply weak detection of closures to $P'$. Since $P'\subset P$ and $\hat z\in\operatorname{Fr}(P')$, we first get $$\rk(\hat z)\leq\rk(P')\leq\rk(P)\leq r_{\mathcal C}+r_{\mathcal D}.$$ We want this to be a strict inequality between $\hat z$ and $r_{\mathcal C}+r_{\mathcal D}$. But if equality holds then all terms above must be equal, so in particular $\rk(\hat z)=\rk(P')=r_{\mathcal C}+r_{\mathcal D}$. So by Lemma \ref{P projections independent}, the projections of $P'$ to the first and third coordinates are independent. By weak detection of closures, this forces the first and third coordinates of $\hat z$ -- i.e. $\hat x_1$ and $\hat x_1$ -- to be $\mathcal M$-independent over $\emptyset$. But this is clearly not true, since $\hat x_1$ and $\hat x_1$ are equal generics.
	
	So the inequality must be strict, i.e. $\rk(\hat z)<r_{\mathcal C}+r_{\mathcal D}$. But $(\hat t,\hat u)$ is a subtuple of $\hat z$ -- so we also get that $\rk(\hat t\hat u)<r_{\mathcal C}+r_{\mathcal D}$, which proves the theorem.
		\end{proof}

  \subsection{Purity of Ramification and the Dimension of $M$}

  Our next goal is to show that in contexts where \textit{purity of ramification} applies, Theorem \ref{T: detecting ramification} is enough to show that $\dim(M)=1$. This is analogous to section 6 of \cite{CasACF0}. In the abstract setting, the notion of purity of ramification is more difficult to state -- the issue being that we need to identify an appropriate class of maps which could reasonably be expected to have the relevant property (see Remark \ref{R: explaining purity def}). To remedy this, we will use the terminology of weakly generic intersections developed in the previous subsection, with one addition:

\begin{definition}\label{D: codimension} Let $\mathcal Y=\{Y_t:t\in T\}$ and $\mathcal Z=\{Z_u:u\in U\}$ be $\mathcal K(A)$-definable families of subsets of a $\mathcal K(A)$-definable set $X$. Let $(x,t,u)$ be a weakly generic $(\mathcal Y,\mathcal Z)$-intersection over $A$. The \textit{codimension of $(x,t,u)$ over $A$} is $$\operatorname{codim}_A(x,t,u)=\dim(t/Ax)+\dim(u/Ax)-\dim(tu/Ax).$$
\end{definition}

One should think of codimension as a measure of how far a weakly generic intersection is from being strongly generic. With this intuition, we now define:

\begin{definition}\label{D: ramification purity} $(\mathcal K,\tau)$ has \textit{ramification purity} if the following holds: Let $\mathcal Y=\{Y_t:t\in T\}$ and $\mathcal Z=\{Z_u:u\in U\}$ be $\mathcal K(A)$-definable families of subsets of a $\mathcal K(A)$-definable set $X$. Suppose there is a generic multiple $(\mathcal Y,\mathcal Z)$-intersection over $A$ which is strongly approximable over $A$. Then there is a generic multiple $(\mathcal Y,\mathcal Z)$-intersection over $A$ which has codimension at most 1 over $A$.
\end{definition}

\begin{remark}\label{R: explaining purity def} Our originally intended statement of Definition \ref{D: ramification purity} was much simpler and closer to the usual statement: namely, that there is a notion of smoothness $X\mapsto X^S$ on $(\mathcal K,\tau)$ such that, if a generically finite-to-one projection $f:X\rightarrow Y$ (where $\dim(X)=\dim(Y)$) topologically ramifies at a point $x$ with $x\in X^S$ and $y\in Y^S$, then there are codimension 1 topological ramification points of $f$ in any neighborhood of $x$. However, even in the specific example of ACVF, we were unable to prove that the ramification point we wish to consider is actually smooth in its domain. In a sense, one needs to replace `smooth varieties' with `varieties universally homeomorphic to smooth varieties' -- and this has no obvious analog in the abstract setting. This is the reason we needed to use the more complicated definition above.
\end{remark}

We now show:

\begin{theorem}\label{T: higher dimensional case} Assume that $\mathcal M$ detects multiple intersections, and $(\mathcal K,\tau)$ has ramification purity. Then $\dim(M)=1$.
\end{theorem}

\begin{proof} Let $\mathcal C=\{C_t:t\in T\}$ be an $\mathcal M(\emptyset)$-definable excellent family of plane curves. Let $\hat t\in T$ be generic, and let $\hat x\in C_{\hat t}$ be generic over $\hat t$. By Lemma \ref{L: generic intersections exist}, $(\hat x,\hat t,\hat t)$ is a weakly generic $(\mathcal C,\mathcal C)$-intersection which is strongly approximable over $\emptyset$. On the other hand, we also have:

\begin{claim} $(\hat x,\hat t,\hat t)$ is a generic multiple $(\mathcal C,\mathcal C)$-intersection over $\emptyset$.
\end{claim}
\begin{claimproof}
    We need to show that $(\hat x,\hat t,\hat t)$ is a topological ramification point of $I\rightarrow T^2$, where $I$ is the set of $(x,t,u)$ with $x\in C_t\cap C_u$. But $\hat x$ is generic in $C_{\hat t}$ over $\hat t$. By the Baire category axiom, it follows that every neighborhood of $\hat x$ contains infinitely many points of $C_{\hat t}$. In particular, every neighborhood of $(\hat x,\hat t,\hat t)$ contains infinitely many points of $I$ which map to $(\hat t,\hat t)\in T^2$. This is enough to prove the claim.
\end{claimproof}

So $(\hat x,\hat t,\hat t)$ is a generic multiple $(\mathcal C,\mathcal C)$-intersection which is strongly approximable. Moreover, $\mathcal C$ is a standard family over $\emptyset$ (since this is weaker than being excellent). So, since $(\mathcal K,\tau)$ has ramification purity, there is a generic multiple $(\mathcal C,\mathcal C)$-intersection over $\emptyset$ which has codimension at most 1 over $\emptyset$. Call this intersection $(x,t,u)$. Note that, since $\mathcal M$ detects multiple intersections, the parameters $t$ and $u$ are $\mathcal M$-dependent over $\emptyset$.

We now have two bounds on $\dim(xtu)$ -- a lower bound coming from the fact that $\operatorname{codim}_{\emptyset}(x,t,u)\leq 1$, and an upper bound coming from the $\mathcal M$-dependence of $t$ and $u$. The next two claims make these explicit:

\begin{claim} $\dim(xtu)\geq 4\cdot\dim(M)-1$.
\end{claim}
\begin{claimproof} Using that $(x,t)$ and $(x,u)$ are generic in $C$, it is straightforward to see that $\dim(x)=2\cdot\dim(M)$, and $\dim(t/x)=\dim(u/x)=\dim(M)$. Since $\operatorname{codim}_{\emptyset}(x,t,u)=1$, this implies $\dim(tu/x)\geq 2\cdot\dim(M)-1$. Thus, by additivity, $$\dim(xtu)=\dim(x)+\dim(tu/x)\geq 4\cdot\dim(M)-1.$$
\end{claimproof}

\begin{claim} $\dim(xtu)\leq 3\cdot\dim(M)$.
\end{claim}
\begin{claimproof} It is enough to show that $\rk(xtu)\leq 3$. Now since $t$ and $u$ are $\mathcal M$-dependent over $\emptyset$, we have $\rk(tu)\leq 3$. So if $\rk(x/tu)=0$, we are done. Otherwise, we get that $C_t\cap C_u$ is infinite. By almost faithfulness, $u$ is thus $\mathcal M$-algebraic over $t$, which gives that $\rk(xtu)=\rk(xt)=3$.
\end{claimproof}

The theorem follows by combining the two bounds above. Namely, by the two previous claims, we have $4\cdot\dim(M)-1\leq 3\cdot\dim(M)$. Rearranging gives $\dim(M)\leq 1$, and since $M$ is infinite, we get $\dim(M)=1$.
\end{proof}

\section{Definable Slopes and Interpreting a Group}\label{S: ACVF}

\begin{assumption}
    \textbf{Throughout Section \ref{S: ACVF}, in addition to Assumption \ref{A: K and M}, we assume that $\dim(M)=1$.}
\end{assumption}

Our goal in this section is to give conditions under which $\mathcal M$ interprets a strongly minimal group. Essentially, the conditions say that the germ of a curve at a generic point is determined by a sequence of definable approximations, each one of which adds one parameter over the previous one -- and that, moreover, these approximations can be `detected' by $\mathcal M$ in a certain sense. This setting is an abstraction of power series expansions, where the definable approximations are Taylor polynomials. 

Once we have our definable `Taylor expansions', we then build a group in a similar way to previous trichotomy papers (e.g. \cite{HaSu} and \cite{CasACF0}) -- by considering a collection of `slopes' equipped with a composition operation. As in \cite{CasACF0}, we find it most convenient to work with \textit{groupoids} rather than groups, because this allows us to work solely with generic points of curves. In the end, this is still sufficient to recover a group, by Hrushovski's group configuration theorem (see Fact \ref{F: gp con}).

\subsection{Group Configurations}

Our main tool for interpreting a group is Hrushovski's \textit{group configuration theorem}, which is well-known in model theory. Let us now make explicit the version of this theorem we will use:

\begin{definition}
    Let $A$ be a parameter set, and let $$\bar s=(s_{12},s_{13},s_{14},s_{23},s_{24},s_{34})$$ be tuples in $\mathcal M^{\textrm{eq}}$. We say that $\bar s$ is a \textit{rank 1 group configuration in $\mathcal M$ over $A$} if the following hold:
    \begin{enumerate}
        \item Each of the six points in $\bar s$ has rank 1 over $A$.
        \item Any two distinct points in $\bar s$ have rank 2 over $A$.
        \item $\rk(\bar s/A)=3$.
        \item If $i,j,k\in\{1,2,3,4\}$ with $i<j<k$ then $\rk(s_{ij}s_{ik}s_{jk}/A)=2$.
    \end{enumerate}
    \end{definition}

The canonical example is as follows: suppose $\mathcal M$ is a strongly minimal expansion of a group. Let $a,b,c\in M$ be independent generics, and let $s_{12}=a$, $s_{23}=b$, $s_{34}=c$, $s_{13}=ba$, $s_{24}=cb$, and $s_{14}=cba$. Then $\bar s$ is a rank 1 group configuration in $\mathcal M$ over $\emptyset$.

The main fact we need is:

\begin{fact}[Hrushovski (see \cite{Bou})]\label{F: gp con} Suppose there is a rank 1 group configuration $\bar s$ in $\mathcal M$ over $A$ as above. Then there are
\begin{itemize}
    \item $B\supset A$,
    \item an $\mathcal M(B)$-interpretable group $G$ which is strongly minimal as an interpretable set in $\mathcal M$, and
    \item For each $i,j\in\{1,2,3\}$ with $i<j$, an $\mathcal M$-generic element $g_{ij}$ over $B$,
\end{itemize}
such that:
\begin{enumerate}
    \item $g_{12}$ and $g_{23}$ are $\mathcal M$-independent over $B$,
    \item $g_{13}=g_{23}g_{12}$, and
    \item For each $i,j\in\{1,2,3\}$ with $i<j$, $g_{ij}$ and $s_{ij}$ are $\mathcal M$-interalgebraic over $B$.
\end{enumerate}
\end{fact}

\subsection{Definable Slopes} We now define our abstraction of Taylor expansions, as described above: namely, we introduce the notion of $(\mathcal K,\tau)$ having `definable slopes'. This will involve several intermediate notions.

\begin{definition}
    Let $\mathcal{LHG}$ be the category of \textit{local homeomorphisms at generic points} in $(\mathcal K,\tau)$, defined as follows:
    \begin{enumerate}
        \item An object in $\mathcal{LHG}$ is a generic point of $K$ (over $\emptyset$).
        \item A morphism $x\rightarrow y$ is a germ of a local homeomorphism from $K$ to $K$ sending $x$ to $y$: that is, we consider homeomorphisms between neighorhoods of $x$ and $y$, modulo agreement in a neighborhood of $x$.
    \end{enumerate}
\end{definition}
    
\begin{definition}\label{D: invertible arc}
    Let $f:x\rightarrow y$ be a morphism in $\mathcal{LHG}$. We call $f$ a \textit{basic invertible arc} if there are a set $X\subset K^2$ and a parameter set $A$ such that:
    \begin{enumerate}
        \item $X$ is $\mathcal K(A)$-definable.
        \item $\dim(X)=1$, and each projection $X\rightarrow K$ is finite-to-one.
        \item $(x,y)$ is generic in $X$ over $A$.
        \item There are neighborhoods $U$ and $V$ of $x$ and $y$, respectively, so that $X\cap(U\times V)$ is the graph of a homeomorphism of germ $f$.
    \end{enumerate}
\end{definition}

\begin{definition}
    Let $f:x\rightarrow y$ be a morphism in $\mathcal{LHG}$. We call $f$ an \textit{invertible arc} if it is a composition of finitely many basic invertible arcs.
\end{definition}

Suppose $X\subset K^2$ is $\mathcal K(A)$-definable and of dimension 1, with each projection $X\rightarrow K$ finite-to-one, and let $(x,y)\in X$ be generic over $A$. It follows from the generic local homeomorphism property that $X$ determines a basic invertible arc $x\rightarrow y$.

\begin{notation}
    Let $x$, $y$, $X$, and $A$ be as above. We call the basic invertible arc $x\rightarrow y$ determined by $X$ the \textit{arc of} $X$ \textit{at} $(x,y)$, denoted $\alpha_{xy}(X)$.
\end{notation}

It is easy to see that for each generic $x\in K$, the identity morphism $\operatorname{id}:x\rightarrow x$ from $\mathcal{LHG}$ is a basic invertible arc (witnessed by the diagonal in $K^2$). Moreover, by switching the roles of the domain and target, one sees that basic invertible arcs are closed under inverses. It follows that the invertible arcs form the morphisms of a groupoid (in fact a subgroupoid of $\mathcal{LHG}$), with the same objects as $\mathcal{LHG}$.

\begin{definition} We let $\mathcal{IA}_\infty$ be the groupoid of invertible arcs.
\end{definition}

\begin{remark} In contrast, the basic invertible arcs typically do not form a groupoid (that is, not every invertible arc is a basic invertible arc). For example, for generic $x\in K$, the identity is the \textit{only} basic invertible arc $x\rightarrow x$ (as is easily checked from the definition); while other invertible arcs $x\rightarrow x$ can typically be constructed as compositions of basic invertible arcs $x\rightarrow y\rightarrow x$ where $y$ is generic over $x$. In a sense, then, the reason we need compositions is to recover the full collection of morphisms at points $(x,y)$ where $x$ and $y$ are dependent.
\end{remark}

Note that $\mathcal{IA}_\infty$ is typically not a definable object in any reasonable sense. In contrast, the main point of our setting of `Taylor approximations' is to require that $\mathcal{IA}_\infty$ can be \textit{approximated} by definable objects:

\begin{definition}\label{D: definable slopes}
    We say that $(\mathcal K,\tau)$ \textit{has definable slopes} if there are groupoids $\{\mathcal IA_n\}_{n\in\mathbb N}$, and covariant functors $T_{m,n}:\mathcal{IA}_m\rightarrow\mathcal{IA}_n$ for all $n<m\in\mathbb N$, such that:
    \begin{enumerate}
        \item The objects of each $\mathcal{IA}_n$ are the same as in $\mathcal{IA}_\infty$, and all $T_{m,n}$ are the identity on objects.
        \item The $T_{m,n}$ are compatible, and $\mathcal{IA}_{\infty}$ is the inverse limit of the $\mathcal{IA}_n$ for $n<\infty$. More precisely, if $m>n>k$ then $T_{mk}=T_{nk}\circ T_{nm}$, and for any two objects $x$ and $y$, the $T_{m,n}$ turn the set of morphisms $x\rightarrow y$ in $\mathcal{IA}_\infty$ into the inverse limit of the sets of morphisms $x\rightarrow y$ in the $\mathcal{IA}_n$ for $n<\infty$. 
        \item Let $f:x\rightarrow y$ be a morphism in $\mathcal{IA}_n$ for some $n<\infty$. Then $f$ is a tuple in $\mathcal K^{eq}$, and $(x,y)$ is $\mathcal K(f)$-definable.
        \item (Uniform Definability of Slopes) Let $T$ and $X\subset K^2\times T$ be $\mathcal K(A)$-definable. Assume that for each $t\in T$, the fiber $X_t\subset K^2$ has dimension 1 and projects finite-to-one to both copies of $K$. Then for each $n$, the set of $(x,y,t,\alpha)$ such that $(x,y)$ is generic in $X_t$ over $At$ and $T_{\infty,n}(\alpha_{xy}(X_t))=\alpha$, is type-definable in $\mathcal K$ over $A$.
        \item (Definability of Composition and Inverse) If $f:x\rightarrow y$ and $g:y\rightarrow z$ are morphisms in $\mathcal{IA}_n$ for some $n<\infty$, then $g\circ f$ is $\mathcal K(f,g)$-definable, and $f^{-1}$ is $\mathcal K(f)$-definable.
        \item (Coordinatization of Slopes) Suppose $f:x\rightarrow y$ is a morphism in $\mathcal{IA}_n$ for some $n<\infty$. If $n=0$, then $f$ is $\mathcal K$-interdefinable with $xy$. If $n>1$, and $g=T_{n,n-1}(f)$, then $g$ is $\mathcal K(f)$-definable and $\dim(f/g)\leq 1$.
    \end{enumerate}
\end{definition}

\begin{remark} Definition \ref{D: definable slopes}(2) (particularly the clause about inverse limits) is a strong assumption: it says that the germ of a curve at a generic point is determined by its sequence of $n$th approximations (at that point) for every $n$. For instance, this property holds in ACVF and in polynomially bounded o-minimal structures, but we don't know if it holds in arbitrary o-minimal structures (because of the existence of `flat functions' -- non-constant functions with all derivatives vanishing at some point).

On the other hand, we will only use the inverse limit property once (in the proof of Lemma \ref{L: non-algebraic slopes exist}), to guarantee the existence of an infinite family of $n$-slopes for some $n$. We note, then, that the statement of Lemma \ref{L: non-algebraic slopes exist} can be proven in many other settings (including all o-minimal fields) using a version of Sard's Theorem. For this to work, one just needs a differentiable Hausdorff geometric structure whose differential structures is definable in an appropriate sense. 

Finally, we note that one could amend all of Section 7 accordingly in this case -- taking Lemma \ref{L: non-algebraic slopes exist} as an axiom instead of the inverse limit clause of (2) -- and all the arguments would still work. The reason we don't do this is that Lemma \ref{L: non-algebraic slopes exist} rather awkwardly involves the relic $\mathcal M$, and we prefer Definition \ref{D: definable slopes} to be purely a property of $(\mathcal K,\tau)$.
\end{remark}

\begin{assumption}
    \textbf{For the rest of Section 7, we assume that $(\mathcal K,\tau)$ has definable slopes, witnessed by the fixed groupoids and functors $\mathcal{IA}_n$ and $T_{m,n}$}.
\end{assumption}

\begin{remark}
    We think of the functors $T_{m,n}$ as \textit{truncation} maps. Given a morphism $f$ in $\mathcal{IA}_m$ for some $m$, and given some $n<m$, we often call $T_{m,n}(f)$ the $n$th \textit{truncation} of $f$.
\end{remark}

\begin{definition}
    Assume that $(\mathcal K,\tau)$ has definable slopes, witnessed by $\{\mathcal{IA}_n\}$ and $\{T_{m,n}\}$. Let $X\subset K^2$ be $\mathcal K(A)$-definable of dimension 1, with both projections $X\rightarrow K$ finite-to-one, and let $(x,y)\in X$ be generic over $A$. For $n<\infty$, we call $T_{\infty,n}(\alpha_{xy}(X))$ the $n$-\textit{slope} of $X$ at $(x,y)$, denoted $\alpha_{xy}^n(X)$.
\end{definition}

In this language, by an easy compactness argument, the uniform definability of slopes (Definition \ref{D: definable slopes}(4)) has the following simpler consequence:

\begin{lemma}\label{L: slopes definable} Let $X\subset K^2$ be $\mathcal K(A)$-definable of dimension 1, with both projections finite-to-one. Let $(x,y)\in X$ be generic over $A$. Then for each $n<\infty$, $\alpha_{xy}^n(X)\in\operatorname{dcl}_{\mathcal K}(Axy)$.
\end{lemma}

\begin{remark} Definition \ref{D: definable slopes}(4) is stronger than Lemma \ref{L: slopes definable}, because Lemma \ref{L: slopes definable} does not capture the uniformity of the way slopes are defined. This will matter in a couple of places, but usually Lemma \ref{L: slopes definable} will suffice.
\end{remark}

\subsection{Coherent Slopes and Codes}

In this subsection, we give axioms for an encoding system of slopes into $\mathcal M$. The idea is to assign each slope to an interalgebraic tuple in $\mathcal M$, so that the model-theoretic properties of $\mathcal K$ appearing in Definition \ref{D: definable slopes} can be replaced with the analogous properties in $\mathcal M$. For this to have a chance of working, we first need to assign slopes to plane curves in $M$ (rather than sets in $K^2$). We then need to identify a class of \textit{coherent} slopes -- roughly, those appearing on coherently defined plane curves. 

Let us begin by `lifting' slopes to $M$. Since $\dim(M)=1$, and $M$ is $\mathcal K(\emptyset)$-definable, it follows that there is a finite-to-one $\mathcal K(\emptyset)$-definable function $\rho:M\rightarrow K$ (one can construct such a function piecewise by projections). \textbf{For the rest of Section 7, we fix such a function $\rho$. Abusing notation, we also write $\rho:M^n\rightarrow K^n$ for all $n$, applying coordinate-wise}. \\

Suppose $C\subset M^2$ is $\mathcal K(A)$-definable of dimension 1 with both projections $C\rightarrow M$ finite-to-one. Let $(x,y)\in C$ be generic over $A$. Then $\rho(C)\subset K^2$ is $A$-definable and projects finite-to-one in both directions, and $(\rho(x),\rho(y))$ is generic in $\rho(C)$ over $A$. By the generic local homeomorphism property, $\rho$ induces a homeomorphism between open subsets of $M^2$ and $K^2$ near $(x,y)$ and $(\rho(x),\rho(y))$, and also between $C$ and $\rho(C)$ near $(x,y)$ and $(\rho(x),\rho(y))$. This observation allows us, in a neighborhood of $(x,y)$, to think of $C$ as a curve in $K^2$, and subsequently to extract slopes at $(x,y)$:

\begin{definition} Let $C\subset M^2$, $A$, and $(x,y)$ be as in the above paragraph. For each $n$, we define the $n$-slope of $C$ at $(x,y)$, denoted $\alpha_{xy}^n(X)$, to be the $n$-slope of $\rho(C)$ at $(\rho(x),\rho(y))$.
\end{definition}

We now proceed to discuss coherent slopes in $M$. Recall (see Remark \ref{R: nontrivial plane curve}) that a plane curve $C\subset M^2$ is \textit{non-trivial} if each projection $C\rightarrow M$ is finite-to-one. Now we define:

\begin{definition} Let $n$ be a non-negative integer. By a \textit{coherent} $n$-\textit{slope configuration}, we mean a 5-tuple $(x,y,C,c,\alpha)$ such that:
\begin{enumerate}
    \item $C\subset M^2$ is a stationary non-trivial plane curve in the sense of $\mathcal M$ (so in particular $C$ is $\mathcal M$-definable), and $c=\operatorname{Cb}(C)$.
    \item $c$ is coherent, and $(x,y)$ is generic in $C$ over $c$ (so $cxy$ is coherent).
    \item $\alpha_{xy}^n(C)=\alpha$.
\end{enumerate}
\end{definition}

\begin{definition}\label{D: coherent slope} Let $x,y\in M$, and let $\alpha:\rho(x)\rightarrow\rho(y)$ be a morphism in $\mathcal{IA}_n$ for some $n<\infty$.
\begin{enumerate}
    \item We say that $\alpha$ is a \textit{coherent} $n$-\textit{slope at} $(x,y)$ if there is a coherent $n$-slope configuration $(x,y,C,c,\alpha)$. 
    \item If $(x,y,C,c,\alpha)$ is a coherent $n$-slope configuration, we say that $c$ is a \textit{coherent representative} of $\alpha$.
    \item Suppose $\alpha$ is a coherent $n$-slope at $(x,y)$. If $\alpha\in\acl(xy)$, we call $\alpha$ an \textit{algebraic coherent} $n$-\textit{slope at} $(x,y)$. Otherwise, we call $\alpha$ a \textit{non-algebraic coherent} $n$-\textit{slope at} $(x,y)$.
\end{enumerate} 
\end{definition}

We now give various properties of coherent slopes. Lemma \ref{L: interdefinability} is essentially a restatement of Definition \ref{D: definable slopes} (3) and (4), but is quite useful. Roughly, it allows us to perform dimension computations with coherent slopes, while (in the presence of any coherent representative) treating the slopes as elements of $\mathcal M^{\textrm{eq}}$.

\begin{lemma}\label{L: interdefinability} Let $\alpha$ be a coherent $n$-slope at $(x,y)$, and let $c$ be any coherent representative of $\alpha$. 
\begin{enumerate}
    \item $(x,y)\in\acl(\alpha)$.
    \item $\alpha$ is definable over $cxy$ (that is, $\alpha\in\operatorname{dcl}_{\mathcal K}(cxy)$).
    \item $\alpha$ and $xy$ are interalgebraic over $c$.
\end{enumerate}
\end{lemma}
\begin{proof} Clause (2) follows from Lemma \ref{L: slopes definable} and the fact that $\rho$ is $\mathcal K(\emptyset)$-definable. Clause (3) follows from (1) and (2). For (1), since $\alpha$ is a morphism from $\rho(x)$ to $\rho(y)$, Definition \ref{D: definable slopes}(3) gives that $\rho(x)$ and $\rho(y)$ are $\mathcal K(\alpha)$-definable. Meanwhile, since $\rho$ is finite-to-one, $x$ and $y$ are algebraic over $(\rho(x),\rho(y))$.
\end{proof}

It will also be convenient to know that, when working with non-algebraic slopes, we can upgrade Definition \ref{D: coherent slope} to include the genericity of $(x,y)$ in $M^2$:

\begin{lemma}\label{L: the point is generic} Let $\alpha$ be a c coherent $n$-slope at $(x,y)$. Then:
\begin{enumerate}
    \item If $\alpha\notin\acl(x,y)$, then $(x,y)$ is generic in $M^2$.
    \item If $(x,y)$ is generic in $M^2$, and $c$ is any coherent representative of $\alpha$, then $\dim(c/xy)=\dim(c)-1$.
\end{enumerate}
\end{lemma}
\begin{proof} For (1), suppose toward a contradiction that $(x,y)$ is not generic in $M^2$. Then by coherence, $\rk(xy)\leq 1$. But $\rk(xy/c)=1$ by assumption, so $\rk(xy/c)=\rk(xy)=1$. So $p=\operatorname{tp}_{\mathcal M}(xy/c)$ does not fork over $\emptyset$. This implies that $c=\operatorname{Cb}(p)\in\acl_{\mathcal M}(\emptyset)$. Then by Lemma \ref{L: interdefinability}, $\alpha\in\acl(cxy)\subset\acl(xy)$, so $\alpha$ is algebraic.

For (2), let $c$ be any coherent representative. Then $\dim(xy/c)=1$, so $\dim(cxy)=\dim(c)+1$. On the other hand, $\dim(xy)=2$ by assumption. So $$\dim(c/xy)=(\dim(c)+1)-2=\dim(c)=1.$$
\end{proof}

Later, under further assumptions, we will interpret a strongly minimal group in $\mathcal M$ by finding codes for a one-dimensional family of coherent slopes. Lemmas \ref{L: non-algebraic slopes exist} and \ref{L: slopes 1 dimensional} guarantee that such one-dimensional families exist.

\begin{lemma}\label{L: non-algebraic slopes exist} There is a non-algebraic coherent $n$-slope for some $n$.
\end{lemma}
\begin{proof} Fix any stationary non-trivial plane curve $C$ with coherent canonical base $c$, and fix a generic element $(x,y)\in C$ over $c$. So for each $n<\infty$, $\alpha_{xy}^n(C)$ is a coherent $n$-slope at $(x,y)$. 

Since the category $\mathcal{IA}_{\infty}$ is the inverse limit of the $\mathcal{IA}_n$ for $n<\infty$ (see Definition \ref{D: definable slopes}(2)), the basic invertible arc $\alpha_{\rho(x)\rho(y)}(\rho(C))$ (i.e. the germ of $\rho(C)$ at $(\rho(x),\rho(y))$) is determined by the truncations $\alpha_{xy}^n(C)$ for $n<\infty$. Since $\rho$ induces a local homeomorphism $C\rightarrow\rho(C)$ near $(x,y)$, it follows that these truncations, together with $(x,y)$, determine the germ of $C$ at $(x,y)$.

Let $A$ be the tuple of all $\alpha_{xy}^n(C)$ for $n<\infty$. It now follows that $c=\operatorname{Cb}(c)$ is determined by $A$ -- or more precisely, $c\in\operatorname{dcl}_{\mathcal K}(A)$. Let us explain. By $\aleph_1$-saturation, it is enough to show that for any $c'\models\tp_{\mathcal K}(c/A)$, we have $c'=c$. So fix such $c'$. Then $c'=\operatorname{Cb}(C')$ for some stationary non-trivial plane curve $C'$, also containing $(x,y)$ as a generic point. Now by uniform definability of slopes, and since $\tp_{\mathcal K}(c'/A)=\tp_{\mathcal K}(c/A)$, it follows that $\alpha_{xy}^n(C')=\alpha_{xy}^n(C)$ for all $n<\infty$. Then by the discussion from the above paragraph, $C'$ agrees with $C$ in a neighborhood of $(x,y)$. But by the Baire category axiom (see Remark \ref{R: non-discrete}), every neighborhood of $(x,y)$ in $C$ is infinite; so in fact $C'$ and $C$ have infinite intersection, and thus by strong minimality they are almost equal. Thus, they have the same canonical base, i.e. $c'=c$. 

Now assume that all coherent slopes are algebraic. Then, in the notation above, we get $A\subset\acl(xy)$, and thus (by the previous paragraph) $c\in\acl(xy)$. In particular, $\dim(c)\leq 2$. But, by non-local modularity, $\dim(c)$ can be arbitrarily large, which is enough to prove the lemma.
\end{proof}

\begin{notation} For the rest of section 7, we let $n_0$ be the smallest $n$ such that there is a non-algebraic coherent $n$-slope.
\end{notation}

By Definition \ref{D: definable slopes}(6), note that every coherent 0-slope is algebraic. So necessarily $n_0\geq 1$, and thus every coherent $n_0$-slope has an $n_0-1$-th truncation. We use this freely below.

\begin{remark} One could also define `non-algebraic coherent $n$-slopes' relative to a parameter set $A$ (so coherence is interpreted over $A$, and non-algebraicity at $(x,y)$ means the slope is not algebraic over $Axy$). This would, a priori, change the way that $n_0$ is defined. We note, however, that the value of $n_0$ does not ultimately depend on parameters (which the reader can check if desired). Thus, we find it simplest to stick to the parameter-free version presented above.
\end{remark}

\begin{lemma}\label{L: slopes 1 dimensional}
Suppose $\alpha$ is a non-algebraic coherent $n_0$-slope at $(x,y)$.
\begin{enumerate}
    \item $\dim(\alpha/xy)=1$. 
    \item If $c$ is any coherent representative of $\alpha$, then $\dim(c/\alpha)=\dim(c/xy)-1$.
\end{enumerate}
\end{lemma}
\begin{proof} Let $\beta$ be the $(n_0-1)$-th truncation of $\alpha$. By the minimality of $n_0$, $\beta\in\acl(xy)$. By Definition \ref{D: definable slopes}(6), $\dim(\alpha/\beta)\leq 1$. Thus $\dim(\alpha/xy)\leq 1$, and since $\alpha$ is non-algebraic, 
$\dim(\alpha/xy)=1$, proving (1).

For (2), let $c$ be any coherent representative of $\alpha$. Repeatedly using Lemma \ref{L: interdefinability} and (1), we compute: $$\dim(c/\alpha)=\dim(c/\alpha xy)=\dim(c\alpha/xy)-\dim(\alpha/xy)=\dim(c/xy)-1.$$
\end{proof}

Finally, we now end this subsection with our definition of codes of slopes in $\mathcal M$. This is the key notion allowing us to recover composition of slopes in $\mathcal M$.

\begin{definition}\label{D: code} Let $\alpha$ be a coherent $n$-slope at $(x,y)$. A tuple  $s$  from $\mathcal M^{\textrm{eq}}$ is a \textit{code} of $\alpha$ if the following hold:
\begin{enumerate}
    \item $(x,y)\in\acl_{\mathcal M}(s)$.
    \item For each coherent representative $c$ of $\alpha$, $s\in\acl_{\mathcal M}(cxy)$.
    \item $\alpha$ and $s$ are $\mathcal K$-interalgebraic.
\end{enumerate}
\end{definition}

Note that codes might not exist -- that is, if $\alpha$ is a coherent $n$-slope at $(x,y)$, there is no reason in general to expect some $s\in\mathcal M^{eq}$ to satisfy (1)-(3) in Definition \ref{D: code}. Moving forward, one of our main jobs will be to find suitable topological conditions guaranteeing that many slopes do have codes. 

\subsection{Detecting Composition when Codes Exist} Definition \ref{D: definable slopes}(5) says that slope composition is definable in an appropriate sense. Our next goal is to transfer this property into $\mathcal M$, for `independent' composition of coherent slopes that have codes. The idea is that the composition of slopes is respected by the $\mathcal M$-definable composition of plane curves. 

The following notion will be useful in the next two sections:

\begin{definition} Let $a_1,...,a_n,b_1,...,b_n\in\mathcal K^{\textrm{eq}}$. We say that $a_1,...,a_n$ are \textit{maximally independent} over $b_1,...,b_n$ if 
\[\dim(a_1...a_n/b_1...b_n)=\dim(a_1/b_1)+...+\dim(a_n/b_n).\]
\end{definition}

Note that if $b_1=...=b_n=b$, this is the same as saying that $a_1,...,a_n$ are independent over $b$.

The following transitivity property follows  easily from additivity of dimension:

\begin{lemma}\label{L: maximal independence additivity} $a_1b_1,...,a_nb_n$ are maximally independent over $c_1,...,c_n$ if and only if $a_1,...,a_n$ are maximally independent over $c_1,...,c_n$ and $b_1,...,b_n$ are maximally independent over $a_1c_1...a_nc_n$.
\end{lemma}

\begin{proposition}\label{P: composition definable} Let $\alpha_1$ be a coherent $n$-slope at $(x_0,y_0)$, and let $\alpha_2$ be a coherent $n$-slope at $(y_0,z_0)$. Assume that $x_0$ and $z_0$ are independent over $y_0$, and $\alpha_1,\alpha_2$ are maximally independent over $(x_0,y_0),(y_0,z_0)$. Then $\alpha_3=\alpha_2\circ\alpha_1$ is a coherent $n$-slope at $(x_0,z_0)$. Moreover, if $s_i$ is a code of $\alpha_i$ for $i=1,2,3$, then $s_3\in\acl_{\mathcal M}(s_1s_2)$.
\end{proposition}
\begin{proof} Let $c_1$ and $c_2$ be coherent representatives of $\alpha_1$ and $\alpha_2$, respectively. We may assume that $c_1,c_2$ are maximally independent over $\alpha_1,\alpha_2$. Note also that $(x_0,y_0),(y_0,z_0)$ are maximally independent over $y_0,y_0$ by assumption. By repeated applications of Lemma \ref{L: maximal independence additivity}, it then follows that $c_1,c_2$ are maximally independent over $y_0,y_0$. Thus, $c_1$ and $c_2$ are independent over $y_0$. But each $c_i$ is independent from $y_0$ by assumption. It then follows from additivity of dimension that $c_1$, $c_2$, and $y_0$ are independent over $\emptyset$. In particular, $c_1c_2y_0$ is coherent, and $\dim(y_0/c_1c_2)=1$. Clearly $x_0$, $y_0$, and $z_0$ are pairwise $\mathcal M$-interalgebraic over $c_1c_2$. Thus, we have:

\begin{claim}\label{point on composition is generic} $c_1c_2x_0y_0z_0$ is coherent, and $\dim(x_0z_0/c_1c_2)=1$.
\end{claim}

Now let $C_1$ and $C_2$ be such that $(x_0,y_0,C_1,c_1,\alpha_1)$ and $(y_0,z_0,C_2,c_2,\alpha_2)$ are coherent $n$-slope configurations. We may assume that each $C_i$ is $\mathcal M$-definable over $c_i$. Let $C_3=C_2\circ C_1$ -- that is, the set of $(x,z)$ such that for some $y$ we have $(x,y)\in C_1$ and $(y,z)\in C_2$. Then $C$ is a non-trivial plane curve which is $\mathcal M(c_1c_2)$-definable and (by Claim \ref{point on composition is generic}) contains $(x_0,z_0)$ as a generic element over $c_1c_2$.

\begin{claim} $\alpha_{x_0z_0}^n(C_3)=\alpha_3$.
\end{claim}
\begin{claimproof} Let $q_1$ stand for $(x_0,y_0)$, $q_2$ stand for $(y_0,z_0)$, and $q_3$ stand for $(x_0,z_0)$. By the generic local homeomorphism property, for each $i=1,2,3$, the points of $C_i$ in a neighborhood of $q_i$ form the graph of a local homeomorphism $f_i$. Shrinking these neighborhoods if necessary, it is clear that $f_3=f_2\circ f_1$.

Now, for each $i$, let $\rho(f_i)$ be the image of $f_i$ under $\rho$, viewed as a local function on $K$ at $\rho(q_i)$. It follows by definition that each $\rho(f_i)$ is an invertible arc, with $n$th truncation $\alpha_{q_i}^n(C_i)$. But since $\rho$ is a local homeomorphism $M\rightarrow K$ near each of $x_0$, $y_0$, and $z_0$, it also follows that $\rho(f_3)=\rho(f_2)\circ\rho(f_1)$. So since the truncation maps are functors, we get $\alpha_{q_3}^n(C_3)=\alpha_{q_2}^n(C_2)\circ\alpha_{q_1}^n(C_1)$. Equivalently, $\alpha_{x_0z_0}^n(C_3)=\alpha_2\circ\alpha_1=\alpha_3$.
\end{claimproof}

Now since $(x_0,z_0)$ is generic in $C_3$, there is a strongly minimal component $C$ of $C_3$, with canonical base $c_3$, such that $(x_0,z_0)$ is generic in $C$ over $c_3$. Then $c_3\in\acl_{\mathcal M}(c_1c_2)$, so $c_3$ is coherent. Moreover, by the strong frontier inequality, $C$ and $C_3$ agree in a neighborhood of $(x_0,z_0)$, which shows that $\alpha_{x_0,z_0}^n(C)=\alpha_{x_0,z_0}^n(C_3)=\alpha_3$. It follows that $(x_0,z_0,C,c_3,\alpha_3)$ is a coherent $n$-slope configuration. So $\alpha_3$ is coherent, and $c_3$ is a coherent representative of $\alpha_3$.

Now assume that $s_i$ is a code of $\alpha_i$ for each $i$. By the definition of codes, we have $s_i\in\acl_{\mathcal M}(c_ix_0y_0z_0)$ for each $i$. Since $c_3\in\acl_{\mathcal M}(c_1c_2)$, this implies $s_1s_2s_3\in\acl_{\mathcal M}(c_1c_2x_0y_0z_0)$. By Claim \ref{point on composition is generic}, $c_1c_2x_0y_0z_0$ is coherent, thus so is $s_1s_2s_3$. But by the definability of composition (Definition \ref{D: definable slopes}(5)), $\alpha_3\in\acl(\alpha_1\alpha_2)$. So by interalgebraicity with the $s_i$ (Definition \ref{D: code})(3), $s_3\in\acl(s_1s_2)$, and thus by coherence, $s_3\in\acl_{\mathcal M}(s_1s_2)$.
    \end{proof}

\begin{remark} Note that the last clause of Proposition \ref{P: composition definable} requires that all three $\alpha_i$ have codes. That is, we do not claim that if $\alpha_1$ and $\alpha_2$ have codes, so does their composition.
\end{remark}

\subsection{Interpreting a Group when Codes Exist} Recall that $n_0$ is the smallest $n$ such that there is a non-algebraic coherent $n$-slope. We now use Proposition \ref{P: composition definable} to interpret a group in $\mathcal M$ assuming enough coherent $n_0$-slopes have codes. The statement of Proposition \ref{P: group from codes} below is somewhat complicated. The reason is that later on, in ACVF, we will need to analyze the specific structure of the interpreted group -- a task which will use the extra data in the proposition. However, we stress the main point: if all coherent $n_0$-slopes have codes, then $\mathcal M$ interprets a strongly minimal group.

\begin{proposition}\label{P: group from codes} Suppose every coherent $n_0$-slope has a code. Then there are
\begin{itemize}
    \item A parameter set $B$,
    \item An $\mathcal M(B)$-interpretable group $(G,\cdot)$ which is strongly minimal as in interpretable set in $\mathcal M$,
    \item A generic element $y\in K$ over $\emptyset$,
    \item For each $i,j\in\{1,2,3\}$ with $i<j$, an element $g_{ij}\in G$, and
    \item For each $i,j\in\{1,2,3\}$ with $i<j$, a morphism $f_{ij}:y\rightarrow y$ in $\mathcal{IA}_{n_0}$,
\end{itemize}
such that:
\begin{enumerate}
    \item $y$ is $\mathcal K(B)$-definable (that is, $y\in\operatorname{dcl}_{\mathcal K}(B)$).
    \item The $(n_0-1)$-th truncation of each $f_{ij}$ is the identity at $y$.
    \item $g_{12}$ and $g_{23}$ are independent generics in $G$ over $B$.
    \item $g_{13}=g_{23}\cdot g_{12}$ and $f_{13}=f_{23}\circ f_{12}$.
    \item For each $i,j\in\{1,2,3\}$ with $i<j$, $g_{ij}$ and $f_{ij}$ are interalgebraic over $B$.
\end{enumerate}
\end{proposition}

\begin{proof} Our strategy is as follows. First, we use the definability of composition and inverse of slopes to find a `group configuration' in the sense of dimension in $\mathcal K$, whose elements are non-algebraic $n_0$-slopes. Then we use Proposition \ref{P: composition definable} to check that the slopes in the configuration are coherent, and subsequently extract codes. Finally, we check that the tuple of codes in the entire configuration is coherent, and conclude that the codes form a group configuration in the sense of rank in $\mathcal M$.

Let us proceed. We first want to find an appropriate collection of points and slopes to form a group configuration. The first two lemmas below accomplish this.

\begin{lemma}\label{L: coherent slopes invertible} Suppose there is a non-algebraic coherent $n_0$-slope at $(x_0,y_0)$ for some $(x_0,y_0)\in M^2$. Then there is a non-algebraic coherent $n_0$-slope at $(y_0,x_0)$.
\end{lemma}
\begin{proof} Suppose that $(x_0,y_0,C,c,\alpha)$ is a non-algebraic coherent $n_0$-slope configuration. It is then easy to check that $(y_0,x_0,C^{-1},c,\alpha^{-1})$ is a non-algebraic coherent $n_0$-slope configuration, where $C^{-1}$ is the set of $(y,x)$ with $(x,y)\in C$.
\end{proof}

\begin{lemma}\label{L: 4 points with 6 coherent slopes} There are independent generic elements $x_1,x_2,x_3,x_4\in M$ such that each of $(x_1,x_2)$, $(x_2,x_3)$, and $(x_3,x_4)$ admits a non-algebraic coherent $n_0$-slope. 
\end{lemma}
\begin{proof} Let $(x_1,x_2)$ be such that there is a non-algebraic coherent $n_0$-slope at $(x_1,x_2)$. By Lemma \ref{L: the point is generic}, $(x_1,x_2)$ is generic in $M^2$. Let $x_3$ be a  realization of $\operatorname{tp}_{\mathcal K}(x_1/x_2)$ independent over $x_1$, and let $x_4$ be a  realization of $\operatorname{tp}_{\mathcal K}(x_2/x_3)$ independent over $x_1x_2$. So $x_1,x_2,x_3,x_4$ are independent generics in $M$, and all of $(x_1,x_2)$, $(x_3,x_2)$, and $(x_3,x_4)$ realize the same type in $\mathcal K$. Thus, each of $(x_1,x_2)$, $(x_3,x_2)$, and $(x_3,x_4)$ admits a non-algebraic coherent $n_0$-slope, and by Lemma \ref{L: coherent slopes invertible}, so does $(x_2,x_3)$.
\end{proof}

\begin{notation}
    For the rest of the proof of Proposition \ref{P: group from codes}, we adopt the following:
    \begin{itemize}
        \item We fix $x_1,x_2,x_3,x_4$ as provided by Lemma \ref{L: 4 points with 6 coherent slopes}. 
        \item We fix non-algebraic coherent $n_0$-slopes $\alpha_{12}$ at $(x_1,x_2)$, $\alpha_{23}$ at $(x_2,x_3)$, and $\alpha_{34}$ at $(x_3,x_4)$, chosen so that $\alpha_{12},\alpha_{23},\alpha_{34}$ are maximally independent over $x_1x_2,x_2x_3,x_3x_4$.
        \item Finally, we let $\alpha_{13}=\alpha_{23}\circ\alpha_{12}$, $\alpha_{24}=\alpha_{34}\circ\alpha_{23}$, and $\alpha_{14}=\alpha_{34}\circ\alpha_{23}\circ\alpha_{12}$, and we denote the tuple $(\alpha_{12},\alpha_{23},\alpha_{34},\alpha_{13},\alpha_{24},\alpha_{14})$ by $\overline\alpha$.
    \end{itemize}
\end{notation}

Lemma \ref{L: gp con in K} follows easily from Lemmas \ref{L: slopes 1 dimensional} and the definability of composition and inverse (Definition \ref{D: definable slopes}(5)). We leave the details to the reader.

\begin{lemma}\label{L: gp con in K} $\overline\alpha$ forms a `one-dimensional group configuration in $\mathcal K$ over $x_1x_2x_3x_4$'. That is:
\begin{enumerate}
    \item Each point in $\overline\alpha$ has dimension 1 over $x_1x_2x_3x_4$.
    \item Any two distinct points in $\overline\alpha$ have dimension 2 over $x_1x_2x_3x_4$.
    \item $\dim(\overline\alpha/x_1x_2x_3x_4)=3$.
    \item If $i,j,k\in\{1,2,3,4\}$ with $i<j<k$, then $\dim(\alpha_{ij}\alpha_{jk}\alpha_{ik}/x_1x_2x_3x_4)=2$.
\end{enumerate}
\end{lemma}

By Lemmas \ref{L: slopes 1 dimensional} and \ref{L: gp con in K}, any two elements $\alpha_{ij},\alpha_{kl}\in\overline{\alpha}$ are maximally independent over $x_ix_j,x_kx_l$. This allows us to apply Proposition \ref{P: composition definable} to the elements of $\overline\alpha$. We first conclude:

\begin{lemma}\label{L: slopes in gp con are coherent} Each of the six elements of $\overline\alpha$ is a coherent $n_0$-slope.
\end{lemma}
\begin{proof} $\alpha_{12}$, $\alpha_{23}$, and $\alpha_{14}$ are coherent by the way they are chosen. Now repeatedly apply Proposition \ref{P: composition definable}. We get first that each of $\alpha_{13}$ and $\alpha_{24}$ is a composition of two maximally independent coherent slopes, so is coherent. Finally, $\alpha_{14}$ is now coherent by the same reasoning, by writing it as $\alpha_{34}\circ\alpha_{13}$. 
\end{proof}

By Lemma \ref{L: slopes in gp con are coherent} and the assumption of the proposition, each $\alpha_{ij}$ has a code.

\begin{notation}
    We fix a code $s_{ij}$ for each $\alpha_{ij}$, and denote the tuple of $s_{ij}$ by $\overline s$. 
\end{notation}

Arguing similarly to Lemma \ref{L: slopes in gp con are coherent}, we now conclude:

\begin{lemma}\label{L: all slopes are coherent} $\overline s$ is coherent.
\end{lemma}
\begin{proof}
    By definition of codes, each $s_{ij}$ is interalgebraic with $\alpha_{ij}$. It follows that $s_{12},s_{23},s_{34}$ are maximally independent over $x_1x_2,x_2x_3,x_3x_4$. It then follows easily that $s_{12}s_{23}s_{34}$ is coherent over $x_1x_2x_3x_4$, and since $x_1x_2x_3x_4$ is coherent by construction, we get that $s_{12}s_{23}s_{34}$ is coherent. But by repeated instances of Proposition \ref{P: composition definable} (exactly as in Lemma \ref{L: slopes in gp con are coherent}), we have $\overline s\in\acl_{\mathcal M}(s_{12}s_{23}s_{34})$. So $\overline s$ is coherent.
\end{proof}

Now, since each $s_{ij}$ is interalgebraic with $\alpha_{ij}$, Lemma \ref{L: gp con in K} holds with each $\overline\alpha$ replaced with $\overline s$. Since $\overline s$ is coherent, all of the statements in Lemma \ref{L: gp con in K} hold with rank instead of dimension. Thus, $\overline s$ is a rank 1 group configuration in $\mathcal M$ over $x_1x_2x_3x_4$, and thus, by Fact \ref{F: gp con}, $\mathcal M$ interprets a strongly minimal group. 

If we only wanted to find a group, we could stop here. However, let us now apply Fact \ref{F: gp con} more carefully in order to verify the extra conditions of Proposition \ref{P: group from codes}. First, we define:

\begin{itemize}
    \item Let $\overline\beta\overline t$ be an independent realization of $\operatorname{tp}_{\mathcal K}(\overline\alpha\overline s/\acl(x_1x_2x_3x_4))$ over $\overline\alpha\overline s$ (by this we mean that for $i<j$ we have distinguished tuples $\beta_{ij}$ and $t_{ij}$).
    \item Let $A=\overline tx_1x_2x_3x_4$.
    \item Let $y=\rho(x_4)$.
    \item For $i,j\in\{1,2,3\}$ with $i<j$, let $f_{ij}=\beta_{j4}\circ\alpha_{ij}\circ\beta_{i4}^{-1}$.
\end{itemize}
We now check:

\begin{lemma}\label{L: power series at point facts}
\begin{enumerate}
    \item $y$ is generic in $K$ over $\emptyset$.
    \item $y$ is $\mathcal K(A)$-definable.
    \item Each $\alpha_{ij}$ has the same $(n_0-1)$-th truncation as $\beta_{ij}$
    \item Each $f_{ij}$ is a morphism from $y$ to $y$ in $\mathcal{IA}_{n_0}$.
    \item The $(n_0-1)$-th truncation of each $f_{ij}$ is the identity at $y$.
    \item Each $f_{ij}$ is interalgebraic with $s_{ij}$ over $A$.
    \item $f_{13}=f_{23}\circ f_{12}$.
    \item $\overline s$ is a rank 1 group configuration in $\mathcal M$ over $A$.
\end{enumerate}
\end{lemma}
\begin{proof}
    \begin{enumerate}
        \item Clear, because $y$ and $x_4$ are interalgebraic.
        \item Clear because $\rho$-is $\0$-definable. 
        \item By the minimality of $n_0$, the $(n_0-1)$-th truncation of $\alpha_{ij}$ is contained in $\acl(x_1x_2x_3x_4)$, and $\alpha_{ij}$ and $\beta_{ij}$ realize the same type over $\acl(x_1x_2x_3x_4)$.
        \item Clear, because $f_{ij}$ is the composition of $\beta_{i4}^{-1}:\rho(x_4)\rightarrow\rho(x_i)$, $\alpha_{ij}:\rho(x_i)\rightarrow\rho(x_j)$, and $\beta_{j4}:\rho(x_j)\rightarrow\rho(x_4)$.
        \item By (3), $\beta_{i4}$ and $\beta_{j4}$ have the same $(n_0-1)$-th truncations as $\alpha_{i4}$ and $\alpha_{j4}$, respectively. So, the $(n_0-1)$-th truncation of $f_{ij}=\beta_{j4}\circ\alpha_{ij}\circ\beta_{i4}^{-1}$ is the same as that of $$\alpha_{j4}\circ\alpha_{ij}\circ\alpha_{i4}^{-1}=\alpha_{j4}\circ\alpha_{ij}\circ(\alpha_{j4}\circ\alpha_{ij})^{-1}$$ $$=\alpha_{j4}\circ\alpha_{ij}\circ\alpha_{ij}^{-1}\alpha_{j4}^{-1}=\mathrm{id}:y\rightarrow y.$$
        \item Clearly, $f_{ij}$ is interdefinable with $\alpha_{ij}$ over $\overline\beta$. Since $\overline\beta\overline t$ and $\overline\alpha\overline s$ realize the same $\mathcal K$-type over $x_1x_2x_3x_4$ (and by definition of codes), it follows that $\overline\beta$ is interalgebraic with $\overline tx_1x_2x_3x_4=A$. Thus $f_{ij}$ is interalgebraic with $\alpha_{ij}$ over $A$. Finally, again by definition of codes, $\alpha_{ij}$ is interalgebraic with $s_{ij}$ over $A$. Thus, $f_{ij}$ is interalgebraic with $s_{ij}$ over $A$.
        \item We have $$f_{23}\circ f_{12}=(\beta_{34}\circ\alpha_{23}\circ\beta_{24}^{-1})\circ(\beta_{24}\circ\alpha_{12}\circ\beta_{14}^{-1})$$ $$=\beta_{34}\circ(\alpha_{23}\circ\alpha_{12})\circ\beta_{14}^{-1}=\beta_{34}\circ\alpha_{13}\circ\beta_{14}^{-1}=f_{13}.$$
        \item By construction (see Lemma \ref{L: coherent preservation}(4)), $A\overline s$ is coherent, and $A$ is $\mathcal M$-independent from $\overline s$ over $x_1x_2x_3x_4$. So since $\overline s$ is a rank 1 group configuration in $\mathcal M$ over $x_1x_2x_3x_4$, the same holds over $A$.
    \end{enumerate}
\end{proof}

Now using Lemma \ref{L: power series at point facts}(8), we can apply Fact \ref{F: gp con} to $\overline s$ and $A$. We obtain:

\begin{lemma}\label{L: gp con app} There are
\begin{itemize}
    \item A parameter set $B\supset A$,
    \item An $\mathcal M(B)$-interpretable group $G$ which is strongly minimal as an interpretable set in $\mathcal M$, and
    \item For each $i,j\in\{1,2,3\}$ with $i<j$, an $\mathcal M$-generic element $g_{ij}\in G$ over $B$,
\end{itemize}
such that:
\begin{enumerate}
    \item $g_{12}$ and $g_{23}$ are $\mathcal M$-independent over $B$,
    \item $g_{13}=g_{23}\cdot g_{12}$, and
    \item For each $i,j\in\{1,2,3\}$ with $i<j$, $g_{ij}$ and $s_{ij}$ are $\mathcal M$-interalgebraic over $B$.
\end{enumerate}
\end{lemma}

We can thus add to our above list of conclusions:
\begin{lemma}\label{L: group slope interalg}
Each $g_{ij}$ is interalgebraic with $f_{ij}$ over $B$.
\end{lemma}
\begin{proof}
     Note that $s_{ij}$ is interalgebraic with each of them over $B$ (by Lemma \ref{L: gp con app}(3) and Lemma \ref{L: power series at point facts}(6)).
\end{proof}

At this point, Lemmas \ref{L: power series at point facts} and \ref{L: group slope interalg} cover all of the requirements of Proposition \ref{P: group from codes}, with one exception: we need $g_{12}$ and $g_{23}$ to be independent generics in $G$ over $B$, and we only know they are $\mathcal M$-independent $\mathcal M$-generics. To remedy this, note that all of our initial assumptions on $B$ and the $g_{ij}$ (i.e. those in Lemma \ref{L: gp con app}) are expressed solely in terms of the structure $\mathcal M$. So Lemmas \ref{L: gp con app} and \ref{L: group slope interalg} will still hold after replacing $Bg_{12}g_{23}g_{13}$ with any tuple realizing the same $\mathcal M$-type over $A$. Now by Lemma \ref{L: coherent preservation}(6), there is one such tuple which is coherent over $A$ (in case $B$ is infinite, we interpret coherence as saying that all finite subsets are coherent; in this case, enumerate $B$ and apply Lemma \ref{L: coherent preservation}(6) inductively). Then, by coherence, we can transfer $\mathcal M$-genericity to full genericity (i.e. Lemma \ref{L: coherent preservation}(1)). In particular, $g_{12}$ and $g_{23}$ are now independent generics in $G$ over $B$, as desired.
\end{proof}

\subsection{Detecting Tangency and Finding Codes}

We now turn to the task of coding $n_0$-slopes. The following is an analog of the notion of `detecting generic non-transversalities' from \cite{CasACF0}:

\begin{definition}\label{D: detects tangency} $\mathcal M$ \textit{detects tangency} if whenever $(x,y)\in M^2$ is generic, and $\alpha$ is a non-algebraic coherent $n$-slope at $(x,y)$ for some $n$, then any two coherent representatives of $\alpha$ are $\mathcal M$-dependent over $(x,y)$.
\end{definition}

We will show that if $\mathcal M$ detects tangency, every coherent $n_0$-slope has a code, and thus by Proposition \ref{P: group from codes}, $\mathcal M$-interprets a strongly minimal group. In \cite{CasACF0}, one has a clear construction of codes: indeed, in ACF$_0$, a rank 2 family of plane curves realizes every non-algebraic coherent $n_0$-slope finitely many times, and one can code a slope using its finite set of occurrences in any such family. However, in the current setting, there is no reason that a non-algebraic coherent slope must occur in a rank 2 family. So we might have an infinite set of occurrences in a larger family, any two of which are $\mathcal M$-dependent by assumption, and we want to `code' this infinite set. The proper tool for doing such a thing is canonical bases. As it turns out, we will code each slope using the canonical base of the dependence between two independent occurrences of that slope in a family.

\begin{proposition}\label{P: codes exist} If $\mathcal M$ detects tangency, then every coherent $n_0$-slope has a code.
\end{proposition}
\begin{proof} The idea is to use the detection of tangency to $\mathcal M$-definably organize the curves in a family by slope, and then take the `equivalence classes' (i.e. canonical bases) to be the codes. For this to work, we need to know that the $\mathcal M$-dependence between coherent representatives in Definition \ref{D: detects tangency} is not `coarser' than the tangency relation, so that the canonical bases we take are genuinely capturing curves up to tangency. This will work because we are only considering one-dimensional families of slopes (i.e. by the minimality of $n_0$). The precise statements we need are given in Lemmas \ref{L: tangent implies coherent} and \ref{L: two curves dont decrease rank} below:

\begin{lemma}\label{L: tangent implies coherent} Let $\alpha$ be a coherent $n_0$-slope at $(x,y)$. Let $c_1$ and $c_2$ be coherent representatives of $\alpha$, and assume that $c_1$ and $c_2$ are independent over $\alpha$. Then $c_2$ is coherent over $c_1xy$.
\end{lemma}
\begin{proof} If $\alpha\in\acl(xy)$, then $c_1$ and $c_2$ are independent over $xy$. By assumption, each $c_i$ is coherent over $xy$, so it follows that $c_1c_2$ is coherent over $xy$, making the lemma obvious. Thus, we assume $\alpha$ is not algebraic.

By Lemma \ref{dim rk comparison}, $\dim(c_2/c_1xy)\leq\rk(c_2/c_1xy)$. We prove the reverse inequality.

Since $\mathcal M$ detects tangency, $\rk(c_2/c_1xy)\leq\rk(c_2/xy)-1$. By assumption $c_2xy$ is coherent, so $\rk(c_2/xy)-1=\dim(c_2/xy)-1$. By Lemma \ref{L: slopes 1 dimensional}, $\dim(c_2/xy)-1=\dim(c_2/\alpha)$. Since the $c_i$ are independent over $\alpha$, $\dim(c_2/\alpha)=\dim(c_2/c_1\alpha)$. By Lemma \ref{L: interdefinability}, $\dim(c_2/c_1\alpha)=\dim(c_2/c_1xy)$.

Putting together everything in the last paragraph gives $\
rk(c_2/c_1xy)\leq\dim(c_2/c_1xy)$, which proves the lemma.
\end{proof}

\begin{lemma}\label{L: two curves dont decrease rank} Let $\alpha$ be a coherent $n_0$-slope at $(x,y)$. Let $c$, $c_1$, and $c_2$ be coherent representatives of $\alpha$, and assume $c$ is independent from $c_1c_2$ over $\alpha$. Then $\operatorname{stp}_{\mathcal M}(c/c_1xy)$ and $\operatorname{stp}_{\mathcal M}(c/c_2xy)$ are parallel (i.e. have a common non-forking extension). 
\end{lemma}
\begin{proof} By transitivity, we have (1) $c$ is independent from $c_1$ over $\alpha$, and (2) $c$ is independent from $c_2$ over $c_1\alpha$ -- equivalently, by Lemma \ref{L: interdefinability}, over $c_1xy$. By (1) and Lemma \ref{L: tangent implies coherent}, $c$ is coherent over $c_1xy$. So, by (2) and Lemma \ref{dim rk comparison}(4), $c$ is $\mathcal M$-independent from $c_2$ over $c_1xy$. Thus, $\operatorname{stp}_{\mathcal M}(c/c_1c_2xy)$ is a non-forking extension of $\operatorname{stp}_{\mathcal M}(c/c_1xy)$. By a symmetric argument, $\operatorname{stp}_{\mathcal M}(c/c_1c_2xy)$ is also a non-forking extension of $\operatorname{stp}_{\mathcal M}(c/c_2xy)$.
\end{proof}

We now prove the proposition. Let $\alpha$ be a coherent $n_0$-slope at $(x,y)$. Let $c_1$ be a coherent representative of $\alpha$. Let $c_2$ be an independent realization of $\operatorname{tp}_{\mathcal K}(c_1/\alpha)$ over $c_1$. Let $p=\operatorname{stp}_{\mathcal M}(c_2/c_1xy)$, and let $b=\operatorname{Cb}(p)$. We show that $s=bxy$ is a code of $\alpha$, by verifying (1), (2), and (3) in Definition \ref{D: code}:

\begin{enumerate}
    \item It is obvious that $(x,y)\in\acl_{\mathcal M}(s)$, since $(x,y)$ is included in $s$.
    \item Let $c$ be a coherent representative of $\alpha$. We show that $s\in\acl_{\mathcal M}(cxy)$. It is enough to show that $b\in\acl_{\mathcal M}(cxy)$.
    
    By replacing $c_2$ with an independent realization if necessary, we may assume that $c_2$ is independent from $c$ over $c_1xy$ (note that this does not change $b$). In particular, this easily implies that $c_2$ is independent from $c_1c$ over $\alpha$. Then by Lemma \ref{L: two curves dont decrease rank}, $q=\operatorname{stp}_{\mathcal M}(c_2/cxy)$ is parallel to $p$, so has canonical base $b$. Since $q$ is a type over $cxy$, this implies that $b\in\acl_{\mathcal M}(cxy)$.
    
    \item We now have to show that $s$ and $\alpha$ are interalgebraic. First, by (2) and Lemma \ref{L: interdefinability}, we have $$s\in\acl(c_1\alpha)\cap\acl(c_2\alpha).$$ Since the $c_i$ are independent over $\alpha$, this implies $s\in\acl(\alpha)$.
    
    Next, we show that $s$ is coherent. Indeed, since $p$ is a type over $c_1xy$, we have $b\in\acl_{\mathcal M}(c_1xy)$. But $c_1xy$ is coherent by assumption, thus so is $c_1bxy$, and thus so is $s$.

    Now toward a contradiction, assume that $\alpha\notin\acl(s)=\acl(bxy)$. By Lemma \ref{L: slopes 1 dimensional}, $\dim(\alpha/xy)=1$. So if $s\in\acl(\alpha)$ but $\alpha\notin\acl(s)$, then $b\in\acl(xy)$. Since $s$ is coherent, this implies $b\in\acl_{\mathcal M}(xy)$. Thus $p$ does not fork over $xy$, and thus $c_1$ and $c_2$ are $\mathcal M$-independent over $xy$. This contradicts that $\mathcal M$ detects tangency.
    \end{enumerate}
\end{proof}

\subsection{Tangent and Multiple Intersections}

We now give conditions under which $\mathcal M$ detects tangency. The idea is to require that an `unusually high slope agreement' between plane curves is always a `topological multiple intersection.' This is analogous to similar facts regarding intersection multiplicity that were utilized in past trichotomy papers (e.g. Bezout's theorem in \cite{Ra} and the argument principle/Rouche's theorem in \cite{MaPi}).

In characteristic $p$ environments, it is possible for any two curves in a family to be first-order tangent at all of their intersection points. However, under mild assumptions, there is always a larger $n$ so that no two independent generic curves are $n$th order tangent at a point. For this reason, the notion of `unusually high' slope agreement is slightly complicated. The informal idea is that for some $n$, the two curves should have the same $n$-slope even though most nearby pairs of curves do not. To make this precise, we will use the machinery of \textit{weakly generic intersections} developed in section 6. We will specifically work with \textit{families of correspondences in $K$}:

\begin{definition} A \textit{$\mathcal K(A)$-definable family of correspondences in $K$} is a $\mathcal K(A)$-definable family $\mathcal X=\{X_t:t\in T\}$ of subsets of $K^2$, such that each $X_t\subset K^2$ has dimension 1 and projects finite-to-one to both copies of $K$.
\end{definition}

In addition to the terminology on weakly generic intersections developed in section 6, we need the following notions involving slopes:

\begin{definition} Let $\mathcal X=\{X_t:t\in T\}$ and $\mathcal Y=\{Y_u:u\in U\}$ be $\mathcal K(A)$-definable families of correspondences in $K$.
\begin{enumerate}
        \item If $n\in\mathbb N$ and $V\subset K^2\times T\times U$ is any open set, we say that $\mathcal X$ and $\mathcal Y$ have \textit{$n$-branching in $V$ over $A$} if for all strongly generic $(\mathcal X,\mathcal Y)$-intersections $(x,y,t,u)$ over $A$, if $(x,y,t,u)\in V$ then $\alpha_{xy}^n(X_t)\neq\alpha_{xy}^n(Y_u)$.
        \item A \textit{generic $(\mathcal X,\mathcal Y)$-tangency over $A$} is a tuple $(x,y,t,u,n)$ such that $(x,y,t,u)$ is a strongly approximable weakly generic $(\mathcal X,\mathcal Y)$-intersection, $n\in\mathbb N$, $\alpha_{xy}^n(X_t)=\alpha_{xy}^n(Y_u)$, and $X$ and $Y$ have $n$-branching in a neighborhood of $(x,y,t,u)$.
    \end{enumerate}
\end{definition}

\begin{definition}\label{D: TIMI}
    $(\mathcal K,\tau,\{\mathcal{IA}_n\},\{T_{mn}\})$ satisfies \textit{TIMI} (`tangent intersections are multiple intersections') if the following holds: suppose $\mathcal X=\{X_t:t\in T\}$ and $\mathcal Y=\{Y_u:u\in U\}$ are $\mathcal K(A)$-definable families of correspondences in $K$. If $(x,y,t,u,n)$ is a generic $(\mathcal X,\mathcal Y)$-tangency over $A$, then $(x,y,t,u)$ is a generic multiple $(\mathcal X,\mathcal Y)$-intersection over $A$. 
\end{definition}

\begin{remark} Recall that at this moment we are working with fixed definable slopes on $(\mathcal K,\tau)$. In the future, after dropping this assumption, we will use the phrase `$(\mathcal K,\tau)$ has definable slopes satisfying TIMI' to mean that $(\mathcal K,\tau)$ has definable slopes $\{\mathcal{IA}_n\}$ and $\{T_{mn}\}$ so that $(\mathcal K,\tau,\{\mathcal{IA}_n\},\mathcal T_{mn})$ satisfies TIMI.
\end{remark}

Recall the notion of detection of multiple intersections, Definition \ref{D: detects noninjectivities}. In those terms, our main result is:

\begin{theorem}\label{T: detection of tangency} Assume that $(\mathcal K,\tau,\{\mathcal{IA}_n\},\{T_{mn}\})$ satisfies TIMI, and $\mathcal M$ detects multiple intersections. Then $\mathcal M$ detects tangency. In particular, $\mathcal M$ interprets a strongly minimal group.
\end{theorem}
\begin{proof}
    Let $\alpha_0$ be a non-algebraic coherent $n$-slope at $(x_0,y_0)$, and let $(x_0,y_0,C_1,c_1,\alpha_0)$ and $(x_0,y_0,C_1,c_1,\alpha_0)$ be coherent $n$-slope configurations. We want to show that $c_1$ and $c_2$ are $\mathcal M$-dependent over $(x_0,y_0)$. We do this by realizing the $C_i$ as generic members of $\mathcal M$-definable families of plane curves, $Z_{t_0}\in\mathcal Z$ and $W_{u_0}\in\mathcal W$; we then apply TIMI to these families (or rather, their images through our fixed map $\rho:M\rightarrow K$, which we call $\mathcal X$ and $\mathcal Y$, respectively). The idea is that the non-algebraicity of $\alpha_0$ forces each of $\mathcal X$ and $\mathcal Y$ to realize infinitely many slopes at $(x_0,y_0)$, which forces $(x_0,y_0,t_0,u_0,n)$ to be a generic $(\mathcal X,\mathcal Y)$ tangency. By TIMI, we conclude that $(x_0,y_0,t_0,u_0)$ is a generic multiple $(\mathcal X,\mathcal Y)$-intersection; then since $\mathcal M$ detects multiple intersections, we get the desired dependence.

    Let us proceed. Let $d_i=\rk(c_i)$ for $i=1,2$. By Lemma \ref{L: the point is generic}, $(x_0,y_0)$ is generic in $M^2$, and each $\dim(c_i/x_0y_0)=d_i-1$. It follows that $d_i\geq 2$: indeed, otherwise $\dim(c_i/x_0y_0)=0$, implying that $\dim(\alpha/x_0y_0)=0$ and contradicting that $\alpha$ is non-algebraic.
    
    Now by Lemma \ref{L: standard families exist}, we can realize each $C_i$ on a generic curve in a standard family (see Definition \ref{D: excellent} and Lemma \ref{L: standard families exist}). That is, we can find standard families of non-trivial plane curves $\mathcal Z=\{Z_t:t\in T\}$ and $\mathcal W=\{W_u:u\in U\}$ over $\emptyset$, and generic $t_0\in T$ and $u_0\in U$, so that $C_1$ is almost contained in $Z_{t_0}$ and $C_2$ is almost contained in $W_{u_0}$. For each $t\in T$, let $X_t=\rho(Z_t)$. Similarly, for each $u\in U$, let $Y_u=\rho(W_u)$. Then $\mathcal X=\{X_t:t\in T\}$ and $\mathcal Y=\{Y_u:u\in U\}$ are $\mathcal K(\emptyset)$-definable families of correspondences in $K$. Note that $\dim(T)=d_1$ and $\dim(U)=d_2$.
    
    We want to apply TIMI to $\mathcal X$ and $\mathcal Y$, at the tuple $(\rho(x_0),\rho(y_0),t_0,u_0,n)$. First, note that $(x_0,y_0,t_0,u_0)$ is a weakly generic $(\mathcal Z,\mathcal W)$-intersection over $\emptyset$ -- and by Lemma \ref{L: generic intersections exist}(3) (recalling that $d_i\geq 2$), $(x_0,y_0,t_0,u_0)$ is also strongly approximable over $\emptyset$. It follows easily that $(\rho(x_0),\rho(y_0),t_0,u_0)$ is a strongly approximable weakly generic $(\mathcal X,\mathcal Y)$-intersection over $\emptyset$.

    Next, by the strong frontier inequality, $Z_{t_0}$ and $C_1$ agree in a neighborhood of $(x_0,y_0)$, as do $W_{u_0}$ and $C_2$. So $\alpha_{x_0y_0}^n(Z_{t_0})=\alpha_{x_0y_0}^n(W_{u_0})=\alpha_0$. Finally, we check:

    \begin{claim} $\mathcal X$ and $\mathcal Y$ have $n$-branching in a neighborhood of $(\rho(x_0),\rho(y_0),t_0,u_0)$ over $\emptyset$.
    \end{claim}
    \begin{claimproof}
        By uniform definability of slopes, $\alpha_0\in\acl(t_0x_0y_0)$. Since $\alpha_0$ is not algebraic, this implies $\dim(t_0/\alpha_0)<\dim(t_0/x_0y_0)=\dim(c_1/x_0y_0)=d_1-1$. So $\dim(t_0/x_0y_0)=d_1-1$ and $\dim(t_0/\alpha_0)\leq d_1-2$.
        
        By uniform definability of slopes, and the definability of dimension, we can find a formula $\phi(x,y,t)\in\operatorname{tp}_{\mathcal K}(x_0,y_0,t_0)$ such that whenever $(x,t,y)\in Z$ is generic and $\phi(x,y,t)$ holds, then for $\alpha=\alpha_{xy}^n(Z_t)$, we have $\dim(t/xy)\leq d_1-1$ and $\dim(t/\alpha)\leq d_1-2$. By the genericity of $(x_0,y_0,t_0)$, $\phi$ holds in a neighborhood of $(x_0,y_0,t_0)$, which we may assume has the form $V_1\times V_2$, where $V_1\subset M^2$ and $V_2\subset T$. Shrinking if necessary, we may assume that $\rho$ is a local homeomorphism on $V_1$, so $\rho(V_1)\subset K^2$ is open. Now let $V=\rho(V_1)\times V_2\times U\subset K^2\times T\times U$.

        To prove the claim, we show that $\mathcal X$ and $\mathcal Y$ have $n$-branching in $V$ over $\emptyset$. To do this, let $(\rho(x),\rho(y),t,u)\in V$ be a strongly generic $(\mathcal X,\mathcal Y)$-intersection over $\emptyset$, where $(x,y)\in V_1$. Let $\alpha=\alpha_{\rho(x)\rho(y)}^n(X_t)=\alpha_{xy}^n(Z_t)$.  
        
        Since $\rho$ is finite-to-one, it follows that $(x,y,t)\in Z$ is generic. Thus $\dim(xyt)=\dim(Z)=d_1+1$. But by the choice of $V_1\times V_2$, $\dim(t/xy)\leq d_1-1$, thus $\dim(xy)=2$.
        
        Now by uniform definability of slopes, $$\dim(txy\alpha)=\dim(txy)=d_1+1.$$ Since $\dim(xy)=2$, this implies $\dim(t\alpha/xy)=d_1-1$. But by the choice of $V_1\times V_2$ again, $\dim(t/\alpha xy)\leq d_1-2$. Thus $\dim(\alpha/xy)\geq 1$. In particular, $\alpha$ is not algebraic.
        
        Now let $\beta=\alpha_{xy}^n(W_u)$. So $\alpha\in\acl(txy)$ and $\beta\in\acl(uxy)$. Since $t$ and $u$ are independent over $\rho(x)\rho(y)$, they are clearly also independent over $xy$. Thus, $\alpha$ and $\beta$ are independent over $xy$. In particular, $\alpha\notin\acl(\beta)$, and thus $\alpha\neq\beta$. This proves the claim. 
    \end{claimproof}
    Now, by the claim, we have verified that $(\rho(x_0),\rho(y_0),t_0,u_0,n)$ is a generic $(\mathcal X,\mathcal Y)$-tangency over $\emptyset$. By TIMI, we conclude that $(\rho(x_0),\rho(y_0),t_0,u_0)$ is a generic multiple $(\mathcal X,\mathcal Y)$-intersection over $\emptyset$. Since $\rho$ is a local homeomorphism near $(x_0,y_0)$, it follows that $(x_0,y_0,t_0,u_0)$ is a generic multiple $(\mathcal Z,\mathcal Y)$-intersection over $\emptyset$. But then since $\mathcal M$ detects multiple intersections, we conclude that $t_0$ and $u_0$ are $\mathcal M$-dependent over $\emptyset$. By $\mathcal M$-interalgebraicity, so are $c_1$ and $c_2$.

    Finally, by coherence, we have $\rk(c_1/x_0y_0)=\dim(c_1/x_0y_0)=d_1-1$, and similarly $\rk(c_2/x_0y_0)=d_2-1$. So if $c_1$ and $c_2$ are $\mathcal M$-independent over $x_0y_0$, then $\rk(c_1c_2/x_0y_0)=d_1+d_2-2$. Since $(x_0,y_0)$ is generic in $M^2$, this implies $\rk(c_1c_2x_0y_0)=d_1+d_2$. Since $c_1$ and $c_2$ are $\mathcal M$-dependent over $\emptyset$, $\rk(c_1c_2)<d_1+d_2$, so $\rk(x_0y_0/c_1c_2)>0$. Since $c_1$ and $c_2$ are canonical bases of strongly minimal sets, this implies $c_1=c_2$. But then by the assumed $\mathcal M$-independence of the $c_i$ over $x_0y_0$, $$0=\rk(c_1/c_2x_0y_0)=\rk(c_1/x_0y_0)=d_1-1.$$ So $d_1=1$. This contradicts that $\dim(t_0/\alpha_0x_0y_0)\leq d_1-2$. We conclude that the $c_i$ are independent over $x_0y_0$, proving that $\mathcal M$ detects tangency.

    Finally, since $\mathcal M$ detects tangency, Proposition \ref{P: codes exist} implies that every coherent $n_0$-slope has a code. Then Proposition \ref{P: group from codes} implies that $\mathcal M$ interprets a strongly minimal group.
\end{proof}

\section{Summing Up}\label{S: summing up}

We are now done with the abstract setting, and we drop all data that we have fixed up until now. The rest of the paper concerns concrete examples. Before moving on, let us summarize our work to this point:

\begin{theorem}\label{T: composite main thm} Let $(\mathcal K,\tau)$ be a Hausdorff geometric structure with enough open maps (in particular, it suffices to assume $(\mathcal K,\tau)$ is either differentiable or has the open mapping property). Assume that $(\mathcal K,\tau)$ has ramification purity, and definable slopes satisfying TIMI. Let $\mathcal M=(M,...)$ be a non-locally modular strongly minimal definable $\mathcal K$-relic. Then $\dim(M)=1$, and $\mathcal M$ interprets a strongly minimal group.
\end{theorem}
\begin{proof} Since non-local modularity is witnessed by a single rank 2 family of plane curves, we may assume the language of $\mathcal M$ is finite. If we then add countably many constants, we can assume the language of $\mathcal M$ is countable and $\acl_{\mathcal M}(\emptyset)$ is infinite.

Next, all of the assumptions on $(\mathcal K,\tau)$ are invariant under adding a countable set of parameters. Since the language of $\mathcal M$ is countable, we may thus assume that every $\emptyset$-definable set in $\mathcal M$ is $\emptyset$-definable in $\mathcal K$. Thus, we have satisfied all requirements in Assumption \ref{A: K and M}, and so we can apply all the results of sections 4-7. 

Now by assumption (or by Proposition \ref{P: open thm implies open maps} or Proposition \ref{P: differentiable implies enough opens}, depending on which property is assumed), $(\mathcal K,\tau)$ has enough open maps. So, by Theorem \ref{T: closure thm}, $\mathcal M$ weakly detects closures. Then, by Theorem \ref{T: detecting ramification}, $\mathcal M$ detects multiple intersections. So by Theorem \ref{T: higher dimensional case}, $\dim(M)=1$. Then, by Theorem \ref{T: detection of tangency}, $\mathcal M$ interprets a strongly minimal group.
\end{proof}

\begin{remark} One also gets the various additional properties of the group in the statement of Proposition \ref{P: group from codes}. We omit them here in order to not overcomplicate the statement. Note, in particular, that changing the language of $\mathcal M$ in the proof above does not affect the additional data in Proposition \ref{P: group from codes}, because that data only involves the background structure $\mathcal K$.
\end{remark}

\section{Examples of Hausdorff Geometric Structures}\label{S: examples}

In this section, we give various examples of Hausdorff geometric structures satisfying some or all of the additional properties we have studied. The most successful example is algebraically closed valued fields, which satisfy all properties (except differentiability in positive characteristic) we have defined.

\subsection{Visceral structrues}
In \cite{DolGooVisceral} Dolich and Goodrick introduce \emph{visceral structures} as a common generalisation of o-minimality and P-minimality, that -- when restricted to the dp-minimal setting -- is very similar to the tame uniform structures of Simon and Walsberg (\cite{SimWal}) called SW-uniformities in \cite{HaHaPeVF}. This formalism also covers the context of 1-h-minimal structures (as pointed out in \cite[Example 2.2.2]{hensel-minII}). It follows immediately from the results of  \cite{DolGooVisceral} that certain strengthenings of visceral structures considered in that paper are Hausdorff geometric structures.

By \cite[Proposition 3.10]{DolGooVisceral} any visceral structure has uniform finiteness, so any visceral structure satisfying the exchange principle is geometric. A visceral structure is said to have \textit{definable finite choice} (DFC) if any definable function with finite fibres has a definable section. In particular, any ordered visceral structure has DFC, and so does any visceral structure with definable Skolem functions. We claim that ($\aleph_1$-saturated) visceral structures with the exchange property and DFC are Hausdorff geometric structures (see Remark \ref{rmk:noDFC} below for a discussion of removing the DFC hypothesis).

First, such structures satisfy the strong frontier inequality by Remark \ref{R: frontier inequality} and \cite[Corollary 3.35]{DolGooVisceral}. The Baire property, as explained in Remark \ref{R: Baire}, would follow from the generic local homeomorphism property and the fact that the topology in Hausdorff visceral structures has no isolated points. So it remains to show the generic local homeomorphism property.
So let $Z\subset X\times Y$, all of the same dimension, with $Z\rightarrow X$ and $Z\rightarrow Y$ finite-to-one, and let $(x,y)\in Z$ be generic. By visceral cell decomposition (\cite{DolGooVisceral}, Theorem 3.23, Definition 3.20), we can assume after shrinking that $X$ and $Y$ are open subsets of $K^n$, where $n=\dim(X)$. Moreover, by either directly applying DFC or by using the frontier inequality (see the proof of Claim \ref{C: ez local homeo}), we can assume after shrinking again that $Z\rightarrow X$ and $Z\rightarrow Y$ are injective, and thus $Z$ defines a bijection $X\leftrightarrow Y$. The result then follows by generic continuity (\cite[Theorem 3.19]{DolGooVisceral}) applied to each coordinate component of $X\leftrightarrow Y$ in each direction.

\begin{remark}\label{rmk:noDFC} In fact, it seems to follow from new results of Johnson (\cite{JohnVisc}) that one can remove the DFC assumption above and conclude that all $\aleph_1$-saturated visceral structures with exchange are Hausdorff geometric structures. For this, one first deduces the strong frontier inequality from Johnson's new version (\cite{JohnVisc}, Theorem 1.10). Now given $Z\subset X\times Y$ as above, one similarly uses the frontier inequality to assume $Z\rightarrow X$ and $Z\rightarrow Y$ are injective (again, as in Claim \ref{C: ez local homeo}), and thus $Z$ defines a bijection $X\leftrightarrow Y$. One then applies the generic continuity clause of Johnson's cell decomposition (\cite{JohnVisc}, Theorem 1.14(2), which is enough by the frontier inequality) to each coordinate component of $X\rightarrow Y$ and $Y\rightarrow X$ as above to get the desired homeomorphism.
\end{remark}

Hensel minimal structures, introduced in \cite{hensel-min} (for residue characteristic $0$ and in \cite{hensel-minII} for positive residue characteristic) cover all pure non-trivially valued henselian fields (of characteristic $0$) as well as many interesting expansions (covering the $V$-minimal algebraically  closed valued fields of \cite{HrKa} and the $T$-convex expansions of power-bounded o-minimal fields, \cite{vdDriesLewen}). One sub-class of Hensel minimal structures of special interest (covering all the above examples) is the class of 1-h-minimal fields. As already mentioned, 1-h-minimal structures are visceral, and it is not hard to verify that they are geometric structures (e.g., \cite[Proposition 2.11]{AcHa}). While 1-h-minimal fields do not necessarily have DFC, every 1-h-minimal field has a 1-h-minimal expansion which does (\cite[Proposition 3.2.3]{hensel-minII}). Thus, we have: 
\begin{corollary}
    Every $\aleph_1$-saturated 1-h-minimal field is a Hausdorff geometric structure. 
\end{corollary}

\def\Fr{\mathrm{Fr}}

\begin{proof}
    By what we have just said, any 1-h-minimal field $\CK$ has a DFC expansion $\CK'$ that is a Hausdorff Geometric Structure. 
    
    In both $\CK$ and $\CK'$, the topology is the valuation topology, and it is definable. Since the dimension in both structures is determined by the topology (see, e.g., \cite[Proposition 2.11]{AcHa}), it follows that for a $\CK$-definable set $X$ we have $\dim_{\CK}(X)=\dim_{\CK'}(X)$. Now consider $X$ a $\CK$-definable set. Then $\Fr(X)$ is also $\CK$-definable, and by the frontier inequality in $\CK'$ we have that  
    
    \[\dim_{\CK}\Fr(X)=\dim_{\CK'}\Fr(X)<\dim_{\CK'}(X)=\dim_{\CK}(X),\] implying the frontier inequality in $\CK$. 

    So we only have to verify the Generic Local Homeomorphism Property. If $Z\sub X\times Y$ are $\CK$-definable of the same dimension, projecting finite-to-one one both components, then for every $\CK'$-generic $(x,y)\in Z$ we know that $Z$ is, in a neighbourhood of $(x,y)$ the graph of a homeomorphism. As this is a $\CK$-definable property of $(x,y)$ (and since dimension in $\CK$ and in $\CK'$ coincide), it must also hold of every $\CK$-generic $(x,y)\in Z$. 
\end{proof}

\begin{remark} As stated before, if one uses the new results of Johnson on visceral structures \cite{JohnVisc}, one does not need DFC, so the above corollary is automatic from the discussion preceding it.
\end{remark}

Both o-minimal and 1-h-minimal expansions of fields have a well-developed (and rather similar) basic differential geometry, turning them into differential Hausdorff geometric structures. Let us give the details in the 1-h-minimal case. The situation in the o-minimal setting is analogous and better known. 

First, we have to define a notion of smoothness meeting Definition \ref{D: smooth}. We declare the smooth locus of a definable set $X$ to be the set of points $x\in X$ where $X$ is, locally near $x$, a $\dim(X)$-dimensional weak strictly differentiable manifold (see \cite[Definition 5.3]{AcHa}). By visceral cell decomposition, every definable set is generically locally a topological manifold, and by Proposition 3.12 \emph{loc. cit.} it is, in fact, generically, a strictly differentiable manifold. This gives all the properties of Definition \ref{D: smooth} except maybe (5), the preservation of smoothness under preimages of “nice enough” projections. But recall (see the remark following Definition \ref{D: smooth}) that this clause is a generalization of, and would follow from, the `submersion theorem' of differential geometry. Moreover, there is a submersion theorem for 1-h-minimal field (\cite[Proposition 4.10]{AcHa}). So we simply apply that proposition exactly as in the earlier remark. 


So now we have to show that the above notion of smoothness admits a differential structure, as in Definition \ref{D: differentiable}. Naturally, in both the 1-h-minimal and the o-minimal settings, the differential structure is given by the tangent space. As before, we focus on the 1-h-minimal setting, as the o-minimal case is similar and more familiar. 

We use the usual notion of a tangent space associated with the notion of (weak) strictly differentiable manifolds (\cite[Definition 5.4]{AcHa}). All properties (1)-(7) of the definition are automatic, and we don't dwell on them. Clause (8) of the definition, asserting that a definable morphism inducing an isomorphism on the level of tangent spaces is open (locally, near the point in question) is an immediate consequence of the Inverse Mapping Theorem (\cite[Proposition 4.4]{AcHa}). So it only remains to check that if $f: (X,x)\to (Y,y)$ is such that $x\in X$ and $y\in Y$ are generic, then the induced map between the tangent spaces is surjective. This, in turn, is an immediate result of the corresponding version of Sard's Theorem \cite[Proposition 5.21]{AcHa}, asserting that the set of singular values of $f$ is non-generic. 

So we conclude: 

\begin{theorem}\label{T: ominimal} Every $\aleph_1$-saturated o-minimal expansion of a real closed field, as well as every $\aleph_1$-saturated 1-h-minimal valued field, is a differential Hausdorff geometric structure.
\end{theorem}


\subsection{\'ez Expansions of Fields}

We turn now toward ACVF in all characteristics. We will work with a general setting of topological fields. In fact, the ensuing results could serve as an alternate proof of Theorem \ref{T: ominimal} for RCVF and for pure characteristic zero Henselian fields.

Our abstract setting is inspired by \textit{\'ez fields}. We briefly recall this notion. First, recall that a field $K$ is \textit{large} if, whenever $V$ is a smooth variety over $K$ and $V(K)\neq\emptyset$, then $V(K)$ is infinite (equivalently, $V(K)$ is Zariski dense in some irreducible component). In \cite{JTWY}, the authors place a canonical non-discrete topology on the set $V(K)$ for every (irreducible) variety $V$ over a large field $K$. The topology is called the \textit{\'etale open topology}, and is defined as the weakest topology where \'etale maps are open (so it is generated by the images $f(W(K))$ of all \'etale morphisms of $K$-varieties $f:W\rightarrow V$). It is shown in \cite[Section 6]{JTWY} and ~\cite{field-top-1,field-top-2} that the \'etale open topology often coincides with existing topologies of interest when the field is close to being henselian (particularly, the order topology in real closed fields, and the valuation topology in characteristic zero Henselian fields). 

One says that the \'etale open topology on $K$ is \textit{induced by a field topology} if there is a field topology $\tau$ on $K$ such that, for every affine variety $V$ over $K$, the \'etale open topology agrees on $V(K)$ with the topology induced on $V(K)$ by $\tau$ (via the product and subspace topologies). In general, the \'etale open topology is induced by a field topology if it respects products (i.e. the \'etale open topology on $K^n$ is the product topology induced by the \'etale open topology on $K$). This is not always the case (the main counterexample being pseudofinite fields, or more generally PAC fields). See \cite[Proposition 4.9 and Section 8]{JTWY} or~\cite{field-top-2} for more details.

The follow-up work \cite{ez} then defined \'ez fields as large fields $K$ such that every definable set $X\subset K^n$ (in the pure field language) is a union of finitely many definable \'etale-open subsets of Zariski closed sets. The idea is that \'ez fields have `local quantifier elimination' -- where the word 'local' is in the sense of the \'etale open topology. \'Ez fields include algebraically closed fields, real closed fields, pseudofinite fields, and characteristic zero Henselian fields (e.g. $p$-adically closed fields). It is shown in \cite{ez} that many tameness properties of definable sets in these examples (e.g. generic continuity of functions) can be adapted to the general \'ez setting.

It is well-known that algebraically closed valued fields and real closed valued fields satisfy a similar local quantifier elimination: every definable set is a union of finitely many valuation-open subsets of Zariski closed sets. In fact, this description of definable sets holds in all Henselian valued fields of characteristic 0~\cite{lou-dimension}. Our goal is to introduce a single notion capturing such a structural decomposition and incorporating \'ez fields. Since we have used the product topology freely throughout the paper, we will restrict to those \'ez fields whose \'etale open topology is induced by a field topology. 

Before making the general definition, we establish some conventions:

\begin{notation}
    Suppose $\mathcal K=(K,+,\cdot,...)$ is an expansion of a field. Throughout the rest of Section 9:
    \begin{enumerate}
        \item By a variety over $K$, we mean a reduced, separate scheme of finite type over $K$. Varieties are denoted with letters such as $V$, $W$, etc. Note that varieties are not definable objects in a first-order sense.
        \item If $V$ is a quasi-projective variety over $K$, we use $V(K)$ to denote the $\mathcal K$-interpretable set of $K$-rational points of $V$.
        \item If $X\subset K^n$ is definable, the \textit{Zariski closure} of $X$ is the smallest affine $K$-variety $V$ such that $X\subset V(K)$. 
        \item On the other hand, we still use the term \textit{Zariski closed} for subsets of $K^n$: $X\subset K^n$ is Zariski closed if $X=V(K)$ for some affine variety $V$ over $K$.
        \item Similarly, if $V$ is a variety over $K$, and $X\subset V(K)$, we say that $X$ is \textit{Zariski dense in $V$} if for every proper closed subvariety $W\subset V$, there is some $x\in X\cap V(K)-W(K)$.
        \item If $V$ is a quasi-affine variety over $K$, and $A\subset K$, we distinguish between $V$ being \textit{defined over $A$} (i.e. defined by polynomials with coefficients in $A$) and $V(K)$ being \textit{definable over $A$} (definable in the structure $\mathcal K$ with parameters $A$).
        \item If $V$ is a quasi-affine variety over $K$, the notation $\dim(V)$ refers to the dimension of $V$ as an algebraic variety. The notation $\dim(V(K))$ will only be used if $\mathcal K$ is a geometric structure; and in this case, it refers to dimension in the sense of geometric structures.
    \end{enumerate}
    \end{notation}

    We will use freely the following:

    \begin{lemma}\label{L: zariski closure definable} Let $\mathcal K=(K,+,\cdot,...)$ be an expansion of a field. Let $X\subset K^n$ be definable over $A$, and let $Z$ be its Zariski closure.
    \begin{enumerate}
        \item $Z(K)$ is definable over $A$ (regardless of whether the variety $Z$ is over $A$).
        \item If $Z^S$ is the smooth locus of $Z$, then $Z^S(K)$ is definable over $A$.
    \end{enumerate}
    \end{lemma}
    \begin{proof} These statements are unaffected by passing to an elementary extension of $\mathcal K$. In particular, as opposed to topological statements about Hausdorff geometric structures, the current lemma only references definable objects.
    
    So, let us assume that $\mathcal K$ is a `monster model' -- i.e. $\kappa$-saturated and $\kappa$-strongly homogenous for some large $\kappa$. In this case, one can detect the parameters in a definition using automorphism invariance. Thus, to prove the lemma, it suffices to note that $Z(K)$ and $Z^S(K)$ are invariant under field automorphisms fixing $X$ setwise.
    \end{proof}
    
    We also need a general notion of an \textit{invariant topology} on a structure:

\begin{definition}\label{D: invariant}
    Let $\mathcal K=(K,...)$ be a structure, and $\tau$ a topology on $K$. We call $\mathcal K$ \textit{invariant} if there is a basis $\mathcal B$ for $\tau$ such that:
    \begin{enumerate}
        \item Each $X\in\mathcal B$ is definable.
        \item Suppose $\phi(x,y)$ is a formula, where $x$ is a single variable and $y$ is a tuple. Let $a$ and $b$ be tuples in the arity of $y$ with $\tp(a)=\tp(b)$. Let $X$ and $Y$ be the sets defined by $\phi(x,a)$ and $\phi(x,b)$, respectively. Then $X\in\mathcal B$ if and only if $Y\in\mathcal B$. 
    \end{enumerate}
\end{definition}

Note that, if $(K,+,\cdot,\tau)$ is a topological field, then $\tau$ induces a natural topology on $V(K)$ for every variety $V$ over $K$ (the affine case is given by the subspace topology from $K^n$, and the general case is obtained by gluing; see  \cite[Page 57, Chapter I Section 10]{mum}, for example). Using this, and inspired by the above discussion, we now define:

\begin{definition}\label{D: ez field}
    An \textit{\'ez topological field expansion} is an $\aleph_1$-saturated structure $\mathcal K=(K,+,\cdot,...)$ over a countable language, equipped with an invariant Hausdorff topology $\tau$ on $K$ (extended canonically to all subsets of powers of $K$), such that:
    \begin{enumerate}
        \item $(K,+,\cdot,\tau)$ is a non-discrete topological field.
        \item Every definable set $X\subset K^n$ is a finite union of definable $\tau$-open subsets of Zariski closed sets.
        \item For each \'etale morphism $f:V\rightarrow W$ of varieties over $K$, the induced map $V(K)\rightarrow W(K)$ is $\tau$-open. 
        \end{enumerate}
\end{definition}

\begin{remark} Suppose $(\mathcal K,\tau)$ is an \'ez topological field expansion. Note that since $\tau$ is invariant, every elementary extension of $\mathcal K$ is canonically a field with an invariant topology satisfying (1) and (3) of Definition \ref{D: ez field}. However, (2) is not first-order unless an additional uniformity condition is assumed. Thus, it is not clear that elementary extensions of $\mathcal K$ are also \'ez topological field expansions (though this holds in most natural examples).
\end{remark}

Note that if $(K,+,\cdot)$ is an $\aleph_1$-saturated \'ez field whose \'etale open topology is induced by a field topology, then $K$ is (trivially) an \'ez topological field expansion when endowed with the \'etale open topology.

It also follows from well-known facts that $\aleph_1$-saturated models of ACVF and RCVF are \'ez topological field expansions. We do not elaborate on RCVF, because it is 1-h-minimal (thus covered by the previous subsection). On the other hand, let us sketch the argument for ACVF.

The main fact we need is:

\begin{fact}[\cite{open-mapping}]\label{F: open mapping} Let $(K,v)$ be an algebraically closed valued field. Let $f:V\rightarrow W$ be a universally open morphism of $K$-varieties. Then the induced map $V(K)\rightarrow W(K)$ is open in the valuation topology.
\end{fact}

\begin{lemma}\label{L: acvf ez} Let $\mathcal K=(K,v)$ be an $\aleph_1$-saturated model of ACVF. Then $(K,v)$ is an \'ez topological field expansion.
\end{lemma}
\begin{proof} Let $\tau$ be the valuation topology. As is well-known, $\tau$ is a non-discrete Hausdorff field topology. Moreover, as the collection of balls is a $\emptyset$-definable basis, $\tau$ is invariant.

So it remains to check (2) and (3) in Definition \ref{D: ez field}.$(\mathcal K,\tau)$. (2) is a well-known consequence of quantifier elimination (see also \cite{lou-dimension}). For (3), recall that \'etale morphisms are universally open (\cite[Lemmas 03WT and 02GO]{stacks-project}), and thus apply Fact \ref{F: open mapping}.  
\end{proof}

\begin{assumption}
\textbf{For the rest of section 9, we fix an \'ez topological field expansion $(\mathcal K,\tau)$, with underlying field $(K,+,\cdot)$.}
\end{assumption}

\begin{remark} Note that, since $\tau$ is a field topology, it refines the Zariski topology: for every affine variety $V$ over $K$, the set $V(K)$ is $\tau$-closed. Similarly, every morphism of $K$-varieties induces a $\tau$-continuous map on $K$-points. We will use these facts throughout.
\end{remark}

\subsection{Basic Properties}

Our general goal is to show that $(\mathcal K,\tau)$ is a Hausdorff geometric structure. First, we note that many of the facts from \cite{ez} transfer (that is, there is no dependence on the particular language or topology). Thus, we now prove a series of basic lemmas about $(\mathcal K,\tau)$. \\

We begin with:

\begin{lemma}\label{L: infinite sets have open set}
Let $X\subset K^n$ be definable and Zariski dense in $\mathbb A^n$. Then $X$ has non-empty interior in the sense of $\tau$. In particular, every infinite definable subset of $K$ has non-empty interior.
\end{lemma}
\begin{proof}
    By Definition \ref{D: ez field}(2), we can write $X=\bigcup_{i=1}^mU_i$, where each $U_i$ is a definable open subset of $Z_i(K)$ for some closed subvariety $Z_i$ of $\mathbb A^n$. Since $X$ is Zariski dense in $\mathbb A^n$, there is some $i$ such that $U_i$ is Zariski dense in $\mathbb A^n$. It follows that $Z_i=\mathbb A^n$, so that $Z_i(K)=K^n$. But then $U_i$ is open in $K^n$ (and since it is Zariski dense, it is non-empty).
\end{proof}

The next two lemmas are consequences of Lemma \ref{L: infinite sets have open set}:

\begin{lemma}\label{L: ez perfect} 
\begin{enumerate}
    \item $K$ is perfect. 
    \item If $\operatorname{char}(K)=p>0$ then the Frobenius map $x\mapsto x^p$ is a homeomorphism.
\end{enumerate}
\end{lemma}
\begin{proof} 
\begin{enumerate}
    \item Suppose $\operatorname{char}(K)=p>0$ and let $K^p$ be the set of $p$th powers. By Lemma \ref{L: infinite sets have open set}, $K^p$ has non-empty interior. Let $U\subset K^p$ be a non-empty definable open set. Since $K^p$ is a subfield, we may translate and assume $0\in U$.
    
    Now let $a\in K$. We show $a\in K^p$. If $a=0$ this is clear. If $a\neq 0$, then $U\cap\frac{1}{a}U$ is a neighborhood of 0, and since $\tau$ is non-discrete, there is some $b\in U\cap\frac{1}{a}U$ with $b\neq 0$. By the choice of $U$, there are $x,y\in K$ with $x^p=b$ and $y^p=ab$. Then $(\frac{y}{x})^p=a$.
    \item By (1), $x\mapsto x^p$ is a continuous bijection. We show it is open. Since $\tau$ is a field topology, and $x\mapsto x^p$ respects addition, it will suffice to show that if $X\subset K$ is open and contains 0, then $X^p$ contains a neighborhood of 0. So, fix such an $X$. We may assume $X$ is definable. Since subtraction is continuous, there is an open definable $Y$ containing 0 such that $Y-Y\subset X$. Since $\tau$ is not discrete, $Y$ is infinite, and thus so is $Y^p$. So by Lemma \ref{L: infinite sets have open set}, $Y^p$ contains an infinite open set $Z$. Let $a\in Z$. Then $Z-a$ is a neighborhood of 0, and $$Z-a\subset Z-Z\subset Y^p-Y^p\subset X^p,$$ as desired.
\end{enumerate}
\end{proof}

\begin{lemma}\label{L: ez geometric} $\mathcal K$ is geometric.
\end{lemma}
\begin{proof} It was shown in \cite{JohYe} that if an expansion of a field satisfies exchange, it also eliminates $\exists^{\infty}$. Thus, we only need to show that $\mathcal K$ satisfies exchange. But if not, then there is a definable $X\subset K^2$ so that (1) the left projection $X\rightarrow K$ is finite-to-one and (2) the right projection $X\rightarrow K$ has infinitely many infinite fibers. By (2), $X$ is Zariski dense in $\mathbb A^2$, so by Lemma \ref{L: infinite sets have open set}, $X$ has non-empty interior. Thus $X$ contains a set $B_1\times B_2$, where each $B_i\subset K$ is non-empty and open. Since $\tau$ is non-discrete, each $B_i$ is infinite. This contradicts (1).
\end{proof}

Now that we know $\mathcal K$ is geometric, we will use the notation $\dim$ on definable sets. We next show:

\begin{lemma}\label{L: open in smooth implies top dim} Let $V$ be a smooth variety over $K$. Let $X$ be a non-empty open definable subset of $V(K)$. Then $\dim(X)=\dim(V)$. In particular, $K$ is a large field.
\end{lemma}
\begin{proof} It is easy to see that $\dim(X)\leq\dim(V)$. We show the reverse inequality. Let $a\in X$, and $n=\dim(V)$. Using that $a$ is a smooth point of $V$, one can show there is a rational map $f:V\rightarrow\mathbb A^n$ which is defined and \'etale at $a$. Let $U\subset V$ be the (necessarily relatively Zariski open) \'etale locus of $f$. Without loss of generality, we may assume $X\subset U(K)$. Now by Definition \ref{D: ez field}(3), the image $\pi(X)\subset K^n$ is non-empty and open, so contains a non-empty definable box $B_1\times...\times B_n$, where each $B_i\subset K$ is open. Since $\tau$ is non-discrete, each $B_i$ is infinite. Thus $\dim(B_1\times...\times B_n)=n$. Finally, since $\pi$ is \'etale on $U$ it is finite-to-one, and thus preserves dimension. So $$\dim(X)=\dim(f(X))\geq\dim(B_1\times...\times B_n)=n.$$
\end{proof}

The next three lemmas give the main technical tools we will need to show that $(\mathcal K,\tau)$ is a Hausdorff geometric structure. We will use the following notion:

\begin{definition}
    Let $X\subset K^n$ be definable, with Zariski closure $Z$. Let $Z^S$ be the smooth locus of $Z$. Let $x\in Z(K)$. We call $x$ \textit{polished for $X$} if there is an open neighborhood $U$ of $x$ in $Z(K)$ such that:
    \begin{enumerate}
        \item $U\subset Z^S(K)$.
        \item $U$ is either contained in $X$ or disjoint from $X$.
    \end{enumerate}
\end{definition}
\begin{lemma}\label{L: ez main lemma first version} Let $X\subset K^n$ be definable, with Zariski closure $Z$. Then there is a closed subvariety $T\subset Z$ with $\dim(T)<\dim(Z)$, such that every $x\in Z(K)-T(K)$ is polished for $X$.
\end{lemma}
\begin{proof}
    By Definition \ref{D: ez field}(2), we can write $X=\bigcup_{i=1}^mU_i$ and $Z(K)=\bigcup_{j=1}^kV_j$, where each $U_i$ (resp. $V_j$) is a definable open subset of $Z_i(K)$ (resp. $W_j(K)$) for some affine varieties $Z_i$ and $W_j$ over $K$. Refining the decomposition if necessary, one shows easily that we may assume each $Z_i$ and $W_j$ is irreducible over $K$ and contained in $Z$ (note that irreducibility is only ensured over $K$, so we cannot assume absolute irreducibility).

    Now, since $K$ is perfect, varieties over $K$ are generically smooth. That is, if $Z^S$ is the (Zariski open) smooth locus of $Z$, then $\dim(Z-Z^S)<\dim(Z)$. We now define $T$ to be the union of $Z-Z^S$ with all those $Z_i$ and $W_j$ of dimension less than $\dim(Z)$. Clearly, then, $\dim(T)<\dim(Z)$.

    To see this works, let $x\in Z(K)-T(K)$. We show that $x$ is polished for $X$. Without loss of generality, assume $x\in X$ (the case $x\in Z-X$ is symmetric). So $x\in U_i$ for some $i$. By the choice of $T$, $\dim(Z_i)\geq\dim(Z)$. But by assumption, $Z_i$ is an irreducible closed subvariety of $Z$. It follows that $Z_i$ is an irreducible component of $Z$.
    
    Recall that to show $x$ is polished, we need to find a neighborhood of $x$ (in $Z(K)$) which is contained in $X\cap Z^S(K)$. We claim that $U=U_i\cap Z^S(K)$ works. Clearly $U\subset X\cap Z^S(K)$, so all we need to show is that $U$ is relatively open in $Z(K)$.

    For this, we establish the following three claims:
    
    \begin{claim} $U$ is relatively open in $Z_i(K)\cap Z^S(K)$.
    \end{claim}
    \begin{claimproof} By assumption, $U_i$ is open in $Z_i(K)$. Thus, automatically, $U=U_i\cap Z^S(K)$ is open in $Z_i(K)\cap Z^S(K)$.
    \end{claimproof}
    
    \begin{claim} $Z_i(K)\cap Z^S(K)$ is relatively open in $Z^S(K)$.
    \end{claim}
    \begin{claimproof} Note that (by generic smoothness and the fact that $\dim(Z_i)=\dim(Z)$) $Z_i\cap Z^S$ is an irreducible component of the smooth variety $Z^S$. Since smooth Zariski-connected varieties are irreducible, the irreducible components of $Z^S$ are relatively Zariski open. Thus $Z_i\cap Z^S$ is Zariski open in $Z^S$; and since $\tau$ refines the Zariski topology, $Z_i(K)\cap Z^S(K)$ is relatively open in $Z^S(K)$.
    \end{claimproof}
    
    \begin{claim} $Z^S(K)$ is relatively open in $Z(K)$.
    \end{claim}
    \begin{claimproof} Since $\tau$ refines the Zariski topology, it is enough to note that $Z^S$ is automatically Zariski open in $Z$.
    \end{claimproof}
    
    By the three claims, $U$ is relatively open in $Z(K)$, completing the proof of Lemma \ref{L: ez main lemma first version}.
\end{proof}

We now use Lemma \ref{L: ez main lemma first version} to characterize the dimension function of $\mathcal K$:

\begin{lemma}\label{L: ez dim is Zar dim} Let $X\subset K^n$ be definable, with Zariski closure $Z$. Then $\dim(X)=\dim Z(K)=\dim Z$.
\end{lemma}
\begin{proof} It is easy to see that $\dim(X)\leq\dim(Z(X))\leq\dim(Z)$. We show that $\dim(X)\geq\dim(Z)$.

Let $T$ be a relatively closed subvariety as in Lemma \ref{L: ez main lemma first version}. By the choice of $Z$, $X$ is Zariski dense in $Z$. In particular, there is some $x\in X-T(K)$. So, by the choice of $T$, $x$ is polished for $X$. It follows by definition that $X$ contains a non-empty open subset of $Z^S(K)$, where $Z^S$ is the smooth locus of $Z$. So by Lemma \ref{L: open in smooth implies top dim}, $\dim(X)\geq\dim(Z^S)$. But by generic smoothness over perfect fields, $\dim(Z^S)=\dim(Z)$, completing the proof.
\end{proof}

Finally, we deduce our main technical lemma:

\begin{lemma}\label{L: ez main lemma second version} Let $X\subset K^n$ be definable over a countable set $A$. Let $Z$ be the Zariski closure of $X$. If $x\in Z(K)$ is generic over $A$, then $x$ is polished for $X$.
\end{lemma}
\begin{proof}
  Let $T$ be a closed subvariety of $Z$ as in Lemma \ref{L: ez main lemma first version}. Let $B\supset A$ be countable so that $T(K)$ is definable over $B$.

  Let $y$ be an independent realization of $\tp(x/A)$ over $B$. So $y\in Z(K)$, and by Lemma \ref{L: ez dim is Zar dim}, we have $$\dim(y/B)=\dim(x/A)=\dim(Z)>\dim(T)\geq\dim(T(K)),$$ so that $y\notin T(K)$. Thus $y$ is polished for $X$. Let $\phi(z,b)$ (for $z$ in the arity of $Z(K)$) be a formula defining a neighborhood of $y$ witnessing that $y$ is polished for $X$. Since $\tp(x/A)=\tp(y/A)$, there is $a$ with $\tp(xa/A)=\tp(yb/A)$. Then by the invariance of $\tau$ (Definition \ref{D: invariant}), it follows that $\phi(z,a)$ defines a neighborhood of $x$ witnessing that  $x$ is polished for $X$.
\end{proof}

\begin{remark}\label{R: ez definable sets generically smooth}
    Lemma \ref{L: ez main lemma second version} says that, near generic points, we can treat definable sets as smooth varieties. More precisely, if $x\in X$ is generic over $A$, then $x$ is also generic in $S(K)$ for some smooth variety $S$ having the same dimension and germ as $X$ at $x$. Thus, if we are only interested in d-local properties of $(X,x)$, we may replace $X$ with $S(K)$. If we are also willing to pass from $A$ to $\acl(A)$ (which does not affect the genericity of $x$), then we may further assume $S$ is irreducible over $A$. This will be our main strategy moving forward.
\end{remark}

\subsection{Homeomorphisms and Reduction to Generically Smooth Maps}

One of the difficulties of the trichotomy in the positive characteristic case is the existence of finite everywhere-ramified morphisms. As has become standard, we will get around this problem by decomposing an arbitrary map into a universal homeomorphism composed with a generically smooth map. Then, in studying topological properties, we can pay attention only to the generically smooth part.

Such a decomposition was given explicitly in~\cite{ez}, using the reduced relative Frobenius construction. This is not quite good enough for us, because we need the universal homeomorphism to preserve normality. We thus give a more intricate argument:

\begin{lemma}\label{L: relative frobenius} Let $V\rightarrow W$ be a dominant morphism of irreducible varieties over $K$. Then there are an irreducible variety $V'$, and a factorization $V\rightarrow V'\rightarrow W$, such that:
\begin{enumerate}
    \item $V\rightarrow V'$ is a universal homeomorphism, and $V(K)\rightarrow V'(K)$ is a $\tau$-homeomorphism.
    \item $V'\rightarrow W$ is generically smooth.
    \item If $V$ is normal then so is $V'$.
\end{enumerate}
\end{lemma}
\begin{proof} It is shown in~\cite[Lemma 1.12, Fact 3.3 and 3.5]{ez} that there is such a factorization satisfying (1) and (2), given by a sufficient iterate of the reduced relative Frobenius of $V\rightarrow W$ (the fact that this construction gives a $\tau$-homeomorphism follows from Lemma \ref{L: ez perfect}(2)). Call this factorization $V\rightarrow V_1\rightarrow W$. Now suppose $V$ is normal (and note that $V_1$ might not be). We show how to refine the construction to satisfy (3).

Let $N\rightarrow V_1$ be the normalization of $V_1$. So $N\rightarrow V_1$ is surjective and birational. Since $V$ is normal, and by the universal property of normalizations~\cite[Lemma 035Q(4)]{stacks-project}, $V\rightarrow V_1$ factors as $V\rightarrow N\rightarrow V_1$. Since dominant quasi-finite morphisms to normal varieties are universally open~\cite[Lemma 0F32]{stacks-project}, the image of $V$ in $N$ is an open (so also normal) subvariety, say $U$.

We now have $V\rightarrow U\rightarrow N\rightarrow V_1\rightarrow W$, where

\begin{enumerate}
    \item $V\rightarrow U$ is surjective and universally open.
    \item $U\rightarrow N$ is an open immersion.
    \item $N\rightarrow V_1$ is surjective and birational.
    \item $V_1\rightarrow W$ is generically smooth.
\end{enumerate}

Since $V\rightarrow V_1$ is a universal homeomorphism, it is universally injective. Thus $V\rightarrow U$ is also universally injective. Note that base surjections are automatically universally surjective~\cite[Lemma 01S1]{stacks-project}. Thus, $V\rightarrow U$ is universally open, universally injective, and universally surjective, and so is a universal homeomorphism. 

We now apply the reduced relative Frobenius again, this time to $V\rightarrow U$. We obtain a factorization $V\rightarrow V'\rightarrow U$ again satisfying (1) and (2) of the lemma (for $V\rightarrow U$). Thus, we have dominant morphisms $$V\rightarrow V'\rightarrow U\rightarrow N\rightarrow V_1\rightarrow W.$$ To complete the proof of the lemma, we note:

\begin{claim} $V\rightarrow V'$ is a universal homeomorphism and a $\tau$-homeomorphism.
\end{claim}
\begin{claimproof} This is because $V\rightarrow V'\rightarrow U$ satisfies (1) and (2).
\end{claimproof}
\begin{claim} $V'\rightarrow W$ is generically smooth.
\end{claim}
\begin{claimproof} Because $V'\rightarrow U\rightarrow N\rightarrow V_1\rightarrow W$ is a chain of generically smooth maps.
\end{claimproof}

And, the main point:

\begin{claim} $V'$ is normal.
\end{claim}
\begin{claimproof}
    We have that $V\rightarrow U$ and $V\rightarrow V'$ are both universal homeomorphisms, which implies that $V'\rightarrow U$ is also a universal homeomorphism~\cite[Lemma 0H2M]{stacks-project}, and thus (by generic smoothness) a birational universal homeomorphism. But then, as a birational universal homeomorphism to a normal variety, $V'\rightarrow U$ is an isomorphism~\cite[Lemma 0AB1]{stacks-project}. Thus $V'$ is normal because $U$ is.
\end{claimproof}
\end{proof}

\begin{remark}\label{R: rel frob preserves genericity} Suppose, in the setup of Lemma \ref{L: relative frobenius}, the morphism $V(K)\rightarrow W(K)$ is defined over $A$. We note that each step in the proof of the lemma was given by a functorial construction. It then follows that the whole sequence $V(K)\rightarrow V'(K)\rightarrow W(K)$ is still defined over $A$. We will use this implicitly.
\end{remark}

\subsection{Proofs of the Axioms}

Armed with the machinery from the previous subsection, we now show that $(\mathcal K,\tau)$ is a Hausdorff geometric structure. We also give conditions under which $(\mathcal K,\tau)$ has enough open maps.

\begin{theorem}\label{T: ez fields} $(\mathcal K,\tau)$ is a Hausdorff geometric structure.
\end{theorem}
\begin{proof}
We already showed that $\mathcal K$ is geometric. Thus, it suffices to establish the strong frontier inequality, the Baire category axiom, and the generic local homeomorphism property. These are given in the next three claims.

\begin{claim} Let $X\subset K^n$ be definable over $A$, and let $a\in\overline{\operatorname{Fr}(X)}$. Then $\dim(a/A)<\dim(X)$.
\end{claim}
\begin{claimproof} Let $Z$ be the Zariski closure of $X$. Note that $a\in Z(K)$, since $a\in\overline X$ and $\tau$ refines the Zariski topology. By Lemma \ref{L: ez dim is Zar dim}, $\dim(Z(K))=\dim(X)$. Now assume toward a contradiction that $\dim(a/A)=\dim(X)$. Then $\dim(a/A)=\dim(Z(K))$, so $a$ is generic in $Z(K)$. By Lemma \ref{L: ez main lemma second version}, $a$ is thus polished for $X$. But the definition of being polished  for $X$ implies that $a\notin\overline{\operatorname{Fr}(X)}$, a contradiction.
\end{claimproof}

\begin{claim} Let $X\subset K^n$ be definable over $A$, and let $a\in X$ be generic over $A$. Let $B\supset A$ be countable. Then every neighborhood of $a$ contains a generic of $X$ over $B$.
\end{claim}
\begin{claimproof} Let $Z$ be the Zariski closure of $X$, and let $Z^S$ be its smooth locus. By Lemma \ref{L: ez dim is Zar dim}, $\dim(X)=\dim(Z(K))$. So since $a$ is generic in $X$ over $A$, it is also generic in $Z(K)$ over $A$. So by Lemma \ref{L: ez main lemma second version}, $a$ is polished for $X$. Let $U\subset Z(K)$ be a neighborhood of $a$ witnessing this. That is, $U$ is open in $Z(K)$, and $U$ is contained in $X\cap Z^S(K)$.

To prove the claim, let $V$ be any neighborhood of $a$ in $X$. Shrinking if necessary, we may assume $V$ is definable and contained in $U$, so in particular $V$ is relatively open in $Z^S(K)$. Adding to $B$ if necessary, we may assume $V$ is definable over $B$.

Now by Lemmas \ref{L: open in smooth implies top dim} and \ref{L: ez dim is Zar dim}, and generic smoothness over perfect fields, we conclude that $$\dim(V)=\dim(Z^S)=\dim(Z)=\dim(X).$$

In particular, if we choose any $b\in V$ generic over $B$, then $b$ is also generic in $X$ over $B$, proving the claim.
\end{claimproof}

\begin{claim}\label{C: ez local homeo} Let $Z\subset X\times Y$ be definable over $A$ and all of dimension $d$, with $Z\rightarrow X$ and $Z\rightarrow Y$ finite-to-one. Let $(x,y)\in Z$ be generic over $A$. Then $Z$ restricts to a homeomorphism between neighborhoods of $x$ and $y$ in $X$ and $Y$, respectively.
\end{claim}
\begin{claimproof} It is enough to show the projection $Z\rightarrow X$ is locally a homeomorphism near $(x,y)$, since by symmetry the same will apply to $Z\rightarrow Y$. For this, it suffices to show separately that $Z\rightarrow X$ is locally injective, locally continuous, and locally open at $(x,y)$. Moreover, local continuity is automatic, since $Z\rightarrow X$ is a projection. So we show local injectivity and local openness.

First, suppose $Z\rightarrow X$ is not locally injective at $(x,y)$. Then $(x,y,y)\in\operatorname{Fr}(T)$, where $T$ is the set of $(u,v,w)$ with $(v\neq w)$ and $(u,v),(u,w)\in Z$. Since $Z\rightarrow X$ is finite-to-one, $\dim(T)\leq d$. Then by the strong frontier inequality, we get $\dim(xy/A)<d$, a contradiction. So $Z\rightarrow X$ is locally injective near $(x,y)$.

Now we show local openness at $(x,y)$. The idea is that, as outlined in the beginning of this subsection, we can reduce to the case that $Z\rightarrow X$ is generically smooth -- thus \'etale -- and then apply Definition \ref{D: ez field}(3). More precisely, by Lemma \ref{L: ez main lemma second version}, there are smooth d-dimensional varieties $W$ and $V\subset\mathbb A^n\times W$ so that:

\begin{enumerate}
    \item $V(K)$ and $W(K)$ are definable over $A$.
    \item $V(K)$ and $W(K)$ have the same germs as $Z$ and $X$ at $(x,y)$ and $x$, respectively.
\end{enumerate}

We may assume, without loss of generality, that $A=\acl(A)$. Thus, we may also assume $V$ and $W$ are irreducible. Since $Z\rightarrow X$ is finite-to-one, it follows that $V\rightarrow W$ is dominant.

Now let $V\rightarrow V'\rightarrow W$ be a factorization as in Lemma \ref{L: relative frobenius}, and let $z$ be the image of $(x,y)$ in $V'(K)$. It follows from Remark \ref{R: rel frob preserves genericity} that $z$ is generic in $V'(K)$. Moreover, since $V(K)\rightarrow V'(K)$ is a homeomorphism, it suffices to show that $V'(K)\rightarrow W(K)$ is open near $z$. But by the choice of $V'$ (and the fact that $\dim(V')=\dim(W)$, $V'\rightarrow W$ is generically \'etale; and since $z\in V'$ is generic, it follows that $V'\rightarrow W$ is \'etale in a Zariski neighborhood of $z$. Local openness then follows as promised from Definition \ref{D: ez field}(3).
\end{claimproof}
\end{proof}

\subsection{Smoothness and Enough Open Maps}

We now proceed to give conditions under which $(\mathcal K,\tau)$ is either differentiable or has the open mapping property. First, we define a canonical notion of smoothness.

\begin{definition}
    For definable $X\subset K^n$, we let $X^S$ be the set of $x\in X$ which are polished for $X$ -- that is, the relative interior of $X$ in the smooth locus of its Zariski closure.
\end{definition}

In other words, the smooth points of $X$ are those points where we can view $X$ d-locally as a smooth variety. We now show:

\begin{lemma}\label{L: ez smooth} The map $X\mapsto X^S$ is a notion of smoothness on $(\mathcal K,\tau)$.
\end{lemma}
\begin{proof}
    Most of the properties in Definition \ref{D: smooth} follow from standard facts about smooth varieties (e.g. closure under products and isomorphisms). The fact that generic points are smooth follows from Lemma \ref{L: ez main lemma second version}. We show only Definition \ref{D: smooth}(5) (smoothness of preimages under appropriate maps).
    
    To recall the setting: we have a projection $f:X\rightarrow Y$ of definable sets over $A$, and a  point $x\in X$ generic over $A$, such that $y=f(x)$ is generic in $Y$ over $A$. We then have a set $Z$, definable over $B$, so that (1) $y$ is generic in $Z$ over $B$ and (2) the germ of $Z$ at $y$ is contained in the germ of $Y$ at $y$. Our goal is to show that $x\in(f^{-1}(Z\cap Y))^S$.
    
    To do this, one follows exactly the same procedure as in the proof of Claim \ref{C: ez local homeo} to reduce everything d-locally to smooth varieties and smooth morphisms. In the end, one only needs to show the following:
    
    \begin{claim} Let $g:V\rightarrow W$ be a smooth morphism of smooth varieties over $K$, and let $U\subset W$ be a smooth subvariety. Then $g^{-1}(U)$ is smooth.
    \end{claim}
    \begin{claimproof} Recall that smooth morphisms are stable under composition and base change~\cite[Lemma 01VA,01VB]{stacks-project}. Moreover, since $K$ is perfect, a variety over $K$ is smooth if and only if its structure morphism to $\operatorname{Spec}(K)$ is smooth. We proceed to repeatedly use these facts.
    
    First, the map $g^{-1}(U)\rightarrow U$ is isomorphic to the base change of $V\rightarrow W$ by $U\rightarrow W$. Since $V\rightarrow W$ is smooth, so is $g^{-1}(U)\rightarrow U$. But $U\rightarrow\operatorname{Spec}(K)$ is smooth by assumption, so the composition $g^{-1}(U)\rightarrow\operatorname{Spec}(K)$ is smooth, and thus $g^{-1}(U)$ is smooth.
    \end{claimproof}
\end{proof}

We now, automatically, have: 

\begin{lemma}\label{L: ez open mapping} Suppose that whenever $g:V\rightarrow W$ is a quasi-finite morphism of smooth $d$-dimensional varieties over $K$, the induced map $g:V(K)\rightarrow W(K)$ is open. Then $(\mathcal K,\tau)$ has the open mapping property.
\end{lemma}
\begin{proof} Let $f:X\rightarrow Y$ be a finite-to-one definable projection, where $\dim(X)=\dim(Y)$, and let $x\in X^S$ with $y=f(x)\in Y^S$. We need to show that $f$ is open near $x$. But this is automatic: the assumption that $x\in X^S$ and $y\in Y^S$ lets us reduce d-locally to a projection of smooth varieties, and we apply the hypothesis of the lemma.
\end{proof}

\begin{remark} We note, for example, that the hypothesis of Lemma \ref{L: ez open mapping} holds of the complex field with the analytic topology (though technically, this structure is not an \'ez topological field expansion, because it has no invariant basis). Indeed, in this case, the induced map $g:V(K)\rightarrow W(K)$ is a holomorphic map of $d$-dimensional complex manifolds with all fibers discrete, and one can apply a usual open mapping theorem from complex analysis (e.g. \cite{GraRem}, p. 107). This was one of the key geometric properties used in \cite{CasACF0}), and is the inspiration for Lemma \ref{L: ez open mapping}. Indeed, we will see later (Corollary \ref{C: acvf axioms}) that the hypothesis of the lemma holds in ACVF in all characteristics.
\end{remark}

We now move on to differentiability. In fact, quite generally, we have:

\begin{lemma}\label{L: ez differentiable} If $K$ has characteristic zero, then $(\mathcal K,\tau)$ is differentiable.
\end{lemma}
\begin{proof} Given a definable $X$ and $x\in X^S$, the germ of $X^S$ at $x$ agrees with a smooth variety $S(K)$, where $\dim(S)=\dim(X)$. We then set $T_x(X)$ to be the Zariski tangent space $T_x(S)$. All of the required properties of Definition \ref{D: differentiable} follow easily. Note that the weak inverse function theorem follows since \'etale maps are open (Definition \ref{D: ez field}(3)), and Sard's theorem follows since dominant morphisms in characteristic zero are generically smooth. 
\end{proof}

\subsection{The Algebraically Closed Case and Purity of Ramification}

We now address the special case that the underlying field $K$ is algebraically closed. In this case, we will prove a `topologized' purity of ramification statement for morphisms of varieties, and subsequently deduce that $(\mathcal K,\tau)$ has ramification purity in the sense of Definition \ref{D: ramification purity}.

\begin{assumption}
    \textbf{Throughout this subsection, we assume that $K$ is algebraically closed.}
\end{assumption}

The main advantages of the algebraically closed case are summarized by the following fact and corollary:

\begin{fact}\label{F: local dimension of irr var}
    Let $V$ be an irreducible variety over $K$, and let $U\subset V(K)$ be definable, relatively open, and non-empty. Then $U$ is Zariski dense in $V$. In particular, $\dim(U)=\dim(V)$.
\end{fact}
\begin{proof} A proof can be found in~\cite[Proposition 2.1.1]{univ_open}. 
\end{proof}

Fact \ref{F: local dimension of irr var} implies that $\tau$ behaves identically to the Zariski topology on the level of constructible sets:

\begin{corollary}\label{C: closures coincide} Let $V$ be a variety over $K$ and let $X\subset V$ be constructible (equivalently, definable in the pure field language). Then the relative Zariski and $\tau$-closures of $X$ in $V(K)$ coincide.
\end{corollary}

We will also use the following fact about rings:

\begin{proposition}\label{P: fiber product reduced} Let $A\leq B$ be a flat extension of integral domains, and suppose the extension of fraction fields $\operatorname{Frac}(A)\leq\operatorname{Frac}(B)$ is finite and separable. Then the tensor product $B\otimes_AB$ is a reduced ring.
\end{proposition}
\begin{proof}
    For convenience, let $K_A$ and $K_B$ denote the respective fraction fields. The idea of the proof is to use flatness to embed $B\otimes_AB$ into $K_B\otimes_{K_A}K_B$, so that we can assume $A$ and $B$ are already fields. The next claim follows easily from flatness:

    \begin{claim}\label{C: reduced first claim} $B\otimes_AB$ embeds into $K_B\otimes_AK_B$.
    \end{claim}

    Let $S=A-\{0\}$, so that $K_A$ is the localization $S^{-1}A$. Note that $K_B$ is already a $K_A$-module, so $S^{-1}K_B\cong K_B$. Moreover, the following is standard:
    
  \begin{claim}\label{C: reduced second claim}   $S^{-1}(K_B\otimes_AK_B)\cong K_B\otimes_AK_B$.
    \end{claim}
    
    Finally, we conclude:
    
    \begin{claim}\label{C: reduced third claim} $B\otimes_AB$ embeds into $K_B\otimes_{K_A}K_B$. In particular, in proving Proposition \ref{P: fiber product reduced}, we may assume $A$ and $B$ are fields.
    \end{claim}
    \begin{claimproof} Recall~\cite[Lemma 00DL]{stacks-project} that we have a canonical isomorphism $$S^{-1}K_B\otimes_{S^{-1}A}S^{-1}K_B\cong S^{-1}(K_B\otimes_AK_B),$$ which simplifies to $$K_B\otimes_{K_A}K_B\cong S^{-1}(K_B\otimes_AK_B).$$ By Claim \ref{C: reduced second claim}, we moreover obtain an isomorphism $K_B\otimes_{K_A}K_B\cong K_B\otimes_AK_B$. Then combining with Claim \ref{C: reduced first claim} gives an embedding of $B\otimes_AB$ into $K_B\otimes_{K_A}K_B$.
    
    Now to reduce to the case that $A$ and $B$ are fields, simply note that if $B\otimes_AB$ had nilpotents, then so would $K_B\otimes_{K_A}K_B$.
    \end{claimproof}

    Now assume that $A$ and $B$ are fields, so that $A\leq B$ is a finite separable extension. By the primitive element theorem, we have $B=A(b)$ for some $b$. Let $p(x)$ be the minimal polynomial of $b$ over $A$. One then sees easily that $$B\otimes_AB\cong A[x,y]/(p(x),p(y))\cong B[y]/(p(y)).$$ By separability, $p$ factors over $B$ into a product of distinct irreducibles $p_i$. Then by the Chinese Remainder Theorem, our tensor product is now $\Pi_{i}B[y]/p_i(y)$, which is a product of fields, and thus reduced.
\end{proof}

We now move toward our version of purity of ramification. First, recall the classical purity of ramification in algebraic geometry (see, e.g., \cite[Lemma 0EA4]{stacks-project}):

\begin{fact}\label{F: purity of ramification} Let $f:V\rightarrow W$ be a dominant morphism of irreducible varieties over $K$. Assume that $V$ is normal and $W$ is smooth. Then every irreducible component of the non-\'etale locus of $f$ has dimension at least $\dim(V)-1$.
\end{fact}

Recall that a function $f:X\rightarrow Y$ between topological spaces is \textit{topologically unramified} at $x\in X$ if $f$ is injective on some neighborhood of $x$. Our goal is to prove an analog of Fact \ref{F: purity of ramification} for $(\mathcal K,\tau)$, where we replace `\'etale' with `topologically unramified' (for the topology $\tau$). 

The proof will consist, essentially, of two steps. First, we use Fact \ref{L: relative frobenius} to reduce to the case that $f$ is generically smooth. Then we show that, in the generically smooth case, the topological ramification locus coincides with the non-\'etale locus up to a codimension 2 error. Using these two steps, one can reduce our statement to Fact \ref{F: purity of ramification}.

The first of the two steps above is automatic. For the second step, we use the following technical fact:

\begin{proposition}\label{P: etale iff unramified}
    Let $V$ and $W$ be smooth irreducible varieties over $K$, and let $f:V\rightarrow W$ be a dominant, generically smooth morphism. Then the following are equivalent:
    \begin{enumerate}
        \item $f$ is \'etale.
        \item The restricted map $f:V(K)\rightarrow W(K)$ is topologically unramified.
    \end{enumerate}
\end{proposition}
\begin{proof} Throughout, let $Z$ be the fiber product $V\times_WV$ (as schemes). Let $Z_{red}$ be the reduced scheme associated to $Z$ (i.e. the fiber product as varieties). So $Z_{red}(K)$ is the set of $(x,y)\in V(K)^2$ with $f(x)=f(y)$. Moreover, let $\Delta$ be the diagonal (viewed as a subvariety of $Z_{red}$ -- equivalently the image of $V$ in $Z_{red}$). Finally, let $\delta:V\rightarrow Z$ be the diagonal morphism. Since $V$ is a variety (thus reduced), $\delta$ factors as $V\rightarrow Z_{red}\rightarrow Z$.

First suppose $f$ is \'etale. Then $f$ is unramified, so $\delta$ is an open immersion, and thus $\Delta$ is Zariski open in $Z_{red}$. By Corollary \ref{C: closures coincide}, $\Delta(K)$ is $\tau$-open in $Z_{red}(K)$. This is a restatement of the assertion that $f: V(K)\rightarrow W(K)$ is topologically unramified.

Now assume $f: V(K)\rightarrow W(K)$ is topologically unramified. As above, we get that $\Delta(K)$ is $\tau$-open in $Z_{red}(K)$, and thus $\Delta$ is Zariski open in $Z_{red}$. We first conclude:

\begin{lemma}\label{L: miracle flatness} $f$ is quasifinite and flat.
\end{lemma}
\begin{proof}
    Since $\Delta$ is Zariski open in $Z_{red}$, it follows that for each $x\in V(K)$, $\{(x,x)\}$ is open in $(Z_{red})_x(K)=\{y:(x,y)\in Z_{red}(K)\}$. Equivalently, $\{x\}$ is open in the fiber $f^{-1}(f(x))$. Thus $f$ has discrete fibers, and so is quasifinite. Flatness then follows by Miracle Flatness~\cite[Lemma 00R4]{stacks-project} by looking at the induced maps of local rings between points.
\end{proof}

By Lemma~\ref{L: miracle flatness}, we have:

\begin{lemma}\label{L: fiber product reduced schemes} $Z$ is reduced, and thus $Z_{red}\rightarrow Z$ is an isomorphism.
\end{lemma}
\begin{proof} After reducing to affine opens, this reduces exactly to Proposition \ref{P: fiber product reduced}. In particular, the fact that $V=\mathrm{Spec} B\rightarrow \mathrm{Spec} A=W$ is quasifinite and generically smooth guarantees that the field extension $\operatorname{Frac}(A)\leq\operatorname{Frac}(B)$ appearing in Proposition \ref{P: fiber product reduced} is finite separable.
\end{proof}

It remains to show that $f$ is unramified, that is, that $\delta$ is an open immersion. But $\delta$ is automatically a (locally closed) immersion~\cite[Lemma 01KJ]{stacks-project}, namely it factorizes as $j\circ i$ where $i$ is a closed immersion and $j$ is an open immersion. And we have shown it has open image; By the irreducibility and the fact that the only surjective closed immersion to a reduced scheme is an isomorphism, we have that we can assume $i$ is an isomorphism. Thus, by Lemma \ref{L: fiber product reduced schemes}, we are done.
\end{proof}

Finally, we are ready to `topologize' Fact \ref{F: purity of ramification}. In what follows, it will be convenient to extend the notion of local dimension to non-definable sets:

\begin{definition}\label{D: local dimension}
    Let $V$ be a variety over $K$, and $X\subset V(K)$ arbitrary. Let $x\in X$, and $d$ a non-negative integer. We say that $X$ has \textit{local dimension at least $d$ at $x$} if every relative neighborhood of $x$ in $X$ contains a definable set of dimension $d$.
\end{definition}

We now show:

\begin{theorem}\label{T: top purity}
    Let $f:V\rightarrow W$ be a dominant morphism of irreducible varieties over $K$. Assume that $V$ is normal and $W$ is smooth. Let $R$ be the topological ramification locus of $f$. Then $R$ has local dimension at least $\dim(V)-1$ at every point.
\end{theorem}

\begin{proof}
    First note that, by Lemma \ref{L: relative frobenius}, we may assume $f$ is generically smooth (since homeomorphisms preserve the topological ramification locus). In particular, the variety $V'$ obtained by Lemma \ref{L: relative frobenius} is still normal, so the hypotheses of the theorem still hold.

    Thus, moving forward, we assume $f$ is generically smooth. As in Proposition \ref{P: etale iff unramified}, let $Z$ be the fiber product $V\times_WV$ (as schemes), and let $Z_{red}$ be the associated reduced scheme. Then let $\Delta$ be the diagonal subvariety of $Z_{red}$, and $\delta:V\rightarrow Z$ the diagonal morphism, which factors as $V\rightarrow Z_{red}\rightarrow Z$. 

    Toward a proof of the theorem, let $x\in R$, and let $U_1$ be any relative neighborhood of $x$ in $R$. So $U_1=U\cap R$ for some open $U\subset V(K)$. Shrinking $U_1$ and $U$ if necessary, we may assume $U$ is definable (in fact so are $U_1$ and $R$, but this won't be directly used). Our goal is to find a definable $X\subset U_1$ with $\dim(X)=\dim(V)-1$.
    
    First, since $f$ topologically ramifies at $x$, one sees that $(x,x)$ is not in the $\tau$-interior of $\Delta(K)$ in $Z_{red}(K)$. By Corollary \ref{C: closures coincide}, $(x,x)$ is not in the Zariski interior of $\Delta$ in $Z_{red}$. It follows easily that $V\rightarrow Z_{red}$ is not an open map in any neighborhood of $x$. But $Z_{red}\rightarrow Z$ is a Zariski homeomorphism, so $\delta$ is also not an open map in any neighborhood of $x$. Thus, $\delta$ is not locally an open immersion at $x$, and thus $f$ ramifies at $x$~\cite[Lemma 02GE]{stacks-project}. 
    
    So $f$ is not \'etale at $x$. By purity of ramification (Fact \ref{F: purity of ramification}), there is an irreducible closed codimension 1 subvariety $T\subset V$ such that $x\in T$ and $f$ is not \'etale on any point of $T$. Let $V^S$ be the smooth locus of $V$. Let $T'$ be the open subvariety $T\cap V^S\subset T$. Finally, let $X=U\cap T'(K)$. So $X$ is definable because $U$ is. To conclude, we need the following two claims:
    
    \begin{claim} $f$ topologically ramifies at every point of $X$, and thus $X\subset U_1$.
    \end{claim}
    \begin{claimproof} By Proposition \ref{P: etale iff unramified} applied to the restricted map $V^S(K)\rightarrow W(K)$, it suffices to observe (by the choice of $T$) that $f$ is non-\'etale at every point of $T$.
    \end{claimproof}
    
    \begin{claim} $\dim(X)=\dim(V)-1$. 
    \end{claim}
    \begin{claimproof} Since normal varieties are smooth outside codimension 2, we have $$\dim(T-V^S)\leq\dim(V-V^S)\leq\dim(V)-2<\dim(V)-1=\dim(T),$$ and thus $\dim(T')=\dim(T)$. So $T'$ is an open subvariety of $T$ of the same dimension; since $T$ is irreducible, this implies $T'$ is Zariski dense in $T$. Then, by Corollary \ref{C: closures coincide}, $T'(K)$ is also $\tau$-dense in $T(K)$. In particular, since $U\cap T(K)$ is non-empty (witnessed by $x$), it follows that $U\cap T'(K)$ is also non-empty, so that (by Fact \ref{F: local dimension of irr var}) $$\dim(X)=\dim(U\cap T'(K))=\dim(T')=\dim(V)-1.$$
    \end{claimproof}
\end{proof}

Finally, we proceed to the proof of ramification purity. The first step is the following approximation of the full statement:

\begin{lemma}\label{L: purity of multiple intersections 1} Let $\mathcal Y=\{Y_t:t\in U\}$ and $\mathcal Z=\{Z_u:u\in U\}$ be $\mathcal K(A)$-definable families of subsets of a $\mathcal K(A)$-definable set $X$, with graphs $Y\subset X\times T$ and $Z\subset X\times U$. Let $I$ be the set of $(x,t,u)$ with $(x,t)\in Y$ and $(x,u)\in Z$, and let $R$ be the topological ramification locus of $I\rightarrow T\times U$. Suppose $(x_0,t_0,u_0)$ is a weakly generic $(\mathcal Y,\mathcal Z)$-multiple intersection over $A$ which is strongly approximable over $A$. Then $R$ has local dimension at least $d-1$ at $(x_0,t_0,u_0)$, where $d=\dim(Y)+\dim(Z)-\dim(X)$.
\end{lemma}
\begin{proof}
The idea is to reduce all sets and maps in the  statement to smooth varieties and smooth morphisms, and then apply Theorem \ref{T: top purity}. 

First, note that $x_0\in X^S$, $t_0\in T^S$, $u_0\in U^S$, $(x_0,t_0)\in Y^S$, and $(x_0,u_0)\in Z^S$ (since each point is generic in the corresponding set). Moreover, we are only interested in d-local data of $X$, $Y$, $Z$, and $I$ near $(x_0,t_0,u_0)$. So it is harmless to assume each of $X$, $Y$, $Z$, $T$, and $U$ is the set of $K$-points of a smooth variety. It is also harmless to replace $A$ with $\acl(A)$, so we may assume these smooth varieties are all irreducible. For ease of notation, let us call these varieties $X$, $Y$, $Z$, $T$, and $U$, and rename the original definable sets as $X(K)$, $Y(K)$, $Z(K)$, $T(K)$, and $U(K)$.

The assumption that $(x,t)\in Y$ and $x\in X$ are generic over $A$ gives that $Y\rightarrow X$ is dominant. Similarly, $Z\rightarrow X$ is dominant. Now, it is harmless to apply a Frobenius power to each of $T$ and $U$, since we are only interested in topological data. So by Lemma \ref{L: relative frobenius}, we may assume the morphisms $Y\rightarrow X$ and $Z\rightarrow X$ are generically smooth. In particular, by the genericity of $(x,t)\in Y$ and $(x,u)\in Z$, each of $Y\rightarrow X$ and $Z\rightarrow X$ is smooth in a Zariski neighborhood of the relevant point. It is harmless to shrink $Y$ and $Z$ so that $Y\rightarrow X$ and $Z\rightarrow X$ are each smooth.

In this case, $I\rightarrow X$ is the fiber product of the two smooth morphisms $Y\rightarrow X$ and $Z\rightarrow X$. Thus $I\rightarrow X$ is also smooth. Since $X$ is smooth, this implies $I$ is also smooth. Let $I'$ be an irreducible component of $I$ containing $(x_0,t_0,u_0)$. Since we are assuming $A=\acl(A)$, $I'(K)$ is definable over $A$.

We now consider the morphism $I'\rightarrow T\times U$ of irreducible smooth varieties. We claim this morphism is dominant. To see this, note that since $(x_0,t_0,u_0)$ is strongly approximable, $I'(K)$ contains a strongly generic $(\mathcal Y,\mathcal Z)$-intersection $(x,t,u)$ over $A$. By definition of strong genericity, $(t,u)$ is generic in $T(K)\times U(K)$ over $A$. This shows that $I'(K)\rightarrow T(K)\times U(K)$ has generic image, and this implies dominance.

We have now set up Theorem \ref{T: top purity} for the morphism $I'\rightarrow T\times U$. Let $R'$ be the topological ramification locus of $I'\rightarrow T\times U$ (and note that $R'\subset R$). Theorem \ref{T: top purity} tells us that $R'$ has local dimension at least $\dim(I')-1$ at $(x_0,t_0,u_0)$. Since $R'\subset R$, we conclude that $R$ also has local dimension at least $\dim(I')-1$ at $(x_0,t_0,u_0)$. 

Finally, to prove the lemma, we show that $\dim(I')\geq d=\dim(Y)+\dim(Z)-\dim(X)$. For this, we again use our strongly generic intersection $(x,t,u)$. Indeed, an easy computation, using strong genericity, yields that $\dim(xtu/A)=d$. Meanwhile, since $(x,t,u)\in I'(K)$, we have $\dim(xtu/A)\leq\dim(I')$, and thus $\dim(I')\geq d$ as desired. 
\end{proof}

We next deduce a slightly stronger statement:

\begin{lemma}\label{L: purity of multiple intersections 2} Let $\mathcal Y=\{Y_t:t\in T\}$ and $\mathcal Z=\{Z_u:u\in U\}$ be $\mathcal K(A)$-definable families of subset of a $\mathcal K(A)$-definable set $X$, with graphs $Y\subset X\times T$ and $Z\subset X\times U$. Let $(x_0,t_0,u_0)$ be a generic $(\mathcal Y,\mathcal Z)$-multiple intersection over $A$ which is strongly approximable over $A$. Then there is a generic $(\mathcal Y,\mathcal Z)$-multiple intersection $(x,t,u)$ satisfying $\dim(xtu/A)\geq d-1$, where $d=\dim(Y)+\dim(Z)-\dim(X)$.
\end{lemma}
\begin{proof}
    Let $I$ be the set of $(x,t,u)$ with $(x,t)\in Y$ and $(x,u)\in Z$, and let $R$ be the topological ramification locus of $I\rightarrow T\times U$. We are asked to find $(x,t,u)$ satisfying:

    \begin{enumerate}
        \item $(x,t,u)\in R$.
        \item $x$, $t$, $u$, $(x,t)$, and $(x,u)$ are generic in $X$, $T$, $U$, $Y$, and $Z$ over $A$, respectively.
        \item $\dim(xtu/A)\geq d-1$.
    \end{enumerate}

    Since $\mathcal K$ is $\aleph_1$-saturated, it suffices to realize any finite fragment of (1)-(3) above. In particular, we will find $(x,t,u)$ realizing (1), (3), and a finite fragment of (2). 
    
    Let us briefly elaborate. For the purposes of this proof, given $\mathcal K$-definable sets $D_1\subset D_2$, let us call $D_1$ \textit{large} in $D_2$ if $\dim(D_2-D_1)<\dim(D_2)$. It is easy to see that, if $D$ is $A$-definable, then the set of $A$-generic elements of $D$ is the intersection of all $A$-definable large subsets of $D$. In particular, to show some partial type is consistent with the $A$-generic locus of $D$, one can show that it is consistent with each $A$-definable large subset of $D$. So our plan is to satisfy (1), (3), and an approximation of (2) replacing each $\mathcal K(A)$-definable set in the statement with a $\mathcal K(A)$-definable large subset.

    Now let us give the details. Suppose $X'$, $Y'$, $Z'$, $T'$, and $U'$ are any $\mathcal K(A)$-definable large subsets of $X$, $Y$, $Z$, $T$, and $U$, respectively. Then, precisely, it suffices to find $(x,t,u)$ satisfying (1), (2'), and (3), where (2') is the conjunction of $x\in X'$, $(x,t)\in Y'$, $(x,u)\in Z'$, $t\in T'$, and $u\in U'$.

    We now construct such $(x,t,u)$. First, let $I'$ be the set of $(x,t,u)\in X'\times T'\times U'$ with $(x,t)\in Y'$ and $(x,u)\in Z'$. Let $R'$ be the topological ramification locus of the projection $I'\rightarrow T'\times U'$. Note that any $(x,t,u)\in R'$ automatically satisfies (1) and (2'). Thus, to additionally satisfy (3), it suffices to find an element $(x,t,u)\in R'$ with $\dim(xtu/A)\geq d-1$.
    
    Now by Lemma \ref{L: germ at generic}, the sets $X$ and $X'$ have the same germ at $x_0$ (that is, by genericity of $x_0$, they both realize the germ of $\tp_{\mathcal K}(x_0/A)$); and the analogous statements hold for $Y$, $Z$, $T$, and $U$. Since they are defined analogously, it follows that $I$ and $I'$ have the same germ at $(x_0,t_0,u_0)$ and thus that $R$ and $R'$ have the same germ at $(x_0,t_0,u_0)$. So $R'$ contains some relative neighborhood of $(x_0,t_0,u_0)$ in $R$. But by Lemma \ref{L: purity of multiple intersections 1}, every relative neighborhood of $(x_0,t_0,u_0)$ in $R$ contains a definable set of dimension $d-1$. Thus, in particular, $R'$ contains a definable set -- say $D$ -- of dimension $d-1$. Let $D$ be definable over $B\supset A$, and let $(x,t,u)\in D$ be generic over $B$. Then $(x,t,u)\in R'$, and $$\dim(xtu/A)\geq\dim(xtu/B)=\dim(D)=d-1,$$ as desired. 
\end{proof}

\begin{remark} Note that in the above lemma, we cannot assume $R$ and $R'$ are definable (this would require a uniformly definable basis for $\tau$, and we do not assume uniformity). This is the reason we needed the more general notion of local dimension in Definition \ref{D: local dimension}.
\end{remark}

Finally, we show:

\begin{theorem}\label{T: acvf purity} $(\mathcal K,\tau)$ has ramification purity.
\end{theorem}
\begin{proof}
    Let $\mathcal Y=\{Y_t:t\in T\}$ and $\mathcal Z=\{Z_u:u\in U\}$ be $\mathcal K(A)$-definable subsets of a $\mathcal K(A)$-definable set $X$, with graphs $Y\subset X\times T$ and $Z\subset X\times U$. Let $(x_0,t_0,u_0)$ be a generic $(\mathcal Y,\mathcal Z)$-multiple intersection over $A$ which is strongly approximable over $A$. We want to find a generic $(\mathcal Y,\mathcal Z)$-multiple intersection over $A$ which has codimension at most 1 over $A$.

    By Lemma \ref{L: purity of multiple intersections 2}, there is a generic $(\mathcal Y,\mathcal Z)$-multiple intersection $(x,t,u)$ over $A$ with $\dim(xtu/A)\geq d-1$, where $d=\dim(Y)+\dim(Z)-\dim(X)$. We claim that $\operatorname{codim}_A(xtu)\leq 1$, which will complete the proof. Indeed, we have 

    \begin{itemize}
        \item $\dim(x/A)=\dim(X)$
        \item $\dim(xt/A)=\dim(Y)$,
    \end{itemize}

    and thus by additivity, $\dim(t/Ax)=\dim(Y)-\dim(X)$. Similarly, we get $\dim(u/Ax)=\dim(Z)-\dim(X)$.
    
    Meanwhile, we have 
    
    \begin{itemize}
        \item $\dim(x/A)=\dim(X)$
        \item $\dim(xtu/A)\geq d-1$,    \end{itemize}
        
        and thus by additivity, $\dim(tu/Ax)\geq d-1-\dim(X)$. So, putting everything together, we have $$\operatorname{codim}_A(xtu)=\dim(t/Ax)+\dim(u/Ax)-\dim(tu/Ax)$$ $$\leq\dim(Y)-\dim(X)+\dim(Z)-\dim(X)-(d-1-\dim(X))$$ $$\dim(Y)+\dim(Z)-\dim(X)-(d-1)=1.$$
\end{proof}

Let us now drop the ambient assumptions on $(\mathcal K,\tau)$. We end 
Section 9 by collecting our results for ACVF and pure \'ez fields:

\begin{corollary}\label{C: acvf axioms} Every $\aleph_1$-saturated algebraically closed valued field, equipped with the valuation topology, is a Hausdorff geometric structure with the open mapping property and ramification purity.
\end{corollary}
\begin{proof} Let $(\mathcal K,\tau)$ be such a field. By Lemma \ref{L: acvf ez}, $(\mathcal K,\tau)$ is a Hausdorff geometric structure. By Theorem \ref{T: acvf purity}, $(\mathcal K,\tau)$ has ramification purity. Finally, to prove the open mapping property, it suffices to verify the hypothesis of Lemma \ref{L: ez open mapping}. So, let $f:V\rightarrow W$ be a quasi-finite morphism of smooth $d$-dimensional $K$-varieties. Then by \cite[Lemma 0F32]{stacks-project}, $f$ is universally open, and thus by Fact \ref{F: open mapping}, $V(K)\rightarrow W(K)$ is open.
\end{proof}

\begin{corollary} Let $(K,+,\cdot)$ be an $\aleph_1$-saturated \'ez field, whose \'etale open topology is induced by a field topology. Then $(K,+,\cdot)$, equipped with the \'etale open topology, is a Hausdorff geometric structure, and if $K$ has characteristic zero then $(K,\tau)$ is differentiable.
\end{corollary}

\begin{proof} By Theorem \ref{T: ez fields} and Lemma \ref{L: ez differentiable}.
\end{proof}

\section{Definable Slopes in ACVF}\label{S: slopes}

Our goal now is to show that algebraically closed valued fields have definable slopes satisfying TIMI (Definition \ref{D: definable slopes} and Definition \ref{D: TIMI}). \textbf{Throughout this section, we fix an $\aleph_1$-saturated algebraically closed valued field $(K,v)$}. In particular, unlike the rest of the paper to this point, our topology now has a uniformly definable basis. This means that we can definably speak about germs of functions at a point (because the equivalence relation of having the same germ is definable). We will do this throughout.

Otherwise, we use very little from the theory ACVF, so we do not elaborate here. Interested readers may refer to \cite[\S 2.1]{HaHrMac1}, \cite[\S 2.1]{HrKa} and references therein for more background. Section 2.1 of the more recent \cite{HaOnPi} contains also some background (and relevant references) on the analytic theory of functions definable in complete models of ACVF that is needed in a couple of points below. 

\subsection{Taylor Groupoids} We will work with approximate versions of TIMI, successively generalizing until we get the full statement. To streamline the presentation, we work with ind-definable sets. For us, an \textit{$A$-ind-definable set} is a countable union of $A$-definable sets; a set is \textit{ind-definable} if it is $A$-ind-definable for some $A$; and a function is ($A$)-ind-definable if its graph is. Thus, an $A$-ind-definable group is a group whose underlying set, composition operation, and inverse operation, are all $A$-indefinable.

We want to work more generally with ind-definable groupoids:

\begin{definition}
    A $\emptyset$-\textit{ind-definable groupoid} is a groupoid $\mathcal C$ such that (i) the set of objects of $\mathcal C$, (ii) the set of morphisms of $\mathcal C$, (iii) the composition operation on morphisms of $\mathcal C$, and (iv) the inverse operation on morphisms of $\mathcal C$, are all $\emptyset$-ind-definable. 
\end{definition}

\begin{example} Define the groupoid $\mathcal{LDH}$ (local definable homeomorphisms) whose objects are the elements of $K$, and whose morphisms $x\rightarrow y$ are germs of definable homeomorphisms between neighborhoods of $x$ and $y$. Then $\mathcal{LDH}$ is a $\emptyset$-ind-definable groupoid.
\end{example}

Our first goal is to introduce, in the language of ind-definable groupoids, an abstract notion of Taylor series.

\begin{definition}\label{first order definable slopes} A \textit{weak Taylor groupoid} consists of the following data and requirements:
    \begin{enumerate}
        \item A $\emptyset$-ind-definable groupoid $\mathcal C$, which is a sub-groupoid of $\mathcal LDH$.
        \item A $\emptyset$-ind-definable group $G$ contained in $\operatorname{dcl}(\emptyset)$. 
        \item A $\emptyset$-ind-definable homomorphism $a$ from morphisms in $\mathcal C$ to $G$. This means that if $g\circ f=h$ as morphisms in $\mathcal C$, then $a(g)a(f)=a(h)$ in $G$.
        \item For each $n\geq 1$, a $\emptyset$-ind-definable map $f\mapsto c_n(f)$ from morphisms in $\mathcal C$ to $K$, such that $c_1(f)\neq 0$ for all $f$.
    \end{enumerate}
    \end{definition}

    For ease of notation, we often denote a weak Taylor groupoid by the underlying groupoid $\mathcal C$, omitting the additional data.

    \begin{definition}
        Let $\mathcal C$ be a weak Taylor groupoid, and let $f:x\rightarrow y$ be a morphism in $\mathcal C$. For each $n\geq 0$, we define the $n$\textit{th truncation} of $f$, dented $T_n(f)$, to be the tuple $(x,y,a(f),c_1(f),...,c_n(f))$ (if $n=0$ this just means $(x,y,a(f))$).
    \end{definition}

    \begin{definition}\label{D: taylor groupoid}
        Let $\mathcal C$ be a weak Taylor groupoid. Then $\mathcal C$ is a \textit{Taylor groupoid} if the following hold:
        \begin{enumerate}
            \item If $g\circ f=h$ as morphisms in $\mathcal C$, then for each $n$, the $n$th truncation of $h$ is determined by the $n$th truncations of $f$ and $g$.
            \item If $f$ is a morphism in $\mathcal C$, then $f$ is determined by its sequence of $n$th truncations for $n\geq 1$.
            \item For fixed $n\geq 1$, every $(x,y,a,c_1,...,c_n)\in K^2\times G\times K^n$ with $c_1\neq 0$ is the $n$th truncation of a morphism in $\mathcal C$.
        \end{enumerate}
    \end{definition}

    Taylor groupoids are very close to asserting Taylor's Theorem about a certain class of functions in an analytic context: each function is definably assigned a countable sequence of `coefficients' at each point, which determine the function in a neighborhood of the point, and these coefficients carry definable composition and inverse operations which are well-defined on finite truncations. The main difference, and the reason for the term `generalized', is the group $G$. One should think of elements of $G$ as powers of the Frobenius map, which are necessary to cover all the functions we need in positive characteristic.

    Since the various properties required in the two notions are quite similar, it is not hard to see that a `big enough' Taylor groupoid implies that $(K,v)$ has definable slopes: 
    
    \begin{lemma}\label{L: groupoid induces slopes}
        Let $\mathcal C$ be a Taylor groupoid. Assume that the set of objects of $\mathcal C$ is $K$, and that every basic invertible arc (in the sense of Definition \ref{D: invertible arc}) is a morphism in $\mathcal C$. Then $(K,v)$ has definable slopes.
    \end{lemma}
    \begin{proof} Definition \ref{D: taylor groupoid}(1) implies that for each $n$, multiplication and inverse in $\mathcal C$ induce a groupoid structure on the $n$th truncations of all morphisms in $\mathcal C$. One then defines $n$-slopes as $n$th truncations, and the axioms of definable slopes follow easily. We point out that the uniform definability of slopes, and the definability of composition and inverse of $n$-slopes, follow from the $\emptyset$-ind-definability of truncation, composition, and inverse in $\mathcal C$; while the fact that $\mathcal C$ is the inverse limit of the truncation categories follows from compactness and Definition \ref{D: taylor groupoid} (2) and (3).
    \end{proof}

    Note that the groupoids and functors constructed in Lemma \ref{L: groupoid induces slopes} are canonically determined by $\mathcal C$. Thus we define:

    \begin{definition}
        If $\mathcal C$ is a Taylor groupoid as in Lemma \ref{L: groupoid induces slopes}, we will call the system of groupoids and truncation maps constructed in Lemma \ref{L: groupoid induces slopes} the \textit{definable slopes on $(K,v)$ induced by $\mathcal C$}.
    \end{definition}

    \subsection{Etale Functions at 0}
    
    Our goal is to construct increasingly large Taylor groupoids until we reach one as in Lemma \ref{L: groupoid induces slopes}. We start with \textit{\'etale functions at $(0,0)$}.

    \begin{definition}\label{D: etale function}
        Let $f:0\rightarrow 0$ be a morphism in $\mathcal{LDH}$ (i.e. a germ of a definable homeomorphism of neighborhoods of 0).
        \begin{enumerate}
            \item We say that $f$ is a \textit{basic \'etale function} at $(0,0)$ if there is a polynomial $P(x,y)\in K[x,y]$ such that $P(0,0)=0$, both partial derivatives of $P$ are non-zero at $(0,0)$, and $f$ is given by the restriction of $P(x,y)=0$ to a neighborhood of $(0,0)$.
            \item We say that $f$ is an \textit{\'etale function} at $(0,0)$ if it is a composition of finitely many basic \'etale functions at $(0,0)$.
            \end{enumerate}
    \end{definition}

    It is evident from the definition that \'etale functions are closed under composition and inverse, and thus define a $\emptyset$-ind-definable groupoid with 0 as its only object.

    \begin{notation}
        We let $\mathcal{EFO}$ (\'etale functions at the origin) denote the $\emptyset$-ind-definable groupoid of \'etale functions at $(0,0)$.
    \end{notation}

    Our first goal is to extend $\mathcal{EFO}$ to a Taylor groupoid, with the group $G$ interpreted as the trivial group.

    \begin{lemma}\label{L: approximations of etale functions} Let $f$ be an \'etale function, and fix $n\geq 1$. Then there is exactly one polynomial $P(x)\in K[x]$ of degree at most $n$ such that $\lim_{x\rightarrow 0}\frac{P(x)-f(x)}{x^n}=0$.
    \end{lemma}
    \begin{proof} The lemma can be expressed using infinitely many first order sentences. In particular, it suffices to carry out the proof in a complete model. So let $\mathbb K$ be a complete model of the same characteristics as $K$. It follows from the implicit function theorem that every basic \'etale function at $(0,0)$ is analytic, i.e. is given by a convergent power series in a neighborhood of 0. Thus, so is $f$. Write $f(x)=\sum_{i=1}^{\infty}c_ix^i$. Then the desired polynomial $P$ is $\sum_{i=1}^nc_ix^i$.
    \end{proof}

    \begin{notation} Let $f$ be an \'etale function at $(0,0)$, and $n\geq 1$. We define the $n$th \textit{coefficient} of $f$, denoted $c_n(f)$, to be the coefficient of $x^n$ in the polynomial provided in Lemma \ref{L: approximations of etale functions}.
    \end{notation}

    We now show:

    \begin{proposition}
        With the maps $c_n$ defined above, the groupoid $\mathcal{EFO}$ becomes a Taylor groupoid.
    \end{proposition}
    \begin{proof} It is clear from the statement of Lemma \ref{L: approximations of etale functions} that the maps $c_n$ are $\emptyset$-ind-definable, so (interpreting the group $G$ as the trivial group) we obtain a weak Taylor groupoid. We now verify (1)-(3) in Definition \ref{D: taylor groupoid}:
    \begin{enumerate}
        \item As in the proof of Lemma \ref{L: approximations of etale functions}, the statement is first-order, so we can work in a complete model. Now let $g\circ f=h$ be \'etale functions at $(0,0)$. Then the $n$th truncations of $f,g,h$ are just the truncations of their power series expansions at $(0,0)$. In other words, (1) just says that composition and inverse are well-defined on truncated polynomials. This is well-known and easy to check.
        \item Suppose $f$ and $g$ have the same $n$th truncations for all $n$. It follows that, if we let $h=f-g$, then $\lim_{x\rightarrow 0}\frac{h(x)}{x^n}=0$ for all $n$. If $h$ is identically 0, we are done. Otherwise, in a small enough ball, we may assume that all fibers of $h$ have size at most some integer $d$, and that $h(x)=0$ only holds for $x=0$. Now let $\Gamma$ be the value group, and for $r\in\Gamma$, define $s(r)$ to be the supremum of all $v(h(x))$ for $x$ with $v(x)=r$. If $s(r)=\infty$ for some $r$, it follows by the Swiss Cheese decomposition that the image under $h$ of the elements of value $r$ contains all non-zero points in some open ball at 0. In particular, since $h$ is at most $d$-to-one, there are only finitely many such $r$. So, passing to a smaller neighborhood of 0, we may assume $s(r)<\infty$ for all $r$. But for each $n$, since $\lim_{x\rightarrow 0}\frac{h(x)}{x^n}=0$, it follows that for sufficiently large $r\in\Gamma$ we have $s(r)\geq nr$. In particular, $s:\Gamma\rightarrow\Gamma$ grows faster than any linear function. This contradicts that $\Gamma$ is a pure divisible ordered abelian group.
        \item Equivalently, we want to show that every $(c_1,...,c_n)\in K^n$ with $c_1\neq 0$ is the $n$th truncation of an \'etale function. This is trivial, since $(c_1,...,c_n)$ is the $n$th truncation of the polynomial $\sum_{i=1}^nc_ix^i$, which is clearly an \'etale function since $c_1\neq 0$.
    \end{enumerate}
            \end{proof}

    \subsection{Generalized Etale Functions at 0}

    We next generalize $\mathcal{EFO}$ to include Frobenius powers. 
   
    \begin{definition}
        Let $G$ denote the \textit{Frobenius group} of maps $K\rightarrow K$, defined as follows: if $K$ has characteristic zero, $G=\{\mathrm{id}\}$. If $K$ has characteristic $p>0$, $G$ is the group of maps $x\mapsto x^{p^n}$. 
    \end{definition}

    We view $G$ as a $\emptyset$-ind-definable group acting $\emptyset$-ind-definably on $K$. The following are the key properties we need, which are evident from the definition:

    \begin{fact}\label{F: frob group properties} For the group $G$ defined above:
    \begin{enumerate}
        \item Every element of $G$ is a $\emptyset$-definable automorphism of $(K,v)$, so in particular also a homeomorphism.
        \item If $a\in G$ is an \'etale function, then $a=\mathrm{id}$.
    \end{enumerate}
    \end{fact}

    \begin{definition}
        A \textit{generalized \'etale function at} $(0,0)$ is a morphism at $(0,0)$ in $\mathcal{LDH}$ of the form $f\circ a$, where $a\in G$ and $f$ is an \'etale function at $(0,0)$.
    \end{definition}

    Since $G$ is $\emptyset$-ind-definable, the class of generalized \'etale functions at $(0,0)$ is also $\emptyset$-ind-definable. We want to show that this class forms a Taylor groupoid in a natural way. Everything will follow from the ensuing three lemmas:
    
   \begin{lemma}\label{L: uniqueness of frobenius power} Each generalized \'etale function at $(0,0)$ has exactly one expression as $f\circ a$, where $a\in G$ and $f$ is an \'etale function at $(0,0)$.
    \end{lemma}
    \begin{proof}
        Assume that $f\circ a=g\circ b$ are two such expressions. Simplifying yields $g^{-1}\circ f=b\circ a ^{-1}$. But $g^{-1}\circ f$ is an \'etale function at $(0,0)$, so then $b\circ a^{-1}\in G$ is also an \'etale function at $(0,0)$. By Fact \ref{F: frob group properties}, $b\circ a^{-1}=\mathrm{id}$. Thus $a=b$, and thus $f=g$. 
    \end{proof}
    
     \begin{notation}
        Let $a\in G$, and let $f$ be an \'etale function at $(0,0)$. By $f^a$ we mean the function whose graph is the image of $f$ under the automorphism $a$: that is, the map sending $a(x)$ to $a(y)$ whenever $f(x)=y$.
    \end{notation}

     \begin{lemma}\label{L: frob of function} Let $f$ be an \'etale function at $(0,0)$, and let $a\in G$. Then:
    \begin{enumerate}
        \item $f^a$ is an \'etale function at $(0,0)$.
        \item The truncations of $f^a$ are the images of the truncations of $f$ under $a$.
        \item $a\circ f=f^a\circ a$.
    \end{enumerate}
    \end{lemma}
    \begin{proof} Both (1) and (2) follow from the fact that $a$ is an automorphism, since the class of \'etale functions at $(0,0)$ and their truncation maps are $\emptyset$-ind-definable. (3) is immediate from the definition of $f^a$.
    \end{proof}

    \begin{lemma}\label{L: composition with frob} Let $f$ and $g$ be \'etale functions at $(0,0)$, and let $a,b\in G$. 
    \begin{enumerate}
        \item $(g\circ b)\circ(f\circ a)=(g\circ f^b)\circ(b\circ a)$.
        \item $(f\circ a)^{-1}=f^{a^{-1}}\circ a^{-1}$.
        \item In particular, the set of generalized \'etale functions is closed under composition and inverse, so forms a $\emptyset$-ind-definable groupoid.
    \end{enumerate}
    \end{lemma}
    \begin{proof} (1) and (2) follow from Lemma \ref{L: frob of function}(3). (3) follows from (1) and (2).
    \end{proof}

    Now let $\mathcal{GEFO}$ be the $\emptyset$-ind-definable groupoid of generalized \'etale functions at $(0,0)$. If $f\circ a$ is a morphism in $\mathcal{GEFO}$, we define its $n$th truncation to be $(0,0,a,c_1,...,c_n)$, where $(0,0,\mathrm{id},c_1,...,c_n)$ is the $n$th truncation of $f$. By Lemma \ref{L: uniqueness of frobenius power}, this is well-defined.

    \begin{proposition} With the Frobenius group $G$ and the truncation maps defined above, $\mathcal{GEFO}$ is a Taylor groupoid.
    \end{proposition}
    \begin{proof}
        The $\emptyset$-ind-definability of the truncation maps follows from Lemma \ref{L: uniqueness of frobenius power}, the $\emptyset$-ind-definability of $G$, and the $\emptyset$-ind-definability of truncation in $\mathcal{EFO}$.
        
        Next, it follows from Lemma \ref{L: composition with frob}(1) that composition in $\mathcal{GEFO}$ induces composition in $G$. We have now shown that $\mathcal{GEFO}$ is a weak Taylor groupoid. We proceed to verify (1)-(3) in the definition of Taylor groupoids (Definition \ref{D: taylor groupoid}:
        
        \begin{enumerate}
            \item It follows from Lemmas \ref{L: composition with frob} and \ref{L: frob of function}(2) that composition and inverse are well-defined on truncations. 
            \item Suppose $f\circ a$ and $g\circ b$ are morphisms in $\mathcal{GEFO}$ with the same $n$th truncations for all $n$. By definition of the truncations in $\mathcal{GEFO}$, this means that $a=b$ and $f$ and $g$ have the same $n$th truncations for all $n$. Since $\mathcal{EFO}$ is a Taylor groupoid, this implies $f=g$.
            \item For any $a\in G$ and any $c_1,...,c_n\in K$ with $c_1\neq 0$, we want to find a morphism in $\mathcal{GEFO}$ with $n$th truncation $(0,0,a,c_1,...,c_n)$. Let $P(x)=\sum_{i=1}^nc_ix^i$. Then the desired truncation is achieved by $P\circ a$.
        \end{enumerate}
    \end{proof}

    \subsection{Generalized \'Etale Functions}

    Finally, we now extend $\mathcal{GEFO}$ to include maps at arbitrary points in $K^2$, not just the origin.

    \begin{definition}
        Let $(x_0,y_0)\in K^2$. A \textit{generalized \'etale function at} $(x_0,y_0)$ is a morphism at $(x_0,y_0)$ in $\mathcal{LDH}$ of them form $f(x-x_0)+y_0$, where $f$ is a morphism in $\mathcal{GEFO}$. A \textit{generalized \'etale function} is a generalized \'etale function at some point.
    \end{definition}

    It is clear that each generalized \'etale function at $(x_0,y_0)$ can be expressed as $f(x-x_0)+y_0$ for a unique morphism $f$ from $\mathcal{GEFO}$.
    
    \begin{notation}
        We denote the generalized \'etale function $f(x-x_0)+y_0$ by $f_{x_0}^{y_0}$.
    \end{notation}

    The following properties are immediate:
    \begin{lemma} Let $f_x^y$ and $g_y^z$ be generalized \'etale functions.
    \begin{enumerate}
        \item $g_y^z\circ f_x^y=(g\circ f)_x^z$.
        \item $(f_x^y)^{-1}=(f^{-1})_y^x$.
        \item In particular, the generalized \'etale functions form the morphisms of a $\emptyset$-ind-definable groupoid whose set of objects is $K$.
    \end{enumerate}
    \end{lemma}

    \begin{notation}
        Let $\mathcal{GEF}$ be the $\emptyset$-definable groupoid of generalized \'etale functions.
    \end{notation}

    We now point out that $\mathcal{GEFO}$ extends canonically to a Taylor groupoid on $\mathcal{GEF}$.

    \begin{definition}
        If $f_x^y$ is a generalized \'etale function and $n\geq 0$, we define the $n$\textit{th truncation} of $f_x^y$ to be $(x,y,a,c_1,...,c_n)$, where $(0,0,a,c_1,...,c_n)$ is the $n$th truncation of $f$.
    \end{definition}

    This time, there is really nothing to check. We leave the details to the reader, and just state:

    \begin{proposition}
        With the group $G$ and the truncation maps given above, $\mathcal{GEF}$ forms a Taylor groupoid.
    \end{proposition}

    We conclude:

    \begin{theorem}\label{T: acvf definable slopes} $(K,v)$ has definable slopes, induced as in Lemma \ref{L: groupoid induces slopes} by the Taylor groupoid $\mathcal{GEF}$.
    \end{theorem}
    \begin{proof} By Lemma \ref{L: groupoid induces slopes}, we only need to show that every basic invertible arc is a morphism in $\mathcal{GEF}$. So, let $X\subset K^2$ be $A$-definable of dimension 1, with both projections $X\rightarrow K$ finite-to-one. Let $(x_0,y_0)\in K^2$ be generic over $A$. By genericity, we may assume that $X$ is defined in a neighborhood of $(x_0,y_0)$ by an irreducible polynomial equation $P(x,y)=0$ with coefficients in $\operatorname{acl}(A)$.
    
    Note that the map $x\mapsto x^p$ defines a generalized \'etale function at every point. Indeed, for any $a$, we have $x^p=(x-a)^p+a^p$, which is clearly a generalized \'etale function at $(a,a^p)$. In particular, we may freely replace $X$ with its image under applying any Frobenius power to either coordinate. Since the Frobenius is $\emptyset$-definable, this will send $(x_0,y_0)$ to a generic point of the resulting set; and since the Frobenius is a homeomorphism, the local behavior of $X$ is preserved.
    
    Now, by the above paragraph, we may assume that $P(x,y)$ does not belong to either $K[x^p,y]$ or $K[x,y^p]$. Let $V$ be the variety defined by $P$. It follows that each projection of $V$ to $K$ induces a separable extension of function fields, which implies that each such projection is generically \'etale. Since $(x_0,y_0)$ is generic in $X$, each projection $V\rightarrow K$ is \'etale at $(x_0,y_0)$. In particular, the partial derivatives of $P$ at $(x_0,y_0)$ are both non-zero. 
    
    Now, at this point, it is harmless to translate $X$ and assume $x_0=y_0=0$. Then since both partial derivatives of $P$ are non-zero at $(0,0)$, $P$ defines a basic \'etale function at $(0,0)$.
    \end{proof}

\subsection{TIMI in ACVF}

It remains to show that with respect to the notion of definable slopes we have constructed, ACVF satisfies TIMI. We need some preliminaries. First, we note the following, which is clear:

\begin{lemma}\label{polynomial determines function} Let $P(x,y,\overline z)$ be a polynomial over $K$ in $m+2$ variables, and fix $(x_0,y_0,\overline z_0)$ with $P(x_0,y_0,\overline z_0)=0$. Assume the partial derivatives of $P$ with respect to $x$ and $y$ are non-zero at $(x_0,y_0,\overline z_0)$. Then the set $\{(x,y):P(x,y,\overline z_0)=0\}\subset K^2$ restricts to a generalized \'etale function $f$ at $(x_0,y_0)$, such that $a(f)$ (the associated element of the Frobenius group $G$) is the identity. 
\end{lemma}

\begin{notation} Let $P(x,y,\overline z)$ be a polynomial whose graph near some point $(x_0,y_0,\overline z_0)$ gives a generalized \'etale function at $(x_0,y_0)$ in the sense of the above lemma. We denote this generalized \'etale function by $P\restriction(x_0,y_0,\overline z_0)$.
\end{notation}

Note that if the partial derivatives are non-zero at $(x_0,y_0,\overline z_0)$ as in Lemma \ref{polynomial determines function}, then they are non-zero in a neighborhood of $(x_0,y_0,\overline z_0)$. Thus, the notation $P\restriction(x_0,y_0,\overline z_0)$ is well-defined in a neighborhood of $(x_0,y_0,\overline z_0)$.

Now our main tool for proving TIMI is the following:

\begin{lemma}\label{definable TIMI} Let $P(x,y,\overline z)$ and $Q(x,y,\overline w)$ be polynomials in $m_P+2$ and $m_Q+2$ variables, respectively, and let $I\subset K^2\times K^{m_P}\times K^{m_Q}$ be the set of $(x,y,\overline z,\overline w)$ with $P(x,y,\overline z)=Q(x,y,\overline w)=0$. Fix $(x_0,y_0,\overline z_0,\overline w_0)\in I$ so that all partial derivatives of $P$ and $Q$ in the $x$ and $y$ variables are non-zero at $(x_0,y_0,\overline z_0,\overline w_0)$. Assume for some $n$ that the $n$th truncations of $P\restriction(x_0,y_0,\overline w_0)$ and $Q\restriction(x_0,y_0,\overline z_0)$ coincide, but every neighborhood of $(x_0,y_0,\overline z_0,\overline w_0)$ contains some $(x,y,\overline z,\overline w)$ so that the $n$th truncations of $P\restriction(x,y,\overline w)$ and $Q\restriction(x,y,\overline z)$ do not coincide. Then the projection $I\rightarrow K^{m_P}\times K^{m_Q}$ is not injective in any neighborhood of $(x_0,y_0,\overline z_0,\overline w_0)$.
\end{lemma}
\begin{proof}
    The lemma can be expressed by infinitely many first-order sentences, so it suffices to carry out the proof in a complete model. In this case, by the Implicit Function Theorem, the graph of $P(x,y,\overline z)=0$ is, in a neighborhood of $(x_0,y_0,\overline z_0)$, the graph of an analytic function $(x,\overline z)\mapsto y_P(x,\overline z)$. Similarly, the graph of $Q$ is given locally by an analytic function $y_Q(x,\overline w)$. Let $g$ be the analytic function $(x,\overline w,\overline z)\mapsto y_P(x,\overline z)-y_Q(x,\overline w)$. For given $\overline z,\overline w$, we let $g_{\overline z\overline w}$ be the analytic function in one variable given by restricting $g$ to tuples with the values $\overline z$ and $\overline w$ in the $K^{m_P}$ and $K^{m_Q}$ coordinates.

    Now the fact that the $n$th truncations of $P\restriction(x_0,y_0,\overline w_0)$ and $Q\restriction(x_0,y_0,\overline z_0)$ coincide, equivalently stated, gives that $x_0$ is a root of multiplicity $n$ of $g_{\overline z_0\overline w_0}$. By the continuity of roots (this follows easily, say, from \cite{Cher2}), the number of roots of $g_{zw}$ including multiplicity, in any sufficiently small neighborhood $U$ of $(\overline x_0,\overline y_0)$, is constant in a neighborhood of $(\overline z_0,\overline w_0)$. But if $(x,y,\overline z,\overline w)\in I$ is sufficiently close to $(x_0,y_0,\overline z_0,\overline w_0)$ such that the $n$th truncations do not coincide on $(x,y,\overline z,\overline w)$, then $(x,y)$ is a root of multiplicity less than $n$ of $g_{zw}$, implying that $g_{zw}$ has at least one other root in $U$. Such a root corresponds to another point $(x',y',\overline z,\overline w)\in I$, violating local injectivity of $I\rightarrow K^{m_P}\times K^{m_Q}$.
\end{proof}

We now conclude:

\begin{theorem}\label{T: ACVF TIMI} With the definable slopes induced by $\mathcal{GEF}$, $(K,v)$ satisfies TIMI.
\end{theorem}
\begin{proof}
    We are given two $A$-definable families of correspondences in $K$, $\mathcal X=\{X_t:t\in T\}$ and $\mathcal Y=\{Y_u:u\in U\}$ (with graphs $X\subset K^2\times T$ and $Y\subset K^2\times U$), and a generic $(\mathcal X,\mathcal Y)$-tangency $(x_0,y_0,t_0,u_0,n)$ over $A$. We want to show that $(x_0,y_0,t_0,u_0)$ is a topological ramification point of the projection $I\rightarrow T\times U$, where $I$ is the set of $(x,y,t,u)$ with $(x,y)\in X_t\cap Y_u$.

    Now, $X_{t_0}$ and $Y_{u_0}$ define generalized \'etale functions at $(x_0,y_0)$. It is harmless to apply any Frobenius power to any copy of $K$, since the Frobenius is $\emptyset$-definable (so doesn't change which points are generic), is a homeomorphism (so doesn't change the topological ramification locus), and is a generalized \'etale function at every point (so doesn't change the relation of two curves having the same $n$-slope). In particular, we may assume the Frobenius map associated to the generalized \'etale functions $X_{t_0}$ and $Y_{u_0}$ is the identity.

    It is also harmless to restrict to any $A$-definable neighborhood of $(x_0,y_0,t_0,u_0)$, and to replace $T$ and $U$ by any sets $A$-definably homeomorphic to $T$ and $U$. In particular, since $t_0\in T$ and $u_0\in U$ are generic, we may assume $T$ and $U$ are open subsets of powers of $K$.

    Now by the genericity of $(x_0,y_0,t_0)\in X$ and $(x_0,y_0,u_0)\in Y$, after restricting to small enough neighborhoods we can assume $X$ and $Y$ are defined by polynomial equations $P(x,y,t)=0$ and $P(x,y,u)=0$. Since the generalized \'etale functions of $X_{t_0}$ and $Y_{u_0}$ have no Frobenius power, it follows that the partial derivatives of $P$ in $Q$ in the $x$ and $y$ variables are non-zero at $(x_0,y_0,t_0,u_0)$. It follows that we are in the situation of Lemma \ref{definable TIMI}: indeed, by assumption $(x_0,y_0,t_0,u_0)$ is strongly approximable over $A$ and $\mathcal X$ and $\mathcal Y$ have $n$-branching over $A$ in a neighborhood of $(x_0,y_0,t_0,u_0)$ -- thus $(x_0,y_0,t_0,u_0)$ belongs to the closure of points $(x,y,t,u)\in I$ (namely strongly generic $(\mathcal X,\mathcal Y)$-intersections) where the $n$th slopes of $X_t$ and $Y_u$ do not coincide.
    \end{proof}

\section{Proofs of the Main Theorems}

We now prove our main theorems for definable relics of ACVF. For this, we need to recall the main result of the recent preprint \cite{HaOnPi} (Fact \ref{F: HOP group case} below). We use some ad hoc terminology to recall their setting most easily.\\

Recall the notion of almost embeddability (Definition \ref{D: finite correspondence}). 

\begin{definition}\label{D: locally equivalent}
    Let $(K,v)$ be an $\aleph_1$-saturated algebraically closed valued field. Let $(G,\oplus_G)$ and $(H,\oplus_H)$ be interpretable groups in $(K,v)$ over some parameter set $A$. Assume that each of $G$ and $H$ is almost embeddable into $K$. We say that $G$ and $H$ are \textit{locally equivalent} if there are $g_{12},g_{23},g_{13}\in G$ and $h_{12},h_{23},h_{23}\in H$ such that:
    \begin{enumerate}
        \item $\dim(g_{12}g_{23}/A)=\dim(h_{12}h_{23}/A)=2$.
        \item $g_{12}\oplus_Gg_{23}=g_{13}$ and $h_{12}\oplus_Hh_{23}=h_{13}$.
        \item For each $1\leq i<j\leq 3$, $g_{ij}$ and $h_{ij}$ are interalgebraic over $A$.
    \end{enumerate}
\end{definition}

\begin{remark} The definition of local equivalence in \cite{HaOnPi} is more general. It does not assume almost embeddability into $K$, and uses the more general $dp$-\textit{rank} (discussed in the next section) instead of dimension. However, we will only encounter groups almost embeddable into $K$; and for these groups, the two notions coincide (see \cite[Theorem 0.3]{Simdp}, and use \cite{DoGoLi} to apply it).
\end{remark}

\begin{fact}\cite{HaOnPi}\label{F: HOP group case} Let $(K,v)$ be an $\aleph_1$-saturated algebraically closed valued field. Let $(G,\oplus_G)$ be a $(K,v)$-interpretable group which is almost embeddable into $K$ and locally equivalent to either $(K,+)$ or $(K^\times,\times)$. Let $\mathcal G=(G,\oplus_G,...)$ be a non-locally modular strongly minimal reduct of the full $(K,v)$-induced structure on $G$. Then $\mathcal G$ interprets a field $F$, $(K,v)$-definably isomorphic to $K$. Furthermore, the $\mathcal G$-induced structure on $F$ is a pure algebraically closed field.
\end{fact}

Let us proceed with our main results. First, Theorem \ref{T: main acvf} is a composite of our results on ACVF thus far, which is independent of Fact \ref{F: HOP group case}:

\begin{theorem}\label{T: main acvf} Let $\mathcal K=(K,v)$ be an $\aleph_1$-saturated algebraically closed valued field. Let $\mathcal M$ be a non-locally modular strongly minimal definable $\mathcal K$-relic. Then $\mathcal M$ interprets a strongly minimal group that is almost embeddable into $K$ and locally equivalent to either $(K,+)$ or $(K^{\times},\times)$.
\end{theorem}
\begin{proof} By Corollary \ref{C: acvf axioms}, $(K,v)$ is a Hausdorff geometric structure with the open mapping property and ramification purity. Moreover, by Theorems~\ref{T: acvf definable slopes} and~\ref{T: ACVF TIMI}, $(K,v)$ has definable slopes induced by $\mathcal{GEF}$ and satisfying TIMI. Now apply Theorem \ref{T: composite main thm}. The result is that $\dim(M)=1$ and $\mathcal M$ interprets a strongly minimal group satisfying all of the conclusions of Proposition \ref{P: group from codes}. To recall, this gives us the following (where $n_0$ is the smallest $n$ such that there is a non-algebraic coherent $n$-slope):

\begin{itemize}
    \item A parameter set $B$,
    \item An $\mathcal M(B)$-interpretable group $(G,\cdot)$ which is strongly minimal as an interpretable set in $\mathcal M$,
    \item A generic element $y\in K$ over $\emptyset$,
    \item For each $i,j\in\{1,2,3\}$ with $i<j$, an element $g_{ij}\in G$, and
    \item For each $i,j\in\{1,2,3\}$ with $i<j$, a morphism $f_{ij}:y\rightarrow y$ in $\mathcal{IA}_{n_0}$,
\end{itemize}
such that:
\begin{enumerate}
    \item $y$ is $\mathcal K(B)$-definable (that is, $y\in\operatorname{dcl}_{\mathcal K}(B)$).
    \item The $(n_0-1)$-th truncation of each $f_{ij}$ is the identity at $y$.
    \item $g_{12}$ and $g_{23}$ are independent generics in $G$ over $B$.
    \item $g_{13}=g_{23}\cdot g_{12}$ and $f_{13}=f_{23}\circ f_{12}$.
    \item For each $i,j\in\{1,2,3\}$ with $i<j$, $g_{ij}$ and $f_{ij}$ are interalgebraic over $B$.
\end{enumerate}

Our goal is to show that $(G,\cdot)$ is almost embeddable into $K$ and locally equivalent to either the additive or multiplicative group of $K$. The first clause is easy: By strong minimality, $G$ is in finite correspondence with $M$; but $M$ admits a finite-to-one map to $K$ (by using that $\dim(M)=1$ and taking piecewise projections) -- so composing gives a finite correspondence between $G$ and a subset of $K$. Thus $G$ is almost embeddable into $K$.

Now we move toward local equivalence with the additive or multiplicative group. Recalling that our definable slopes in ACVF are induced by the groupoid $\mathcal{GEF}$ from the previous section, we can express each $f_{ij}$ as an $n_0$-truncated morphism $y\rightarrow y$ in $\mathcal{GEF}$ with trivial $(n_0-1)$-th truncation. Recall that such truncations are notated $(y,y,a,c_1,...,c_{n_0})$, where $a$ is a Frobenius power and $c_1,...,c_{n_0}$ are the coefficients of an $n_0$-truncated polynomial sending 0 to 0. Since our $(n_0-1)$-th truncations are trivial, all but the last entry in such an expression are determined. Precisely, we have the following options for $f_{ij}$:

\begin{itemize}
    \item Case 1: $n_0=1$. Then we have $$f_{ij}=(y,y,\textrm{id},h_{ij})$$ for some $h_{ij}\in K$.
    \item Case 2: $n_0>1$. Then we have $$f_{ij}=(y,y,\textrm{id},1,0,0,...,0,h_{ij})$$ for some $h_{ij}\in K$.
\end{itemize}
The first case corresponds to scaling maps $x\mapsto h_{ij}x$, and the second corresponds to maps of the form $x+h_{ij}x^{n_0}$. In any case, we can fix $h_{ij}\in K$ as above for each $i<j$. Since $y\in\operatorname{dcl}_{\mathcal K}(B)$, each $h_{ij}$ is interdefinable with $f_{ij}$ over $B$, and so is also interalgebraic with $g_{ij}$ over $B$. One thus easily concludes (1) and (3) from Definition \ref{D: locally equivalent}, interpreting $H$ as either $(K,+)$ or $(K^\times,\times)$ (and replacing $A$ from Definition \ref{D: locally equivalent} with our $B$). So to complete the proof, we must show that either $h_{13}=h_{12}+h_{23}$ or $h_{13}=h_{12}h_{23}$. But $f_{13}=f_{12}\circ f_{23}$, and composition of truncations in $\mathcal{GEF}$ corresponds to composition of truncated polynomials -- so we just need to check the composition operation on the truncated polynomials in the two cases above:

\begin{itemize}
    \item Case 1: $n_0=1$. Then $h_{13}x$ is the composition of $h_{12}x$ and $h_{23}x$, which is $h_{12}h_{23}x$, and thus $h_{13}=h_{12}h_{23}$.
    \item Case 2: $n_0>1$. Then $x+h_{13}x^{n_0}$ is the $n_0$-truncated composition of $x+h_{12}x^{n_0}$ and $x+h_{23}x^{n_0}$. One easily checks this truncated composition to be $x+(h_{12}+h_{23})^{n_0}$, and thus $h_{13}=h_{12}+h_{23}$.
\end{itemize}
\end{proof}

Now combining our work with Fact \ref{F: HOP group case}, we obtain the full restricted trichotomy for definable relics of ACVF:

\begin{theorem}\label{T: combined acvf thm} Let $\mathcal K=(K,v)$ be an algebraically closed valued field. Let $\mathcal M$ be a non-locally modular strongly minimal definable $\mathcal K$-relic. Then $\mathcal M$ interprets a field $F$, $(K,v)$-definably isomorphic to $K$. Moreover, the structure induced on $F$ from $\mathcal M$ is a pure algebraically closed field.
\end{theorem}

\begin{proof}
   First, we check that we can assume $(K,v)$ is $\aleph_1$-saturated. This is a straightforward reduction, and similar arguments are given in \cite[Theorem 9.6]{CasACF0} and \cite[Lemma 4.8]{HaSu}. We give a sketch. Let $(K_1,v_1)$ be an $\aleph_1$-saturated elementary extension of $(K,v)$. Then the same formulas defining $\mathcal M$ in $(K,v)$ also define a definable $(K_1,v_1)$-relic, say $\mathcal M_1$. It is easy to see that $\mathcal M_1$ is an elementary extension of $\mathcal M$, so is strongly minimal and not locally modular. 
   
   Now suppose $\mathcal M_1$ interprets a field definably isomorphic to $K_1$, say $F$. Then there are formulas (with parameters from $K_1$) defining the set $F$, the field structure on $F$, the interpretation of $F$ in $\mathcal M_1$, and a field isomorphism $F\leftrightarrow K_1$. But all of these properties are first-order, so they can be pulled down to $(K,v)$ and $\mathcal M$. Thus, using different instances of the same formulas, one can find a definably isomorphic copy of $K$ in $\mathcal M$.
   
   So we assume $(K,v)$ is $\aleph_1$-saturated. By Theorem \ref{T: main acvf}, $\mathcal M$ interprets a strongly minimal group $G$ locally equivalent to $(K,+)$ or $(K^\times,\times)$. Now apply Fact \ref{F: HOP group case} to the structure induced on $G$ from $\mathcal M$.
\end{proof}

Finally, we deduce a full solution of the restricted trichotomy for pure algebraically closed fields:

\begin{corollary}[Restricted Trichotomy Conjecture]\label{C: RTC} Let $K$ be an algebraically closed field, and let $\mathcal M$ be a strongly minimal $K$-relic. If $\mathcal M$ is not locally modular, then $\mathcal M$ interprets the field $K$.
\end{corollary}
\begin{proof} By elimination of imaginaries, we may assume $\mathcal M$ is a definable $K$-relic (that is, its universe is a subset of some $K^n$). By an identical argument to Theorem \ref{T: combined acvf thm} above, we may assume $K$ is $\aleph_1$-saturated (really, $\aleph_0$ is enough here), and thus there is a non-trivial valuation on $K$. Let $v$ be such a valuation. Then $\mathcal M$ is still interpreted in $(K,v)$. One can now finish in two ways: either apply Theorem \ref{T: combined acvf thm} directly (thus using \cite{HaOnPi}); or apply Theorem \ref{T: composite main thm} to show $\dim(M)=1$, and then conclude with the main result of \cite{HaSu}.
\end{proof}

\section{Imaginaries in ACVF}\label{S: imaginaries}

In the final section, we discuss strongly minimal structures \textit{interpreted}, rather than \textit{defined}, in ACVF. That is, we allow the universe of the structure to be in an imaginary sort. Our main result is a reduction to the real sorts in residue characteristic zero, which works in residue characteristic $p$ assuming a conjectural condition about the distinguished sort $K/\mathcal O$ which we expect to be true.

We follow the strategy used in \cite{HaHaPeVF} for classifying interpretable fields in various valued fields by reducing them to four distinguished sorts (the valued field, the residue field, the value group, and the set of closed balls of radius zero). The guiding principle is that because of the `richness' and `uniformity' provided by families of plane curves in non-locally modular strongly minimal structures, such a structure should be (1) only able to `live' in one of these four sorts, and (2) not able to live in the two `non-field' sorts. We thus hope to reduce to structures genuinely coming from either the valued field or the residue field.

\subsection{The Setting}

Throughout this section, $\mathcal K=(K,v)$ is an $\aleph_1$-saturated model of ACVF. We let $\Gamma$ denote the value group, $\mathcal O$  the valuation ring, and $\textbf{k}$  the residue field. $\mathcal M=(M,...)$ is a non-locally modular strongly minimal $\mathcal K$-relic, whose universe $M$ need not be definable. Absorbing parameters into $(K,v)$ and $\mathcal M$, we assume that $\acl_{\mathcal M}(\emptyset)$ is infinite, and that every $\emptyset$-definable set in $\mathcal M$ is $\emptyset$-definable in $(K,v)$.

\begin{convention}
    Recall (see Definition \ref{D: finite correspondence}) that a \textit{finite correspondence} between two interpretable sets $X$ and $Y$ is an interpretable set $C\subset X\times Y$ such that the projections of $C$ to both factors are finite-to-one and surjective.
\end{convention}

\subsection{Ranks}
It was shown in \cite{DoGoLi} that ACVF is dp-minimal, and it thus follows that any ACVF relic has finite dp-rank (if the universe of the relic is $X/E$ for some definable set $X\sub K^n$ and definable equivalence relation $E$ then $\dprk(X^n/E)<n$). See \cite[\S 2.1]{HaHaPeVF}  for more relevant details on dp-rank. We remind the readers that dp-rank (as opposed to dimension) is sub-additive, but need not be additive:
\[
    \dprk(ab/A)\le \dprk(a/A)+\dprk(b/Aa).
\]

We use $\dprk$ for dp-rank in $(K,v)$, and $\mr$ for Morley rank in $\mathcal M$. Since $M$ is infinite, $\dprk(M)\geq 1$. 

We will repeatedly use the following:

\begin{lemma}\label{L: rank relationship} Let $X\subset M$ be infinite and interpretable over $A$. Let $a=(a_1,...,a_n)\in X^n$ with $\dprk(a/A)=n\cdot\dprk(X)$. Then $a$ is $\mathcal M$-generic in $M^n$ over $A$.
\end{lemma}
\begin{proof} If not, then some $a_m\in\acl_{\mathcal M}(Aa_1...a_{m-1})$. But then for that $m$ we have $$\dprk(a_m/Aa_1...a_{m-1})=0,$$ and so sub-additivity implies that $$\dprk(a/A)\leq\sum_{j=1}^n\dprk(a_j/Aa_1...a_{j-1})\leq 0+(n-1)\cdot\dprk(X)<n\cdot\dprk(X),$$ a contradiction.
\end{proof}

It follows from the main result of \cite{JohCminimalexist} that finiteness is definable in ACVF$^{eq}$. In such situations, dp-rank is also definable in definable families, but this is not needed in the sequel. 

\subsection{Local Interalgebraicity}

In \cite{HaHaPeVF}, Halevi, Hasson, and Peterzil classify the interpretable fields in various valued fields of interest. Their main tool is the notion of an interpretable set $Y$ being \textit{locally almost strongly internal} to an interpretable set $D$: this means that some infinite interpretable subset of $Y$ admits an interpretable finite-to-one map to $D$. In particular, the authors showed that every interpretable set is locally almost strongly internal to one of four distinguished sets (the valued field, the value group, the residue field, and the set of closed balls of radius zero). They then classify the fields locally almost strongly internal to each of these four sorts.

We will use a very similar strategy. However, for our purposes, it is more convenient to work with finite correspondences instead of finite-to-one maps (roughly, because without assuming a field, one cannot always generate enough definable functions). Thus, we will introduce a nearly identical, but slightly weaker, notion. Note that the reduction to the four distinguished sorts will remain valid.

\begin{definition} Given interpretable sets $Y$ and $D$, we define the $D$-\textit{critical number} of $Y$, $\operatorname{Crit}_{D}(Y)$, to be the largest dp rank of an interpretable subset $X\subset Y$ which is in interpretable finite correspondence with a subset of some $D^n$. We moreover call any such $X\subset Y$ $D$-\textit{critical} in $Y$. Finally, we say that $Y$ is \textit{locally interalgebraic with} $D$ if $\operatorname{Crit}_D(Y)\geq 1$.
\end{definition}

By compactness, we have:

\begin{fact}\label{critical in tuples} Let $Y$ and $D$ be interpretable, and let $X$ be $D$-critical in $Y$. Let $A$ be a parameter set, let $y\in Y$, and let $d_1,...,d_n\in D$. If $y$ is interalgebraic with $d_1,...,d_n$ over $A$, then $\dprk(y/A)\leq\dprk(X)$.
\end{fact}

The following is also clear:
\begin{fact}\label{F: critical number finite correspondence} Suppose $Y$, $Z$, and $D$ are interpretable sets, and $Y$ is in finite correspondence with a subset of $Z$. Then $\operatorname{Crit}_D(Y)\leq\operatorname{Crit}_D(Z)$.
\end{fact}

The following  key observation is \cite[Proposition 5.5]{HaHaPeVF}, and forms the basis for our approach: 

\begin{fact} Every infinite interpretable set is locally interalgebraic with at least one of $K$, $\textbf{k}$, $\Gamma$, and $K/\mathcal O$.
\end{fact}

We will need the following stronger result (see \cite[Lemma 2.8]{HaOnPi}):

\begin{fact}\label{F: K pure} Let $Y$ be an infinite interpretable set which is not locally interalgebraic with any of $\textbf{k}$, $\Gamma$, or $K/\mathcal O$. Then $Y$ definably embeds into $K^n$ for some $n$.
\end{fact}

\subsection{Richness}

Our main goal is to `rule out' the two non-field sorts $\Gamma$ and $K/\mathcal O$, by showing in some sense, that they are not complex enough to admit a non-locally modular strongly minimal structure. Our tool for showing this is \textit{richness}. Roughly, we want to say that a sort $D$ is `rich' if it admits arbitrarily large families of infinite subsets of a fixed $D^n$, parametrized by tuples from $D$, satisfying an appropriate `non-redundancy' condition.

\begin{definition}
    Let $A$ be a parameter set, let $k$ be a non-negative integer, and $a$ and $b$ finite tuples. We call $(a,b,A)$ a $k$-\textit{rich configuration} if the following hold:
    \begin{enumerate}
        \item $a\notin\acl(Ab)$.
        \item Whenever $\overline b=\{b_1,b_2,...\}$ is an infinite sequence of distinct realizations of $\tp_{(K,v)}(b/Aa)$, then $a\in\acl(A\overline b)$.
        \item $\dprk(b/A)\geq k$.
    \end{enumerate}
\end{definition}

\begin{definition}\label{D: rich sort}
    An interpretable set $D$ is \textit{rich} if there is $n$ such that for arbitrarily large $k$, there are $m$, elements $a\in D^n$ and $b\in D^m$, and a parameter set $A$, such that $(a,b,A)$ is a $k$-rich configuration. 
\end{definition}

\begin{example} The motivating example of the above notions is given by a family of plane curves in our strongly minimal relic $\mathcal M$. Namely, let $\{C_t:t\in T\}$ be an almost faithful $\mathcal M(A)$-definable family of plane curves in $\mathcal M$, where $T\subset M^k$ has Morley rank $k$. Let $t\in T$ with $\dprk(t/A)=\dprk(T)$, and let $x\in C_t$ with $\dprk(x/At)=\dprk(C_t)$. Then one shows easily that $(x,t,A)$ is a $k$-rich configuration. In particular, since $\mathcal M$ is not locally modular (and thus $k$ can be arbitrarily large), it follows that the universe $M$ is rich (where in the notation of Definition \ref{D: rich sort}, we use $n=2$).
\end{example}

The notion of a $k$-rich configuration may seem a bit obscure. Its main benefit to us is that it is preserved under finite correspondences -- a useful feature given the definition of local interalgebraicity:

\begin{lemma}\label{L: rich preservation} Assume that $(a,b,A)$ is a $k$-rich configuration, and let $c$ and $d$ be finite tuples with $\acl(Aa)=\acl(Ac)$ and $\acl(Ab)=\acl(Ad)$. Then $(c,d,A)$ is a $k$-rich configuration.
\end{lemma}
\begin{proof} Interalgebraicity gives automatically that $\dprk(d/A)\geq k$ and $c\notin\acl(Ad)$. Now suppose $\overline d=(d_1,d_2,...)$ are infinitely many realizations of $\tp_{(K,v)}(d/Ac)$. For each $i$, let $(a_i,b_i)$ be such that $\tp(a_ib_icd_i/A)=\tp(abcd/A)$. Since $a\in\acl(Ac)$, we may assume after passing to a subsequence that each $a_i=a_1$. Since each $d_i\in\acl(Ab_i)$, and the $d_i$ are distinct, we may assume after passing to a further subsequence that the $b_i$ are all distinct (here we are using that all of the tuples involved are finite). Now since $\tp(a_1b_1/A)=\tp(ab/A)$, $(a_1,b_1,A)$ is a $k$-rich configuration. It follows that $a_1\in\acl(Ab_2b_3...)$. But each $b_i$ is interalgebraic with $d_i$ over $A$, and $a_1$ is interalgebraic with $c$ over $A$. Thus $c\in\acl(Ad_2,d_3,...)$.
\end{proof}

\begin{corollary}\label{richness preserved} Let $X$ and $Y$ be interpretable sets. Assume that $X$ is in definable finite correspondence with a subset of some $Y^n$. If $X$ is rich, then so is $Y$.
\end{corollary}

\subsection{Main Proposition}

We now give the main technical result of this section. This result should be thought of as saying that if $M$ has any interaction at all with some sort $D$, then it has a large amount of interaction with $D$.

\begin{proposition}\label{P: main prop} Suppose $M$ is locally interalgebraic with an interpretable set $D$. Assume that $X$ is $D$-critical, and that $X$, as well as some finite correspondence witnessing it, are $A$-definable. Let $\mathcal C=\{C_t:t\in T\}$ be an $\mathcal M(A$)-definable almost faithful family of plane curves 
with $T\subset M^2$ of Morley rank 2. Let $(u_0,v_0)\in M^2$ be $\mathcal M$-generic over $A$, and let $(x_0,y_0)\in X^2$ be such that $\dprk(x_0y_0/Au_0v_0)=2\dprk(X)$. Then: 
\begin{enumerate}
    \item There is at least one, but only finitely many, $t\in T$ with $(u_0,v_0),(x_0,y_0)\in C_t$. 
    \item If $(u_0,v_0),(x_0,y_0)\in C_t$ for some $t$ then $C_t\cap X^2$ is infinite.
\end{enumerate}
\end{proposition}
\begin{proof} It follows from Lemma \ref{L: rank relationship} that $(x_0,y_0)$ is $\mathcal M$-generic in $M^2$ over $Au_0v_0$, so that $(u_0,v_0)$ and $(x_0,v_0)$ are $\mathcal M$-independent $\mathcal M$-generics in $M^2$ over $A$. Since $M^2$ is stationary, this easily implies (1).

Now we show (2). Absorbing parameters, we assume $A=\emptyset$. Let $t$ be such that $(u_0,v_0),(x_0,y_0)\in C_t$. By (1), $t\in\acl(u_0v_0x_0y_0)$. Now assume toward a contradiction that $C_t\cap X^2$ is finite. Then $(x_0,y_0)\in\acl(t)$, so $(x_0,y_0)$ is interalgebraic with $t$ over $(u_0,v_0)$. Now $t$ belongs to the set of $t'$ with $(u_0,v_0)\in C_{t'}$, so it follows that $\mr(t/u_0v_0)\leq 1$. In particular, by strong minimality, $t$ is interalgebraic with an element of $M$ over $(u_0,v_0)$. But since $x_0,y_0\in X$, $(x_0,y_0)$ is interalgebraic with a tuple $\bar d$ of elements of $D$. Putting everything together, the tuples $t$, $(x_0,y_0)$, and $\bar d$ are interalgebraic over $(u_0,v_0)$. Since $X$ is $D$-critical, Lemma \ref{critical in tuples} now implies that $\dprk(x_0y_0/u_0v_0)\leq\dprk(X)$. But by assumption $\dprk(x_0y_0/u_0v_0)=2\cdot\dprk(X)$, forcing $\dprk(X)=0$, and thus contradicting that $M$ is locally interalgebraic with $D$. 
\end{proof}

\subsection{Corollaries of the Main Proposition}

We now deduce two useful corollaries from Proposition \ref{P: main prop}. 

\begin{corollary}\label{C: internal implies rich} Suppose $M$ is locally interalgebraic with an interpretable set $D$. Then $D$ is rich.
\end{corollary}
\begin{proof} Let $X\subset M$ be $D$-critical. We show that $X$ is rich, which implies by Corollary \ref{richness preserved} that $D$ is rich.

Fix $k$. We will find a $k$-rich configuration $(a,b,A)$ where $a\in X^2$ and $b$ is a tuple from $X$. First let $j=k+2$. Then let $\mathcal C=\{C_t:t\in T\}$ be an almost faithful family of plane curves in $\mathcal M$, where $T\subset M^{j+2}$ is generic (see \cite{CasACF0}, Fact 2.27 and Lemma 2.42). Adding parameters, we assume $\mathcal C$ is $\mathcal M$-definable over $\emptyset$. Given $z\in M^{j}$, we let $\mathcal C(z)$ be the family with graph $\{(x,y,s)\in M\times M\times M^2:(x,y)\in C_{(s,z)}\}$. If $z$ is $\mathcal M$-generic in $M^j$, then $\mathcal C(z)$ is an almost faithful family indexed by a generic subset of $M^2$.

Let us fix $z\in X^j$ with $\dprk(z)=j\cdot\dprk(X)$. It follows from Lemma \ref{L: rank relationship} that $z$ is $\mathcal M$-generic in $M^j$.

Now let $(u,v)\in X^2$ with $\dprk(uv/z)=2\dprk(X)$, and let $(x,y)\in X^2$ with $\dprk(xy/uvz)=2\dprk(X)$. It follows from Lemma \ref{L: rank relationship} that $(u,v)$ is $\mathcal M$-generic in $M^2$ over $z$.

We are in the situation of Proposition \ref{P: main prop}. We conclude that there is $s\in M^2$ such that for $t=(s,z)$, we have $(u,v),(x,y)\in C_t$ and $C_t\cap X^2$ is infinite. Now since $C_t\cap X^2$ is infinite, there is $a=(a_1,a_2)\in C_t\cap X^2$ with $\dprk(a/uvxyt)\geq 1$. We conclude:

\begin{claim} $(a,t,ux)$ is a $k$-rich configuration.
\end{claim}
\begin{claimproof} Since $\dprk(a/uvxyt)\geq 1$, $a\notin\acl(uxt)$. Now $\dprk(t)\geq\dprk(z)=j\cdot\dprk(X)$. Since $(u,x)\in X^2$, $\dprk(ux)\leq 2\dprk(X)$. So by subadditivity (remember that $j=k+2$), $\dprk(t/ux)\geq k\cdot\dprk(X)\geq k$.

Finally, let $\bar t=t_1,t_2,...$ be distinct realizations of $\tp_{(K,v)}(t/uxa)$. Since $\mathcal C$ is almost faithful, there is $N$ such that $C_{t_1}\cap C_{t_N}$ is finite, which shows that $a\in\acl(ux\bar t)$.
\end{claimproof}

To prove the corollary, we need to modify $(a,t,ux)$ by replacing $t$ by a tuple from $X$. The key to doing this is:

\begin{claim} $t$ is interalgebraic with $zvy$ over $ux$. In particular, $(a,zvy,ux)$ is a $k$-rich configuration.
\end{claim}
\begin{claimproof}
    Clearly $z\in\acl(t)$ because $z$ is a subtuple of $t$. Now it follows from the non-triviality of $C_t$ that $v\in\acl(tu)$ and $y\in\acl(tx)$. On the other hand, by Proposition \ref{P: main prop}, there are only finitely many $C_{t'}$ containing both $(u,v)$ and $(x,y)$, which shows that $t\in\acl(uvxy)$. The claim now follows by Lemma \ref{L: rich preservation}.
\end{claimproof}

Finally, note that $a\in X^2$ and $zvy\in X^{j+2}$. So since $k$ was arbitrary, we conclude that $X$ is rich, and thus so is $D$.
\end{proof}

\begin{remark} The above proof of Corollary \ref{C: internal implies rich} would have been easier if, in the definition of richness of $X$, the `parameter' tuple $b$ was not required to be a sequence of elements of $X$. Indeed, in this case we could have taken $(a,t,\emptyset)$ as our $k$-rich configuration. The problem is that our ultimate goal is to show that $\Gamma$ and $K/\mathcal O$ are not rich; and if we adopted this weaker version of richness, then $K/\mathcal O$ \textit{would} be rich in positive characteristic (indeed, taking $\mathcal O$-linear combinations of Frobenius powers gives arbitrarily large families of subsets of $(K/\mathcal O)^2$ indexed by powers of $\mathcal O$). Recall, however, that $K/\mathcal O$ is not stably embedded -- and we expect that it is impossible to find such large families in $(K/\mathcal O)^n$ without using external parameters.
\end{remark}

\begin{corollary}\label{C: almost internality} Suppose $M$ is locally interalgebraic with an interpretable set $D$. Then $M$ is almost internal to $D$.
\end{corollary}
\begin{proof} Let $X\subset M$ be $D$-critical. Fix a rank 2 family $\mathcal C=\{C_t:t\in T\}$ of non-trivial plane curves as in Proposition \ref{P: main prop}, and assume it is $\emptyset$-definable in $\mathcal M$. Now let $v$ be any element of $M-\acl(0)$. Let $u\in X$ be such that $\dprk(u/v)=\dprk(X)$, and let $(x,y)\in X^2$ be such that $\dprk(xy/uv)=2\dprk(X)$. So $(u,v)$ is $\mathcal M$-generic in $M^2$. Then by Proposition \ref{P: main prop}, there is $t\in T$ with $(u,v),(x,y)\in C_t$ and $C_t\cap X^2$ infinite. Let $(x',y')\in C_t\cap X^2$ with $\dprk(x'y'/txy)\geq 1$. So $(x',y')$ is $\mathcal M$-generic in $C_t$ over $txy$, and thus there are only finitely many curves from $\mathcal C$ containing both $(x',y')$ and $(x,y)$ (see \cite{CasACF0}, Lemma 2.38). Since $C_t$ is one such curve, $t\in\acl(xyx'x')$. But since $C_t$ is non-trivial, we also have $v\in\acl(tu)$. Thus in total we have $v\in\acl(uxyx'y')$. Now $u,x,y,x',y'\in X$. By compactness, all but finitely many elements of $M$ are in $\acl(X)$. Adding finitely many parameters to cover the remaining elements, we get that $M$ is almost internal to $X$. But by definition $X$ is almost internal to $D$, so it follows that $M$ is almost internal to $D$.
\end{proof}

\subsection{Ruling out the Non-field Sorts}

Our next goal is to show that $\Gamma$ and $K/\mathcal O$ are not rich. For $K/\mathcal O$, our approach will only work in residue characteristic 0.

Let $(D,+,...)$ be either $\Gamma$ or $K/\mathcal O$. We can equip $D$ with a non-discrete group topology: for $\Gamma$, we use the order topology, and for $K/\mathcal O$, we use the topology generated by balls of negative radius (equivalently, the topology induced by valuation on $K/\mathcal O$, identifying $v(\mathcal O)$ with 0). It was noted in \cite{HaHaPeVF} that, with this topology, $D$ forms a \textit{uniform structure} in the sense of Simon and Walsberg \cite{SimWal} (the o-minimal case appears already in the work of Simon and Walsberg; for $K/\mathcal O$ see \cite[Lemma 5.13]{HaHaPeVF}) We do not define this notion here. Instead, we give the following facts that we will use:

\begin{fact}\label{F: local homeo} Let $X\subset D^n$ be $A$-definable with $\dprk(X)=k$. Let $a\in X$ with $\dprk(a/A)=k$. Then there are a projection $\pi:X\rightarrow D^k$, a set $B\supset A$, and a $B$-neighborhood $U$ of $a$ in $X$, such that $\dprk(a/B)=k$ and $\pi$ restricts on $U$ to a homeomorphism with an open subset of $D^k$. 
\end{fact}

\begin{fact}\label{F: no isolated} Let $X\subset D^n$ be $A$-definable, and let $a\in X$ be an isolated point. Then $a\in\acl(A)$.
\end{fact}

For the first of the above facts see \cite[Proposition  4.6]{SimWal} (and \cite[Corollary 3.13]{HaHaPeVF} to ensure $\dprk(a/B)=k$), and for the second see \cite[Lemma 3.6]{SimWal}.  

The following notion was studied in \cite{HaHaPeVF}: 

\begin{definition}
    $D$ is \textit{locally linear} if whenever $f:U\rightarrow D$ is an $A$-definable function on an open subset of $D^n$, and $a\in U$ with $\dprk(a/A)=n$, there is a neighborhood of $a$ on which $f$ agrees with a map of the form $x\mapsto g(x)+b$, where $g:D^n\rightarrow D$ is a definable group homomorphism and $b\in D$ is a constant.
\end{definition}

One can quickly improve on this definition:

\begin{lemma}\label{L: coset point} Assume $D$ is locally linear. Let $X\subset D^n$ be $A$-definable, and let $a\in X$ with $\dprk(a/A)=\dprk(X)$. Then there are $B\supset A$ and a $B$-definable neighborhood $U$ of $a$ in $X$, such that $\dprk(a/B)=\dprk(X)$ and $X$ agrees with a coset of a definable subgroup of $D^n$.
\end{lemma}
\begin{proof}
    Let $k=\dprk(X)$. By Fact \ref{F: local homeo}, we can assume there is a projection $\pi:X\rightarrow D^k$ which is a homeomorphism onto an open set. We thus view $X$ locally as the graph of a function $D^k\rightarrow D^{n-k}$. Now apply the definition of local linearity to each coordinate component of this function.
\end{proof}

The main fact we need is:

\begin{fact} $\Gamma$ is locally linear. $K/\mathcal O$ is locally linear if $K$ has residue characteristic zero. 
\end{fact}

For $\Gamma$ this is well known, since $\Gamma$ is an o-minimal pure divisible ordered abelian group. For $K/\mathcal O$  this is \cite[Proposition 6.24]{HaHaPeVF}.\\

We do not know if $K/\mathcal O$ is locally linear in residue characteristic $p$, but we suspect that it is. 

Continuing, we now show:

\begin{proposition}\label{P: gamma not rich} Assume $D$ is locally linear. Then $D$ is not rich.
\end{proposition}
\begin{proof} Let $a\in D^n$, $b\in D^m$, $A$, and $k$ be such that $(a,b,A)$ is a $k$-rich configuration (for now, we do not assume anything about $k$). Let $r=\dprk(ab/A)$, and let $X\subset D^n\times D^m$ be $A$-definable with $(a,b)\in X$ and $\dprk(ab/A)=r$. Since $(a,b,A)$ is a $k$-rich configuration, and using compactness, we can assume that the intersection of any infinitely many distinct fibers $X_{b'}\subset D^n$ is finite.

By Lemma \ref{L: coset point}, there are $B\supset A$, a $B$-definable neighborhood $U$ of $a$ in $X$, and a subgroup $H\leq D^{n+m}$, such that $\dprk(ab/B)=r$ and $X$ agrees with the coset $H+(a,b)$ on $U$. Note, then, that whenever $(a',b')\in U\cap X$, we have $H+(a,b)=H+(a',b')$, and thus $X$ also agrees with $H+(a',b')$ on $U$.

Now let $G\times\{0\}$ be the kernel of the projection $H\rightarrow D^m$ -- so $G$ is a subgroup of $D^n$. By the above, whenever $(a',b')\in U\cap X$, the fiber $X_{b'}$ agrees with $G+a'$ in a neighborhood of $a'$. Indeed, in some small enough neighborhood of $a'$, we have $$x\in X_{b'}\iff(x,b')\in X\iff(x,b')\in H+(a',b')\iff (x,b')-(a',b')\in H$$ $$\iff(x-a',0)\in H\iff x-a'\in G\iff x\in G+a'.$$ Now in particular, this implies that for any $b,b'$ with $(a,b),(a,b')\in X\cap U$, the fibers $X_b$ and $X_{b'}$ agree in a neighborhood of $a$ (because they both agree with $G+a$).

Now assume toward a contradiction that $D$ is rich. This says that, in the situation constructed above, we can make $\dprk(ab/A)=\dprk(ab/B)$ arbitrarily large, while leaving $\dprk(a/A)$ (and thus also $\dprk(a/B)$) bounded. In particular, we may fix all of the data in such a way that $b\notin\acl(Ba)$. Thus, we can find $\bar b=b_1,b_2,...$, an infinite sequence of distinct realizations of $\tp_{(K,v)}(b/Ba)$. So each $(a,b_i)\in U$, and thus (by the previous paragraph) any finitely many of the fibers $X_{b_i}$ coincide on a neighborhood of $a$. 

Finally, since $(a,b,A)$ is a $k$-rich configuration, we have $a\notin\acl(Ab)$. So by Fact \ref{F: no isolated}, $a$ is not isolated in $X_b$, and thus $a$ is not isolated in any $X_{b_i}$. Thus, any finitely many $X_{b_i}$ have infinite intersection. By compactness, the intersection of all of the $X_{b_i}$ is infinite, a contradiction.
\end{proof}

\begin{corollary}\label{C: rule out groups}
$M$ is not locally interalgebraic with $\Gamma$. If $\textbf{k}$ has characteristic zero, $M$ is not locally interalgebraic with $K/\mathcal O$. If $\textbf{k}$ has characteristic $p>0$ and $K/\mathcal O$ is locally linear, $M$ is not locally interalgebraic with $K/\mathcal O$.
\end{corollary}

\subsection{Concluding}

We now show our main result:

\begin{theorem}\label{T: imaginaries main thm} Assume that $K$ has residue characteristic 0, or that $\textbf{k}$ has characteristic $p>0$ and $K/\mathcal O$ is locally linear. Then $M$ definably embeds into a power of $K$ or $\textbf{k}$.
\end{theorem}

\begin{proof} By Corollary \ref{C: rule out groups}, $M$ is not locally interalgebraic with $\Gamma$ or $K/\mathcal O$. 

First, suppose $M$ is locally interalgebraic with $k$. Then by Corollary \ref{C: almost internality}, $M$ is almost internal to $\textbf k$. It follows by an easy type-counting argument that $M$ is stable and stably embedded. But all such sets are internal to $\textbf k$ (\cite{HaHrMac1}, Proposition 3.4.11). Combining with stable embeddedness and elimination of imaginaries in $\textbf k$, we then embed $M$ into some $k^n$.

Now assume $M$ is not locally interalgebraic with $\textbf k$. Then by Fact \ref{F: K pure}, $M$ definably embeds into a power of $K$.
\end{proof}

We now drop the ambient assumptions and end the paper by collecting all results on ACVF:

\begin{theorem}\label{T: biggest acvf thm}
    Let $\mathcal K=(K,v)$ be an algebraically closed valued field, and let $\mathcal M$ be a non-locally modular strongly minimal $\mathcal K$-relic. Assume the residue field $\textbf k$ has characteristic zero, or $K/\mathcal O$ is locally linear. Then $\mathcal M$ interprets a field $(F,\oplus,\otimes)$, $(K,v)$-definably isomorphic to either $K$ or $\textbf{k}$. Moreover, the induced structure on $F$ from $\mathcal M$ is a pure algebraically closed field.
\end{theorem}
\begin{proof} By Theorem \ref{T: imaginaries main thm}, we may assume that $M\subset K^n$ or $M\subset \textbf k^n$ for some $n$. If $M\subset \textbf k^n$, we reduce to Corollary \ref{C: RTC}. If $M\subset K^n$, we reduce to Theorem \ref{T: combined acvf thm}.
\end{proof}

\noindent\textit{Acknowledgements:} This project started while the first and third authors attended the 2021 Thematic Program on Trends in Pure and Applied Model Theory at the Fields Institute. We thank the Fields Institute for their hospitality. We would also like to thank Yatir Halevi for discussions around the model theory of valued fields,  Will Johnson for his help with visceral theories, Amnon Yekutily and Ilya Tyomkin for their help with algebro-geometric aspects of the paper, and Martin Hils for pointing out a mistake in an earlier version of the paper.

\bibliographystyle{amsalpha}
\bibliography{ref}
\end{document}